\documentclass[11pt,a4paper]{amsart}

\usepackage[utf8]{inputenc}
\usepackage{amssymb,amsmath,amsbsy,amscd,amsfonts,amsthm,latexsym}
\usepackage[mathcal]{eucal}
\usepackage{times, euler}
\usepackage{eurosym}

\usepackage{stackrel, enumerate}
\usepackage{mathabx}
\usepackage[usenames]{color}

\usepackage[top=3cm,bottom=3cm,left=2cm,right=2cm,footskip=17mm]{geometry}

\usepackage{hyperref}

\hypersetup{
    colorlinks=true,
    linkcolor=blue,
    filecolor=cyan,      
    urlcolor=blue,
    citecolor=magenta
}

\usepackage{xy}
\xyoption{all}

\numberwithin{equation}{section}

\theoremstyle{plain}
    \newtheorem{theorem}[equation]{Theorem}
    \newtheorem{lemma}[equation]{Lemma}
    \newtheorem{corollary}[equation]{Corollary}
    \newtheorem{proposition}[equation]{Proposition}

\theoremstyle{definition}
    \newtheorem{definition}[equation]{Definition}
    \newtheorem{example}[equation]{Example}
    \newtheorem{observation}[equation]{Observation}
    
    \newtheorem{remark}[equation]{Remark}
    \newtheorem{remarks}[equation]{Remarks}

    \newcommand{\C}{\mathbb{C}}
    \newcommand{\N}{\mathbb{N}}
    \newcommand{\Z}{\mathbb{Z}}
    \newcommand{\Q}{\mathbb{Q}}

    \renewcommand{\O}{\mathfrak o}

    \renewcommand{\H}{\mathcal H}
    \renewcommand{\k}{\Bbbk}
    
   	\renewcommand{\phi}{\varphi}
	\renewcommand{\epsilon}{\varepsilon}
  
\newcommand{\into}{\hookrightarrow}

\newcommand{\dd}{d}  
\newcommand{\restrict}{\big{\vert}}

\newcommand{\germ}{\mathfrak}
\newcommand{\id}{\mathrm{id}}
\newcommand{\ol}{\overline}

\newcommand{\transpose}{t}
\newcommand{\ot}{\otimes}

    \DeclareMathOperator{\Irr}{Irr}
   \DeclareMathOperator{\Hom}{Hom}
    \DeclareMathOperator{\End}{End}

    \DeclareMathOperator{\Ad}{Ad}
    \DeclareMathOperator{\Aut}{Aut}

    \DeclareMathOperator{\ind}{ind}
    \DeclareMathOperator{\SL}{SL}
    \DeclareMathOperator{\GL}{GL}
    \DeclareMathOperator{\TT}{T}
    \DeclareMathOperator{\BB}{B}
    \DeclareMathOperator{\UU}{U}
	
\DeclareMathOperator{\Prim}{Prim}
\DeclareMathOperator{\Image}{Im}

    \DeclareMathOperator{\Sp}{Sp}
    \DeclareMathOperator{\Rep}{\mathcal{R}}
	\DeclareMathOperator{\vol}{vol}
	\DeclareMathOperator{\res}{res}
	\DeclareMathOperator{\infl}{inf}
	
	\DeclareMathOperator{\lspan}{span}
	
	\DeclareMathOperator{\Mod}{Mod}
	\DeclareMathOperator{\ch}{ch}
	
	\DeclareMathOperator{\pind}{i}
	\DeclareMathOperator{\pres}{r}

	\DeclareMathOperator{\trace}{tr}
	\DeclareMathOperator{\diag}{diag}
	
	\DeclareMathOperator{\Lie}{Lie}

\keywords{{Harish-Chandra induction, parabolic induction, compact $p$-adic groups, representations of profinite groups}}

\subjclass[2010]{Primary 22E50; Secondary 20G25, 20C33, 20C15, 20C07.}

\begin{document}

\begin{abstract} 
Harish-Chandra induction and restriction functors play a key role in the representation theory of reductive groups over finite fields.  
In this paper, extending earlier work of Dat, we introduce and study generalisations of these functors  which apply to a wide range of finite and profinite groups, typical examples being {compact open subgroups of reductive groups over non-archimedean local fields}. We prove that these generalisations are compatible with two of the  tools commonly used to study the (smooth, complex) representations of such groups, namely Clifford theory and the orbit method. As a test case, we examine in detail the induction and restriction of representations from and to the Siegel Levi subgroup of the symplectic group  $\Sp_4$ over a finite local principal ideal ring of length two. We obtain in this case a Mackey-type formula for the composition of these induction and restriction functors which is a perfect analogue of the well-known formula for the composition of Harish-Chandra functors. In a different direction, we study representations of the Iwahori subgroup $I_n$ of $\GL_n(F)$, where $F$ is a non-archimedean local field. We establish a bijection between the set of irreducible representations of $I_n$ and tuples of primitive irreducible representations of smaller Iwahori subgroups, where primitivity is defined by the vanishing of suitable restriction functors.
\end{abstract}

\title{A variant of Harish-Chandra functors}

\author{Tyrone Crisp}
\address{Max Planck Institute for Mathematics, 
Vivatsgasse 7, 53111 Bonn, Germany}
  \email{tyronecrisp@mpim-bonn.mpg.de}

\author{Ehud Meir}
\address{Department of Mathematics,
University of Hamburg,
Bundesstr. 55,
20146 Hamburg, Germany}
  \email{meirehud@gmail.com}

\author{Uri Onn}
\address{Department of Mathematics, Ben Gurion
  University of the Negev, Beer-Sheva 84105, Israel}
  \email{urionn@math.bgu.ac.il}

\date{\today}
\maketitle

\setcounter{tocdepth}{1}
\tableofcontents
\thispagestyle{empty}

\section{Introduction}\label{sec:intro}

\subsection{Overview}

Harish-Chandra (or parabolic) induction and restriction are fundamental operations in the representation theory of reductive groups over finite fields, 
allowing efficient transport of representations between such groups and establishing a close connection to the representation theory of
finite Coxeter groups; see \cite{Zelevinsky, vanLeeuwen} for a particularly elegant development of this connection for finite classical groups. 
Recall that Harish-Chandra induction is an instance of the following general construction. Given a finite group $G$, and subgroups $L$ and $U$ such that $L$ normalises $U$, one obtains a functor $\pind_L^G$ from the complex representations of $L$ to the complex representations of $G$ by tensor product with the $\H(G)$-$\H(L)$ bimodule $\H(G)e_U$, where~$\H(G)$ is the complex group algebra of $G$ and $e_U$ is the idempotent associated with the trivial representation of $U$. Dually, tensoring with the bimodule $e_U\H(G)$ gives a functor $\pres^G_L$ that is adjoint to $\pind_L^G$. A variant of  Mackey's double-coset formula applies to the composite functor $\pres^G_L\pind_L^G$, yielding a decomposition of the endomorphism algebra of an induced representation $\pind_L^G(M)$ into a direct sum  indexed by the double-coset space $LU\backslash G/LU$;  cf. \cite[Theorem 2.3.1]{vanLeeuwen}. In the case of Harish-Chandra induction, $G$ is the group of rational points of a connected reductive group defined over a finite field, and $L$ and $U$ are the respective groups of rational points of a Levi factor of and the unipotent radical of a rational parabolic subgroup $P$ of $G$. The Bruhat decomposition gives a parametrisation of the $P$-double cosets in $G$ by the double cosets of the Weyl group of $L$ in the Weyl group $G$, and the Mackey formula becomes  
\begin{equation}\label{eq:Mackey_intro}
\pres^G_{L} \pind_{L}^G \cong \bigoplus_{g\in W_{L}\backslash W_G / W_{L}} \pind_{L\cap gLg^{-1}}^{L} \, \Ad_g \, \pres^{L}_{g^{-1}Lg\cap L}.
\end{equation}
See \cite{DM} for the precise general formulation and proof, and for a sampling of the applications of this formula; and {see} \cite{Harish-Chandra, Springer_cusp} for the original work of Harish-Chandra.

In this paper we study induction and restriction functors which generalise the Harish-Chandra functors to a rich family of profinite groups, to which the family of reductive groups over finite fields is only a partial first approximation. Our motivating examples are classical groups over compact discrete valuation  rings, but our framework covers many other cases, including arbitrary open compact subgroups of reductive groups
over local fields. Certain representations of such open compact  subgroups play an important role in the construction and classification of smooth representations of the reductive groups via the theory of types. However, the representation theory of these compact subgroups per se is not so well understood.        

Before we introduce the functors that are at the heart of the present paper we remark that the most obvious generalisation  of the Harish-Chandra functors to the setting considered here tends to produce representations that are far from irreducible, and in this sense lacks the efficiency of the \lq classical\rq~ Harish-Chandra functors. For a concrete example, let $\O$ be the ring of integers in a non-archimedean local field $F$ (so $F$ is either the field of Laurent series over a finite field, or a finite extension of the $p$-adic numbers). Let $\germ p$ denote the maximal ideal of $\O$, and for every $\ell \in \N$ set~$\O_\ell=\O/\germ p^\ell$. Let $\TT_n\subset \BB_n \subset \GL_n$ denote the standard diagonal torus and the standard upper-triangular Borel subgroup in the general linear group, and let~$\UU_n$ denote the unipotent radical of~$\BB_n$. The subgroup $\TT_n(\O_\ell)$ normalises $\UU_n(\O_\ell)$, and so the construction described in the first paragraph gives a functor from representations of $\TT_n(\O_\ell)$ to representations of $\GL_n(\O_\ell)$, which in particular  sends the trivial representation of  $\TT_n(\O_\ell)$ to the permutation representation  of $\GL_n(\O_\ell)$ given~by 
\begin{equation}\label{eqn.HC.for.GLn}
\H(\GL_n(\O_\ell))e_{\UU_n(\O_\ell)} \otimes_{\H(\TT_n(\O_\ell))} 1 \cong \H(\GL_n(\O_\ell)/\BB_n(\O_\ell)).
\end{equation}

When $\ell=1$  the Mackey formula \eqref{eq:Mackey_intro} gives a decomposition of \eqref{eqn.HC.for.GLn} according to the regular representation of the symmetric group on $n$ letters. For $\ell > 1$  the decomposition of \eqref{eqn.HC.for.GLn} into irreducibles gets very quickly out of control, owing to the complicated nature of the double-coset space $\BB_n(\O_\ell)\backslash \GL_n(\O_\ell)/ \BB_n(\O_\ell)$. Misleadingly simple is the case $n=2$, where the induced representation has $\ell+1$ irreducible components (see \cite{Casselman}); already for $n=3$ the decomposition of the induced representation is rather complicated and, in particular, depends on the degree of the residue field~$\O_1=\O/\germ p$, see~\cite{OnnSingla}. 

Our proposed variant of Harish-Chandra induction, in this $\GL_n$ example, sends the trivial representation of $\TT_n(\O_\ell)$ to the image of the intertwining operator
\[
\H\left(\GL_n(\O_\ell)/\BB_n(\O_\ell)\right) \to \H\left(\GL_n(\O_\ell)/\BB_n^{\transpose}(\O_\ell)\right)
\]
which averages right  $\BB_n(\O_\ell)$-invariant functions on $\GL_n(\O_\ell)$ by the right action of $\UU_n^\transpose(\O_\ell)$, where $\transpose$ means transpose, to obtain  right  $\BB_n^\transpose(\O_\ell)$-invariant functions. This image is isomorphic---regardless of $\ell$---to the module $\H(\GL_n(\O_1)/\BB_n(\O_1))$, on which $\GL_n(\O_\ell)$ acts through the quotient map $\GL_n(\O_\ell)\to \GL_n(\O_1)$. 

This process of passing to the image of a canonical intertwining operator between two induced representations fits into a rather general setting, which we shall now describe. In the main body of the paper we study representations of profinite groups, such as groups of matrices over compact discrete valuation rings; but our results also apply to (and are interesting for) finite groups, such as matrix groups over the finite rings $\O_{\ell}$, and in order to minimise the technicalities in this introduction we shall restrict our attention here to the finite case.

Let $G$ be a finite group, and suppose that $U$, $L$ and $V$ are subgroups of $G$ such that $L$ normalises $U$ and $V$, and such that the map
\[
U\times L\times V \hookrightarrow G
\]
given by multiplication in $G$ is injective. We let $e_U$ and $e_V$ denote the idempotents in the complex group algebra $\H(G)$ associated to the trivial representations of $U$ and $V$, and we consider the $\H(G)$-$\H(L)$ bimodule $\H(G)e_U e_V$. Let $\pind_{U,V}$ be the functor from the category $\Rep(L)$ of complex representations of $L$ to the category of complex representations of $G$ defined by tensoring with this bimodule:
\[
\pind_{U,V}:\Rep(L) \to \Rep(G), \qquad M \mapsto \H(G)e_U e_V \otimes_{\H(L)}M.
\]
Similarly, define 
\[
\pres_{U,V}:\Rep(G) \to \Rep(L), \qquad N \mapsto e_Ue_VH(G) \otimes_{\H(G)}N.
\]

This definition is closely related to, and directly inspired by, a construction of Dat \cite{Dat_parahoric}. Note, though, that we consider only complex representations, whereas Dat studied representations over more general commutative rings. The relationship between our definition and Dat's, for complex coefficients, is discussed further in Remark \ref{rem:CMO-vs-Dat} and in Section \ref{parahoric_section}. The main novelty of the above definition relative to Dat's is that we do not require the product $ULV$ to be a group, so that for instance we could as above take $G$ to be $\GL_n(\O_\ell)$, and let $L=\TT_n(\O_\ell)$, $U=\UU_n(\O_\ell)$ and $V=\UU_n^\transpose(\O_\ell)$. When $\ell=1$, a theorem of Howlett and Lehrer  (\cite[Theorem 2.4]{Howlett-Lehrer_HC}; cf. Example \ref{ex:pind-field}) implies that the functors $\pind_{U,V}$ and $\pres_{U,V}$  in this $\GL_n$ example are isomorphic to the functors of Harish-Chandra induction and restriction. When $\ell>1$ these functors are  proper subfunctors of the more obvious  generalisations of the Harish-Chandra functors mentioned above.

\subsection{Description of the main results}
Basic properties of the functors $\pind_{U,V}$ and $\pres_{U,V}$, in the abstract setting {for profinite groups}, are presented in Section \ref{sec:definition}. For instance, these functors are adjoints on both sides; they do not depend on the order of $U$ and $V$, up to natural isomorphism; they preserve finite-dimensionality; and they satisfy a version of \lq induction in stages\rq.

The analysis of the functors $\pind_{U,V}$ and $\pres_{U,V}$ becomes considerably less complicated in cases where the product map $U\times L\times V\to G$ is a bijection. 
In many examples, such as the $\GL_n$ example considered above, this is not the case, but there is a normal subgroup $G_0 \lhd G$ such that the product map 
$U_0\times L_0\times V_0\to G_0$ is a bijection, where $H_0$ means $H\cap G_0$. In the $\GL_n$ example we can take $G_0$ to be the principal congruence subgroup $G_0=\{g\in \GL_n(\O_\ell)\ |\ g\equiv 1 \textrm{ modulo }\germ p\}$.
 
Suppose that $G$ admits such a normal subgroup $G_0$. The representation categories $\Rep(L)$ and $\Rep(G)$ decompose according to  $L_0$- and $G_0$-isotypic components, and the individual components can be described using Clifford theory. In Section~\ref{sec:Clifford} we prove that the Clifford analysis is compatible with the induction and restriction functors $\pind_{U,V}$ and $\pres_{U,V}$.

More precisely, let $\psi$ be an irreducible representation of $L_0$, and let $\phi=\pind_{U_0,V_0}(\psi)$ be the corresponding (irreducible) induced representation of $G_0$. Let $L(\psi)$ and $G(\phi)$ denote the inertia groups of $\psi$ and $\phi$. We prove in Theorems \ref{thm:C1}, \ref{thm:C2} and \ref{thm:C3} that there is a commutative diagram for induction (and a similar diagram for restriction):
\[
 \xymatrix@C=80pt{ \Rep(L)_\psi \ar[r]^-{\pind_{U,V}}  & \Rep(G)_\phi \\
 \Rep (L(\psi))_\psi \ar[r]^-{\pind_{U(\phi),V(\phi)}} \ar[u]^-{\cong}   & \Rep(G(\phi))_\phi  \ar[u]_-{\cong} \\
 \Rep^{\gamma } (L(\psi)/L_0) \ar[r]^-{\pind_{U(\phi)/U_0, V(\phi)/V_0}} \ar[u]^-{\cong}   & \Rep^{\gamma }(G(\phi)/G_0)  \ar[u]_-{\cong} 
 }
\] 
where $\Rep(H)_\theta$ stands for the representations of $H$ whose restriction to $H_0\lhd H$ contains the irreducible representation $\theta$; and $\Rep^{\gamma}(H(\theta)/H_0)$ stands for projective representations (for a certain cocycle $\gamma$) of the quotient $H(\theta)/H_0$.

As for the groups $L_0$ and $G_0$,   in many of our motivating examples they are amenable to the orbit  method: their irreducible representations correspond bijectively to coadjoint orbits in the Pontryagin duals of certain Lie algebras $\germ l_0$ and $\germ g_0$. This situation is studied in Section \ref{sec:orbit}, where we show that under appropriate assumptions the induction functor 
$\pind_{U_0,V_0}:\Rep(L_0) \to \Rep(G_0)$ 
corresponds to a natural inclusion of coadjoint orbits 
$\Lambda^*:L_0\backslash \widehat{\germ l_0} \into G_0\backslash \widehat{\germ g_0}$. That is, the diagram  
\[
 \xymatrix@R=30pt@C=50pt{ \Irr(L_0) \ar[r]^-{\pind_{U_0,V_0}} \ar[d]_-{\cong} & \Irr(G_0) \ar[d]^-{\cong}\\
{L_0}\backslash \widehat{\germ l_0} \ar[r]^-{\Lambda^*}   & G_0\backslash \widehat{\germ g_0}  
 }
\]
commutes.

Returning from the abstract setting to our motivating examples, the functors $\pind_{U,V}$ and $\pres_{U,V}$ provide a new approach to the representation theory of classical groups over compact discrete valuation rings, and the results of Sections \ref{sec:Clifford} and \ref{sec:orbit} provide tools to analyse these functors. In Section \ref{sec:sp} we illustrate the method for the symplectic group $\Sp_4(\O_2)$. The main result is a Mackey-type formula for the composition of restriction and induction to/from the Siegel Levi subgroup. The formula is the same as the usual formula \eqref{eq:Mackey_intro} for the composition of Harish-Chandra induction and restriction for the corresponding group $\Sp_4(\O_1)$ over the residue field of $\O$, which lends some support to the analogy between our functors and the Harish-Chandra functors.
This analogy is further supported by an analysis for the groups $\GL_n$, which will be presented in a sequel to this paper. 

The general methods developed in Sections \ref{sec:Clifford} and \ref{sec:orbit} and used in Section \ref{sec:sp} apply equally well to Dat's parahoric induction and restriction functors. In Section \ref{parahoric_section} we prove that Dat's parahoric induction and restriction  functors are not isomorphic to ours, in the example of the Siegel Levi in $\Sp_4(\O_2)$; this gives a negative answer to Dat's question \cite[Question 2.15]{Dat_parahoric}. We also prove that the parahoric induction and restriction functors do not satisfy the analogue of \eqref{eq:Mackey_intro} in this example.         

While our primary motivation for studying the functors $\pind_{U,V}$ and $\pres_{U,V}$ is their application to  classical groups, these functors are defined in much broader generality, and we believe that they have a useful role to play in the representation theory of more general matrix groups. In  Section \ref{sec:Iwahori}  we use these functors to study one such example, the representation theory of the Iwahori subgroup $I_n$ of $\GL_n(\O)$. The Iwahori in $\SL_2$ was previously studied from a similar point of view in \cite{Dat_parahoric} and \cite{Crisp_parahoric}. The main result of this section, Theorem \ref{thm:Iw_primitive}, states that the functors $\pind_{U,V}$ and $\pres_{U,V}$ in this context give a bijection 
\begin{equation}\label{eq:intro_prim}
\Irr(I_n) \longleftrightarrow \bigsqcup_{n_1+\cdots+n_k=n} \Prim(I_{n_1})\times \cdots \times \Prim(I_{n_k}) 
\end{equation}
between the irreducible representations of $I_n$, and tuples of {\em primitive} irreducible representations of smaller Iwahori subgroups (where {primitivity} is defined by the vanishing of the functors $\pres_{U,V}$). The problem of classifying the irreducible representations of $I_n$ remains a very difficult one---it contains the problem of counting the conjugacy classes in the group of upper-triangular matrices over the residue field $\O_1$---but the bijection \eqref{eq:intro_prim} shows that part of this classification is very simple and combinatorial in nature.

\subsection{Related constructions} 
Representations of open compact subgroups of reductive groups over local fields have received much attention in the past two decades. One approach, taken by Lusztig and Stasinski, is to generalise Deligne-Lusztig theory  \cite{Deligne-Lusztig} (which is itself a generalisation of the Harish-Chandra theory) to such groups; see \cite{Lusztig1, Stasinski}, and also \cite{Chen-stasinski} and \cite{Lusztig2}. Another approach, taken by Hill~\cite{Hill_Jord, Hill_nilp, Hill_Reg,   Hill_SSandCusp}, consists of a direct Clifford-theoretic analysis  of representations according to their restrictions to congruence kernels. In particular, in \cite{Hill_Jord} Hill establishes a Jordan decomposition for characters of general linear groups over 
{rings of integers in $p$-adic fields}, analogous to the Jordan decomposition of irreducible characters of finite reductive groups established by Lusztig, cf. \cite{Lusztig84}. Hill's work relies on an analysis of certain Hecke algebras building on the work of Howe and Moy ~\cite{Howe-Moy}. Another approach was proposed by the  third author  in~\cite{Onn} using a different variant of Harish-Chandra induction that allows one to import representations from automorphism groups of finite modules over discrete valuation rings, yielding a complete and characteristic-independent treatment in rank two. The work of Dat \cite{Dat_parahoric}, in which representations of parahoric subgroups of $p$-adic reductive groups are studied using methods closely related to those of the present paper, has already been mentioned above.

It would be of great interest to understand how all these approaches align with the one taken in this paper. The relationship between our work and that of Dat is addressed in Remark \ref{rem:CMO-vs-Dat} and Section \ref{parahoric_section}. As for the other works cited above, let us make a couple of general observations. 

The first point to note is that the natural filtration on the valuation ring $\O$ does not enter a priori into the definition of our induction/restriction functors, and in this sense our approach is more elementary than those of the above-cited works.  It is consequently more general---applying for instance to $\GL_n(R)$ for an arbitrary (pro)finite commutative ring $R$---although the usefulness of our methods beyond the setting of discrete valuation rings remains to be tested.

A second difference is one of scope. Our functors are defined with a view to making the induced representations as small as possible, and the set of representations which cannot be obtained by induction in our sense (the \lq cuspidal\rq~ representations) will be accordingly large---certainly larger than the corresponding sets for the approaches listed above. Our   goal is to develop an analogue of Harish-Chandra theory which mirrors as closely as possible  the theory for reductive groups over a finite field, yielding a description of arbitrary representations in terms of cuspidal ones and of Weyl group combinatorics. We leave untouched for now the problem of constructing (let alone classifying) the cuspidal representations.

 \subsection{Acknowledgments}  We thank George Willis and Helge Gl\"ockner for helpful discussions on tidy subgroups. The first two authors were partly supported by the Danish National Research Council through the Centre for Symmetry and Deformation (DNRF92).  The third author  acknowledges the support of the Israel Science Foundation and of the Australian Research Council.

%%%%%%%%%%%%%%%%%%%
%                                                       %
%    Section 2 -- Basic properties      %
%                                                       %
%%%%%%%%%%%%%%%%%%%

\section{Notation, definitions, and basic properties}\label{sec:definition}

In this section we define and develop basic properties of the functors $\pind_{U,V}$ and $\pres_{U,V}$ in an abstract setting. The pivotal point in this section is Proposition~\ref{prop:z}, which allows us to generalise many of Dat's results from \cite[Section 2]{Dat_parahoric} to the situation considered in this paper. We begin by setting up the notation that will be used throughout the paper.  

\subsection{Notation}\label{subsec:notation}

For a profinite group $G$ we let $\Rep(G)$ denote the category of smooth, complex representations of~$G$, that is, linear representations $\phi:G\to \GL_{\C}(M)$ in which each vector in $M$ is fixed by some open subgroup of~$G$. We will denote such a representation either by the map $\phi$ or by the space $M$, as convenient. If $M$ is any representation of~$G$ (not necessarily smooth), we let $M^\infty$ denote the $G$-subspace of vectors fixed by some open subgroup of~$G$.

Let $\H(G)$ denote the algebra of locally constant, complex-valued functions on $G$, with product given by convolution with respect to some Haar measure on $G$. Different choices of Haar measure give isomorphic algebras, the isomorphism being multiplication by the ratio $\vol_1(G)/\vol_2(G)$ of the total volumes of the two measures. The category $\Rep(G)$ is equivalent to the category of nondegenerate {left} $\H(G)$-modules, i.e. those modules $M$ which satisfy $M=\H(G)M$.  If $G$ is finite  then we will usually use counting measure as the Haar measure on $G$, in which case the map sending $g\in G$ to the $\delta$-function $\delta_g$ at $g$ extends to an isomorphism from the complex group ring $\C(G)$ to $\H(G)$.    

Let $\Irr(G)$ denote the set of isomorphism classes of irreducible smooth representations of $G$. When chances for confusion are slim we will also write $\rho\in \Irr(G)$ for an actual irreducible representation. For each $\rho \in \Irr(G)$ let $\ch_\rho\in \H(G)$ be the character $g\mapsto \trace(\rho(g))$ of $\rho$, and let $e_\rho \in \H(G)$ be the idempotent defined by 
\[
e_\rho : g \mapsto \frac{\dim_\C(\rho)}{\vol(G)} \ch_\rho(g^{-1}).
\] 
If $M$ is a smooth representation of $G$  then $e_\rho$ acts on $M$ by projecting $M$ onto its $\rho$-isotypical submodule. For the special case of the  trivial representation we write $e_G$ for the corresponding idempotent, namely, the  function on $G$ with constant value $1/\vol(G)$. The element $e_G$ acts on each smooth representation $M$ by projecting onto the submodule $M^G$ of $G$-fixed vectors. 

If $M$ is a nondegenerate left $\H(G)$-module, and $N$ is a nondegenerate right $\H(G)$-module, then by definition 
\[
N\otimes_{\H(G)} M = N\otimes_{\C} M \big / \lspan\{nf\otimes m - n\otimes fm\ |\ n\in N,\ m\in M,\ f\in \H(G)\}.
\]
Equivalently, viewing $N$ and $M$ as smooth representations of $G$, $N\otimes_{\H(G)} M$ is the space of coinvariants for the action $g:n\otimes m\to ng^{-1}\otimes gm$ of $G$ on $N\otimes_{\C} M$.  

If $H$ is a closed subgroup of $G$, then $\H(G)$ is a smooth representation of $H$ under both left- and right-translation, and consequently $\H(G)$ is an $\H(H)$-bimodule. Given a smooth representation $M$ of $H$ we write  
\[ 
\ind_{H}^G M = \left.\left\{ f:G\xrightarrow[\text{constant}]{\text{locally}} M\ \right|\ f(h g)=h\cdot f(g), \forall h \in H,g\in G \right\}
\]
for the induced representation, on which $G$ acts by right-translation. This is isomorphic to the tensor product $\H(G)\otimes_{\H(H)} M$. If the subgroup $H$ is a semidirect product $U\rtimes L$, then representations may be induced from $L$ to $G$ by first inflating to $H$ (i.e. pulling back along the quotient map $H\to L$), and then applying the functor $\ind_H^G$. The resulting functor from $\Rep(L)$ to $\Rep(G)$ is isomorphic to the functor of tensor product with the $\H(G)$-$\H(L)$ bimodule $\H(G)e_U  \cong \H(G/U)$, where $\H(G/U)$ denotes the space of locally constant functions on $G/U$.

Whenever a group $G$ acts on a set $X$ we write $G(x)$ for the stabiliser in~$G$ of $x \in X$.

The first three chapters of \cite{Renard} are a convenient reference for all of the above. Many of the examples considered here will be groups of matrices over compact subrings of non-archimedean local fields; see \cite[Chapter V]{Renard}, for instance, for more background on these.

\subsection{{Virtual Iwahori decompositions}}

Let us begin by describing the kind of groups that we shall be interested in, and giving several examples.

\begin{definition}\label{def:vI}
Let $G$ be a profinite group. A \emph{virtual Iwahori decomposition} of $G$ is a triple of closed subgroups $(U,L,V)$ of $G$, where $L$ normalises $U$ and $V$, such that
\begin{enumerate}[(1)]
\item The multiplication map $U\times L \times V \to G$ is an open embedding (and therefore a homeomorphism onto its image).
\item $G$ contains arbitrarily small open, normal subgroups $K$ for which the multiplication map 
\[
(U\cap K) \times (L\cap K) \times (V\cap K) \to K
\]
is a homeomorphism.
\end{enumerate}
An \emph{Iwahori decomposition} of $G$ is a virtual Iwahori decomposition for which the multiplication map in (1) is surjective (and therefore, a homeomorphism).
\end{definition}
 
The following immediate observation shows that the notion of virtual Iwahori decomposition is inherited by  subgroups and  quotients.   

\begin{observation}\label{ex:vI-subquotients} 
Let $(U,L,V)$ be a virtual Iwahori decomposition of $G$.
\begin{enumerate}[\rm(1)]
\item If $J$ is a closed subgroup of $G$, then $(U\cap J, L\cap J, V\cap J)$ is a virtual Iwahori decomposition of $J$. 
\item If $(X,H,Y)$ is a virtual Iwahori decomposition of $L$, then $(U\rtimes X, H, Y\ltimes V)$ is a virtual Iwahori decomposition of $G$. 
\item If $K$ is an open normal subgroup of $G$ with an Iwahori decomposition as in part (2) of Definition~\ref{def:vI}, then $(U/(U\cap K), L/(L\cap K), V/(V\cap K))$ is a virtual Iwahori decomposition of $G/K$. 
\end{enumerate}
\end{observation} 

{The concept of Iwahori decomposition first appeared in the work of Iwahori and Matsumoto on $p$-adic Chevalley groups \cite{IwMat}. The \lq virtual\rq~ version defined above is likewise motivated by examples occurring naturally in the study of reductive groups:} 

\begin{example}\label{ex:vI-p-adic}
Let $\mathbf G$   be a  connected reductive group  over a non-archimedean local field $F$, and let $G$ be any compact open subgroup of $\mathbf G(F)$. There is a maximal $F$-split torus $\mathbf T\subset \mathbf G$ (depending on $G$) with the property that if $\mathbf L$ is an $F$-rational Levi subgroup of $\mathbf G$ containing $\mathbf T$, and $\mathbf U$ and $\mathbf V$ are the unipotent radicals of an opposite pair of $F$-rational parabolic subgroups of $\mathbf G$ with common Levi factor $\mathbf L$, then the triple of subgroups $(G\cap \mathbf U(F), G\cap \mathbf L(F), G\cap \mathbf V(F))$ is a virtual Iwahori decomposition of $G$. {This follows from the Bruhat-Tits theory: one can take $\mathbf T$ to be any torus whose associated apartment in the  affine building of $\mathbf{G}(F)$ contains a point fixed by $G$. An explicit filtration of $G$ by open normal subgroups admitting Iwahori decompositions is constructed in \cite[Section 1.2]{Schneider-Stuhler}; cf. \cite[2.11]{Dat_parahoric}.} 
\end{example}

\begin{example}\label{ex:vI-GLnO}
For a specific instance of the previous example, let $G=\GL_n(\O)$, where $\O$ is the ring of integers in a non-archimedean local field. Given an ordered partition $n=n_1+\cdots+ n_m$ of $n$ as a sum of positive integers, let $L\cong \GL_{n_1}(\O)\times \cdots \times \GL_{n_m}(\O)$ be the corresponding subgroup of block-diagonal matrices in $G$. Let $U$ be the group of upper-triangular matrices in $G$ with diagonal blocks $1_{n_1\times n_1}\times \cdots \times 1_{n_m\times n_m}$, and let $V$ be the transpose of $U$. Then the triple $(U,L,V)$ is a virtual Iwahori decomposition of $G$; the principal congruence subgroups 
\[
K_\ell \coloneq \ker\left( \GL_n(\O)\to \GL_n(\O_\ell)\right)
\]
(where $\O_\ell = \O/\germ p^{\ell}$, $\germ p$ being the maximal ideal of $\O$) all admit Iwahori decompositions. Passing to quotients by the $K_\ell$ yields virtual Iwahori decompositions of the finite groups $\GL_n(\O_\ell)$. 
\end{example}

\begin{example}\label{ex:vI-parahoric}
A second virtual Iwahori decomposition of $G=\GL_n(\O)$ is given by $(U,L, V_1)$, where $U$ and $L$ are as in Example \ref{ex:vI-GLnO}, and $V_1= V\cap K_1$. In this case the image $ULV_1$ of the product mapping $U\times L\times V_1\to G$ is a subgroup of $G$. For instance, if the partition is $n=1+\cdots+1$, then $ULV_1$ is the standard \emph{Iwahori subgroup} of $\GL_n(\O)$, comprising those matrices which are upper-triangular modulo~$\germ p$.  
\end{example}

If $G$ is finite   then the condition (2) in Definition \ref{def:vI} is always satisfied, e.g.\ by the trivial subgroup $K=\{1\}$. Since the smooth representation theory of a profinite group $G$ is determined  in a very simple way  by the representations of the finite quotients of $G$, the condition (2) is therefore not essential to much of the sequel. On the other hand, this condition is convenient in places for shortening some proofs, and it is satisfied by all of our motivating examples. Nevertheless, let us note the following {quite} general construction of examples which satisfy condition (1) without---at least a priori---satisfying~(2). 

\begin{example}\label{ex:vI-tidy}
Let $\mathcal G$ be a totally disconnected locally compact group and let $\alpha:\mathcal G\to\mathcal G$ be a topological group automorphism. Suppose that the contraction subgroups 
\[
\mathcal U_\alpha=\{g\in \mathcal G\ |\ \alpha^n(g)\to 1\textrm{ as }n\to \infty\}\quad\text{and}\quad \mathcal V_\alpha  = \mathcal U_{\alpha^{-1}}
\]
are closed in $\mathcal G$.  This is {always} the case, for example, if $\mathcal G$ is a $p$-adic Lie group.

These contraction subgroups are both normalised by the closed subgroup 
\[
\mathcal L_\alpha = \{g\in \mathcal G\ |\ \{\alpha^n(g)\ |\ n\in \Z\}\textrm{ is precompact in $\mathcal G$}\},
\]
and the multiplication map 
\[
\mathcal U_\alpha \times \mathcal L_\alpha \times \mathcal V_\alpha \to \mathcal G 
\]
is an open embedding. So if $G$ is any compact open subgroup of $\mathcal G$, then the triple $(\mathcal U_\alpha\cap G, \mathcal L_\alpha\cap G, \mathcal V_\alpha\cap G)$ satisfies condition (1) of Definition \ref{def:vI}. Moreover, $G$ contains arbitrarily small open subgroups $K$ for which the multiplication map 
\[
(\mathcal U_\alpha\cap K) \times (\mathcal L_\alpha\cap K)\times (\mathcal V_\alpha\cap K)\to K
\]
is a {homeomorphism} (the so-called \emph{tidy} subgroups for $\alpha$). It is not clear to us whether $G$ contains arbitrarily small open \emph{normal} subgroups {$K$} with this property.
If  $\mathcal G$ is an analytic Lie group over a local field and the automorphism $\alpha$ is analytic (keeping the assumption that the contraction groups are closed), then it is at least true that {$\mathcal G$ contains arbitrarily small open subgroups $K$ with Iwahori decomposition $(\mathcal U_\alpha\cap K, \mathcal L_\alpha\cap K, \mathcal V_\alpha\cap K)$; cf. Example \ref{ex:orbit_tidy} for the characteristic $0$ case.}
We thank George Willis and Helge Gl\"ockner for a discussion of this example. See \cite{Baumgartner-Willis}   and \cite{Gloeckner} for details.  
\end{example}

\subsection{Definition and basic properties of the functors $\boldsymbol\pind$ and $\boldsymbol\pres$}
We now come to the main definition of the paper.  Whenever $H$ is a closed subgroup of a profinite group $G$, {the space} $\H(G)$ is a bimodule over $\H(H)$. If $L$, $U$ and $V$ are closed subgroups of $G$, and  $L$ normalises $U$ and $V$, then the action of $\H(L)$ on $\H(G)$ commutes with the idempotents $e_U\in \H(U)$ and $e_V\in \H(V)$. Thus $\H(G)e_U e_V$ is an $\H(G)$-$\H(L)$ bimodule, and $e_U e_V \H(G)$ is an $\H(L)$-$\H(G)$ bimodule.

\begin{definition}\label{def:pind}
Let $(U,L,V)$ be a virtual Iwahori decomposition of a profinite group $G$. Define the following functors:
\[
\pind_{U,V}: \Rep(L)\to \Rep(G), \qquad \pind_{U,V}:M\mapsto \H(G)e_U e_V \otimes_{\H(L)} M 
\]
\[
\pres_{U,V}: \Rep(G) \to \Rep(L), \qquad \pres_{U,V}: N \mapsto e_U e_V \H(G) \otimes_{\H(G)} N.
\]
\end{definition}

\begin{remark}\label{rem:CMO-vs-Dat} The definition in the case where $ULV$ is a subgroup of $G$ is due to Dat, who considered situations like Example \ref{ex:vI-parahoric}, see \cite[2.6, 2.11]{Dat_parahoric}. The novelty of Definition \ref{def:pind} is that we relax the requirement that $ULV$ be a group, so as to cover cases like Example \ref{ex:vI-GLnO}. See Section \ref{parahoric_section} for an example of the difference between our definition and Dat's definition of {\em parahoric induction}.
Also  note that Dat makes a further assumption in \cite{Dat_parahoric}, namely that the group $L$ should contain an open normal subgroup $L^\dagger$ such that the set $UL^\dagger V$ is a pro-$p$ subgroup of $G$. This assumption, which is needed to ensure the integrality of certain constructions in \cite{Dat_parahoric}, plays no role here, where all representations are over $\C$.
\end{remark}

 Let us make a few further remarks on Definition~\ref{def:pind}.  Firstly, since $\H(G)e_U e_V$ is the image of the bimodule map $f\mapsto fe_V$  from $\H(G)e_U$ to $\H(G)e_V$, and since every $M\in \Rep(L)$ is  a direct sum of representations of finite quotients of $L$, and hence flat as a module over $\H(L)$,  the module $\pind_{U,V}(M)$ is isomorphic to the image of the map
\begin{equation}\label{eq:intertwiner} 
J_V:\H(G)e_U\otimes_{\H(L)} M \xrightarrow{\ f\otimes m\mapsto fe_V\otimes m\ } \H(G)e_V\otimes_{\H(L)} M.
\end{equation}
The module $\H(G)e_U\otimes_{\H(L)} M$ is isomorphic as a representation of $G$ to the induced representation $\ind_{LU}^G(M)$, where $M$ is inflated to a representation of $LU$ by letting $U$ act trivially (cf. Section \ref{subsec:notation}). We similarly have $\H(G)e_V\otimes_{\H(L)} M\cong \ind_{LV}^G(M)$, and the map $J_V$ corresponds in this picture to the \lq standard intertwining operator\rq~
\[
J_V:\ind_{LU}^G(M) \to \ind_{LV}^G(M), \qquad J_V(f): g\mapsto \int_V f(vg)\,\dd v .
 \]
Similarly,  $\pres_{U,V}(N)$ is isomorphic to the image of the canonical projection
 \[ e_U:N^V \to N^U, \qquad n\mapsto \int_U un\,\dd u\]
 from the $V$-invariants to the $U$-invariants of $N$.

{As a final remark on Definition \ref{def:pind}, we note that  the definition makes sense if we assume only that $L$, $U$ and $V$ are closed subgroups of $G$ such that $L$ normalises $U$ and $V$. Some of the properties of the functors $\pind_{U,V}$ and $\pres_{U,V}$ that we shall establish below remain valid in this degree of generality: e.g., parts \eqref{item:pind-UV-VU}, \eqref{item:pind-Hom}, \eqref{item:pind-compat} and~\eqref{item:pind-stages} of Theorem \ref{thm:pind-properties}. For the applications we have in mind, the assumption that $(U,L,V)$ is a virtual Iwahori decomposition is both a natural and a useful one.}

\begin{example}\label{ex:pind-field} 
Let $G$ be a {reductive group over a finite field}, and let $LU$ and $LV$ be an opposite pair of parabolic subgroups of $G$. A theorem of Howlett and Lehrer (see \cite[Theorem 2.4]{Howlett-Lehrer_HC}) asserts that in this case the  map \eqref{eq:intertwiner}
is an isomorphism for every $M\in\Rep(L)$, and this implies that the functor $\pind_{U,V}$ is equal to the Harish-Chandra induction functor $M\mapsto \H(G)e_V \otimes_{\H(L)} M$ (and isomorphic to the analogous functor with $U$ in place of $V$). Similarly, $\pres_{U,V}$ is isomorphic to the functor  of  {Harish-Chandra restriction}. See \cite[Chapter 4]{DM} for background on Harish-Chandra functors for  {finite reductive groups}.
\end{example}

\begin{example}
Example \ref{ex:pind-field} notwithstanding, the map \eqref{eq:intertwiner} is usually far from being an isomorphism. For instance, if $G$ is a compact open subgroup of a reductive group $\mathbf{G}(F)$ as in Example \ref{ex:vI-p-adic}, and $(U,L,V)$ is the virtual Iwahori decomposition of $G$ corresponding to an opposite pair of proper parabolic subgroups of $\mathbf{G}$, then the subgroups $LU$ and $LV$ have infinite index in $G$, and hence   the representations $\ind_{LU}^G(M)$ and $\ind_{LV}^G(M)$ are infinite-dimensional for every  $M\in \Rep(L)$. By contrast, the representation $\pind_{U,V}(M)$ is finite-dimensional whenever $M$ is: see Theorem \ref{thm:pind-properties}\eqref{item:pind-nonzero}.
\end{example}

\begin{example}\label{ex:pind-commuting} 
Suppose that $(U,L,V)$ is a virtual Iwahori decomposition of $G$ such that the subgroups $U$ and $V$ commute with one another. Then the product $H\coloneq ULV$ is an open subgroup of $G$, isomorphic to $(U\times V)\rtimes L$. We have $e_U e_V = e_{U\times V}$, and there are isomorphisms of $\H(G)$-$\H(L)$ bimodules
\[
\H(G)e_U e_V = \H(G)e_{U\times V} \cong \H(G/(U\times V)).
\]
Consequently the functor $\pind_{U,V}$ is of the form 
$\Rep(L) \xrightarrow{\infl} \Rep(H) \xrightarrow{\ind} \Rep(G)$
discussed in Section \ref{subsec:notation}.
\end{example}

We shall now establish some basic properties of the functors $\pind_{U,V}$ and $\pres_{U,V}$. Many of these properties were established in \cite{Dat_parahoric}  for the case where $(U,L,V)$ is an actual, as opposed to a virtual, Iwahori decomposition of~$G$. The proofs in \cite{Dat_parahoric} mostly carry over with only minor changes to the case of a virtual Iwahori decomposition, thanks to the following analogue of \cite[Proposition 2.2]{Dat_parahoric}. The proofs of these propositions, though, are quite different.

\begin{proposition}\label{prop:z}
Let $G$ be a profinite group and let $L$, $U$ and $V$ be closed subgroups of G such that $L$ normalises $U$ and $V$.
For every $M \in \Rep(G)$ there is a linear automorphism $z_M\in \GL(M)$, commuting with the actions of $L$, $e_U$ and $e_V$, such that $z_M^{-1} e_U e_V$ is an idempotent in $\End(M)$.  
\end{proposition}

\begin{proof}
Each smooth representation $M\in \Rep(G)$ may be regarded as a representation of the infinite dihedral group $\Gamma=\langle s,t\ |\ s^2=t^2=1\rangle$, by sending $s\mapsto 2e_U-1$ and $t\mapsto 2e_V-1$. Since $G$ is profinite, every $M\in \Rep(G)$ is isomorphic to a direct sum of finite-dimensional unitary representations of $G$, which restrict to finite-dimensional unitary representations of $\Gamma$ (unitary because the idempotents $e_U$ and $e_V$ are self-adjoint in $\H(G)$). It follows that every $M\in \Rep(G)$ is semisimple as a representation of $\Gamma$, and so $M$ decomposes (uniquely) as the direct sum of its $\Gamma$-isotypic components. 

We claim that in each irreducible representation $W$ of $\Gamma$ there is a  nonzero $z_W\in \C$ such that $z_W^{-1} pq$ is an idempotent in $\GL(W)$, where $p=\frac{1}{2}(s+1)$ and $q=\frac{1}{2}(t+1)$. Indeed, since the dihedral group has an abelian normal subgroup of index two, every irreducible representation of $\Gamma$ is either one- or two-dimensional. In the one-dimensional case $p$ and $q$ commute and so we may take $z_W=1$. In the two-dimensional case, $pq$ and $(pq)^2$ are two nonzero maps between the one-dimensional subspaces $qW$ and $pW$, and so there is a  (unique)  nonzero scalar $z_W$ such that $pq=z_W^{-1}(pq)^2$.  

Having established the claim, we let $z_M\in \GL(M)$ to be the automorphism of $M$ which acts as the scalar $z_W$ on the $W$-isotypical component of $M$. It is clear from the construction that $z_M$ commutes with $e_U$ and $e_V$,  and that $z_M^{-1} e_U e_V$ is an idempotent. If $T\in \End(M)$ commutes with $e_U$ and $e_V$ then $T$ preserves the $\Gamma$-isotypic components, and so commutes with $z_M$. In particular, $z_M$ commutes with the $L$-action on $M$. 
\end{proof}

\begin{remark}\label{rem:z}
If $W$ is a two-dimensional irreducible unitary representation of the infinite dihedral group, then $z_W=\cos ^2 (\alpha_W)$, where $\alpha_W$ is the angle between the images of $p$ and $q$ in the Hilbert space $W$. Thus the eigenvalues of $z_M$ all lie in the interval $(0,1]$. {If the multiplication map $U\times L\times V\to G$ is a homeomorphism}, and $M$ is an irreducible representation of $G$, then $z_M$ is the scalar operator
\[
z_M = \begin{cases} \dim \pres_{U,V}(M) / \dim M & \text{if }\pres_{U,V} (M) \neq 0, \\ 1 & \text{if }\pres_{U,V}(M)=0; \end{cases}
\]
see \cite[Proposition 1.11]{Crisp_parahoric}. Moreover, Dat has shown that if $L$ contains an open normal subgroup $L^\dagger$ such that $UL^\dagger V$ is a pro-$p$ subgroup of $G$, then the eigenvalues lie in $\Z[1/p]$; see  \cite[Proposition 2.2]{Dat_parahoric}. 
\end{remark}

With the automorphisms $z_M$ in hand, many of the arguments from \cite[Section 2]{Dat_parahoric} carry over to our setting, and establish the following properties of the functors $\pind_{U,V}$ and $\pres_{U,V}$.

\begin{theorem}\label{thm:pind-properties} 
Let $(U, L,V)$ be a virtual Iwahori decomposition of a profinite group $G$,  and consider the functors $\pind_{U,V}$ and $\pres_{U,V}$. Then: 
 \begin{enumerate}[\rm(1)]

  \item\label{item:pind-UV-VU}  There are natural isomorphisms $\pind_{U,V}\cong \pind_{V,U}$ and $\pres_{U,V}\cong \pres_{V,U}$.

  \item\label{item:pind-Hom} $\pind_{U,V}$ is naturally isomorphic to the functor 
  \[
  \pind'_{U,V} : M\mapsto \Hom_{\H(L)}(e_V e_U \H(G), M)^{\infty} ,
  \]
  and is therefore right-adjoint to $\pres_{U,V}$.
 
  \item\label{item:pres-Hom} $\pres_{U,V}$ is naturally isomorphic to the functor 
\[
\pres'_{U,V} : N\mapsto \Hom_{\H(G)} (\H(G)e_V e_U, N)^\infty, 
\]  
  and is therefore right-adjoint to $\pind_{U,V}$. 
  \item\label{item:pind-compat} Let $(U',L,V')$ be a second virtual Iwahori decomposition of $G$, such that 
\[
U= (U\cap U')  (U\cap V'),\quad V=(V\cap U') (V\cap V'),\]
\[ U'=(U'\cap U)(U'\cap V),\quad \text{and}\quad V'=(V'\cap U)(V'\cap V).
\]
Then $\pind_{U,V}\cong \pind_{U',V'}$ and $\pres_{U,V}\cong \pres_{U',V'}$.  
\item\label{item:pind-finite} Let $K$ be an open normal subgroup of $G$ with an Iwahori decomposition $(U_K, L_K, V_K)\coloneq (U\cap K, L\cap K, V\cap K)$. The diagrams
\[
\xymatrix@C=60pt{
\Rep(L) \ar[r]^-{\pind_{U,V}} & \Rep(G) \\
\Rep(L/L_K) \ar[u]^-{\infl} \ar[r]^-{\pind_{U/U_K,V/V_K}} & \Rep(G/ K) \ar[u]_-{\infl} 
}
\qquad\text{and}\qquad
 \xymatrix@C=60pt{
\Rep(G) \ar[r]^-{\pres_{U,V}} & \Rep(L) \\
 \Rep(G/K) \ar[u]^-{\infl} \ar[r]^-{\pres_{U/U_K,V/V_K}} & \Rep(L/L_K)\ar[u]_-{\infl} 
}
\]
commute up to natural isomorphism. (Here $\infl$ denotes inflation.)
\item\label{item:pind-nonzero} $\pind_{U,V}(M)$ is nonzero whenever $M$ is nonzero, and $\pind_{U,V}(M)$ is finite-dimensional whenever $M$ is finite-dimensional.

\item\label{item:pind-stages} If $(X,H,Y)$ is a virtual Iwahori decomposition of $L$, then 
\[
    \pind_{U,V} \circ \pind_{X,Y} \cong \pind_{U\rtimes X,Y\ltimes V}
\]
   as functors $\Rep(H)\to \Rep(G)$.
\end{enumerate}
\end{theorem}

Parts \eqref{item:pind-Hom} and \eqref{item:pres-Hom} are instances of the following general fact, whose proof generalises the argument of \cite[Corollaire 2.7]{Dat_parahoric}:

\begin{lemma}\label{lem:ind_coind}
Let $H$ and $K$ be closed subgroups of a profinite group $G$. Let $X\subseteq \H(G)$ be an $\H(H)$-$\H(K)$ subbimodule, and denote by $X^*$  the image of $X$ under the involution 
$f^*(g)=\overline{f(g^{-1})}$ on $\H(G)$; note that $X^*$ is an $\H(K)$-$\H(H)$ bimodule. Suppose that for every open normal subgroup $H_1\subseteq H$ there is an open normal subgroup $G_1\subseteq G$ satisfying 
\begin{equation}\label{ind_coind_eq}
e_{H_1}X \subseteq e_{G_1}\H(G).
\end{equation}
 Then the functors $\Rep(K)\to \Rep(H)$ defined by 
\[
M \mapsto X\otimes_{\H(K)} M  \qquad \text{and} \qquad M \mapsto \Hom_{\H(K)} (X^*, M)^{\infty}
\]
are naturally isomorphic.
\end{lemma}

\begin{proof}
Consider the natural transformation 
\[
\Phi  :  X\otimes_{\H(K)} M \to \Hom_{\H(K)}(X^*, M), \qquad 
\Phi (x_1\otimes m) : x_2^* \mapsto (x_2^* x_1)\restrict_K \cdot m  
\] 
where $(x_2^* x_1)\restrict_K$ means the restriction of the convolution product $x_2^* x_1\in \H(G)$ to the subgroup $K$. Fix an open normal subgroup $H_1\subseteq H$. We will show that the map $\Phi$ restricts to an isomorphism between the respective subspaces of $H_1$-fixed vectors.

Let $G_1$ be an open normal subgroup of $G$ satisfying condition \eqref{ind_coind_eq}, {so that} the subspace ${e_{H_1}X\subset \H(G)}$ consists exclusively of $G_1$-invariant functions. We may then replace $G$ by $G/G_1$, and assume for the rest of the proof that $G$ is a finite group. Furthermore, the module $M$ decomposes as a direct sum of finite-dimensional modules, and the natural map $\Phi$ commutes with direct sums, so we may assume that $M$ is finite-dimensional.

Now, the pairing 
\[
X^*\times X \to \C \qquad (x_2^*, x_1) \mapsto (x_2^*x_1)(1) 
\]
is nondegenerate, since it is the restriction to $X$ of the natural $L^2$-inner product on $\H(G)$. It follows from this, and from the standard duality theory of finite-dimensional vector spaces, that the map 
\[
\Psi: X\otimes_{\C} M \to \Hom_{\C}(X^*, M) \qquad \Psi(x_1\otimes m): x_2^* \mapsto (x_2^*x_1)(1)\cdot m 
\]
is an isomorphism. The map $\Psi$ descends to an isomorphism of $K$-coinvariants 
\begin{equation}\label{ind_coind_proof1}
  X\otimes_{\H(K)} M \xrightarrow{\cong} \left(\Hom_{\C}(X^*, M)\right)_K,
\end{equation}
where the $K$-action on $\Hom_{\C}(X^*,M)$ is by conjugation. Averaging over $K$ gives an isomorphism of $K$-coinvariants with $K$-invariants:
\begin{equation}\label{ind_coind_proof2}
\left(\Hom_{\C}(X^*, M)\right)_K \xrightarrow[\cong]{T \mapsto \int_K k T k^{-1}\, \dd k} \Hom_{\H(K)}(X^*,M),
\end{equation}
and the map $\Phi$ is the composition of the isomorphisms \eqref{ind_coind_proof1} and \eqref{ind_coind_proof2}.
\end{proof}

\begin{proof}[Proof of Theorem \ref{thm:pind-properties}]
To prove part \eqref{item:pind-UV-VU}, let $z_{\H(G)}$ be the automorphism of $\H(G)$ obtained by applying Proposition \ref{prop:z} to the left-translation action of $G$. Then the maps
\[
e_U e_V \H(G) \xrightarrow{f\mapsto e_V f} e_V e_U \H(G) \quad \text{and} \quad e_V e_U \H(G) \xrightarrow{f\mapsto z_{\H(G)}^{-1} e_U f} e_U e_V \H(G)
\]
are mutually inverse isomorphisms of $\H(L)$-$\H(G)$ bimodules, giving rise to a natural isomorphism of functors ${\pres_{U,V}\cong \pres_{V,U}}$. A similar argument, using the right action of $G$ on $\H(G)$, gives $\pind_{U,V}\cong \pind_{V,U}$.

To prove part \eqref{item:pind-Hom} we apply Lemma \ref{lem:ind_coind} with $H=G$, $K=L$, and $X=\H(G)e_U e_V$. The hypothesis \eqref{ind_coind_eq} is trivially satisfied and we conclude that the functor $\pind_{U,V}$ is naturally isomorphic to $\pind'_{U,V}$. The standard $\otimes$-$\Hom$ adjunction implies that $\pind'_{U,V}$ is right-adjoint to $\pres_{V,U}$, and we have $\pres_{V,U} \cong \pres_{U,V}$ by part~(1); see \cite[I.2.2 (Corollaire)]{Renard} for a formulation and proof of the adjunction in the present context.

To prove part \eqref{item:pres-Hom} we apply Lemma \ref{lem:ind_coind} again, this time with $H=L$, $K=G$ and ${X=e_U e_V \H(G)}$. To verify the hypothesis \eqref{ind_coind_eq}, fix an open normal subgroup $H_1\subseteq L$. Then there is an open normal subgroup $G_{0}\subseteq G$ having an Iwahori decomposition $(U_{0}, L_{0}, V_{0})$, where $L_{0}$ is contained in $H_1$. Here $Y_{0}$ means $Y\cap G_{0}$ for every subset $Y \subset G$. 
We then have
\[
e_{H_1}e_U e_V\H(G) \subseteq e_{L_{0}}e_U e_V \H(G) = e_U(e_{U_{0}}e_{L_{0}}e_{V_{0}})e_V \H(G) \subseteq e_{G_{0}}\H(G),
\]
so \eqref{ind_coind_eq} is satisfied {by $G_1=G_0$}. Now Lemma \ref{lem:ind_coind} implies that $\pres_{U,V}$ is isomorphic to $\pres'_{U,V}$, which is right-adjoint to $\pind_{U,V}$ {by the argument of part \eqref{item:pind-Hom}}. 

Part \eqref{item:pind-compat} follows from Proposition \ref{prop:z}, as in \cite[Lemme 2.9]{Dat_parahoric}.

Part~\eqref{item:pind-finite} follows from the equality $e_K = e_{U_K}e_{L_K} e_{V_K}$, as was remarked in \cite[p.272]{Dat_parahoric}. {It is also a consequence of  Theorem \ref{thm:C1}, below.}

The finite-dimensionality assertion in part \eqref{item:pind-nonzero} follows from part \eqref{item:pind-finite}: every finite-dimensional smooth representation of $L$ is inflated from a representation of some finite quotient $L/L_K$, and the functor $\pind_{U/U_K,V/V_K}$ obviously preserves finite-dimensionality. To prove that $\pind_{U,V}(M)\neq 0$ as long as $M\neq 0$, fix a nonzero $m\in M$ and let $f\in \ind_{LU}^G (M)$ be the function supported on the open set $ULV\subset G$, and given there by $f(ulv)=l\cdot m$. The image of $f$ under the intertwiner $J_V:\ind_{LU}^G (M) \to \ind_{LV}^G (M)$ (see \eqref{eq:intertwiner}) is nonzero, because 
\[
(J_V f)(1) = \int_V f(v)\, \dd v =m,
\]
and so $\pind_{U,V}(M)\cong \Image(J_V)$ is nonzero.

For part \eqref{item:pind-stages}, convolution over $L$ gives an isomorphism of $\H(G)$-$\H(H)$ bimodules
\[
\H(G)e_U e_V \otimes_{\H(L)} \H(L)e_X e_Y \xrightarrow{\cong} \H(G)e_U e_V e_X e_Y. 
\]
Now $e_X$ commutes with $e_V$ since $X$ normalises $V$, and we have $e_U e_X = e_{U\rtimes X}$ and $e_V e_Y = e_{Y\ltimes V}$, and so we have produced an isomorphism between the bimodules representing the functors $\pind_{U,V}\circ \pind_{X,Y}$ and $\pind_{U\rtimes X, Y\ltimes V}$.
\end{proof}

If the triple $(U,L,V)$ is an actual Iwahori decomposition of $G$, then the functors $\pind_{U,V}$ and $\pres_{U,V}$ enjoy the following additional properties, which we recall from \cite{Dat_parahoric} and \cite{Crisp_parahoric} for the reader's convenience:

\begin{theorem}\label{thm:pind-Iw}
Suppose that $(U,L,V)$ is an Iwahori decomposition of a profinite group $G$. Then:
\begin{enumerate}[\rm(1)]
\item\label{item:pind-Iw-irr} $\pind_{U,V}$ sends $\Irr(L)$ to $\Irr(G)$, while $\pres_{U,V}$ sends $\Irr(G)$ to $\Irr(L)\sqcup \{0\}$.
\item\label{item:pind-Iw-prespind} $\pres_{U,V}\circ\pind_{U,V}\cong \id_{\Rep(L)}$.
\item\label{item:pind-Iw-pindpres} If $M\in \Irr(G)$ and $\pres_{U,V}(M)\neq 0$, then $\pind_{U,V}  \pres_{U,V} (M)\cong M$.
\item\label{item:pind-Iw-Hom} If $M\in \Irr(G)$, then either $\Hom_L(M^U, M^V)$ is zero, in which case $\pres_{U,V}(M)=0$; or $\Hom_L(M^U, M^V)$ is one-dimensional, in which case it is spanned by the operator $e_Ve_U$, which is an isomorphism, and $\pres_{U,V}(M)\cong M^U\cong M^V$. 
\item\label{item:pind-Iw-UV} Given $M\in \Irr(G)$ and $N\in \Irr(L)$, one has $M \cong \pind_{U,V}(N)$ if and only if $N$ is a common subrepresentation of $M^U$ and $M^V$.
\item\label{item:pind-Iw-character} For each $(\phi,M) \in \Irr(G)$ there is a nonzero scalar $c$ such that 
\[
  e_U e_\phi e_V = c e_U  e_{\pres_{U,V}(\phi)} e_V  
\]
as operators on $\H(G)$.
\end{enumerate}
\end{theorem}

\begin{proof}
Parts \eqref{item:pind-Iw-irr} and \eqref{item:pind-Iw-prespind} are proved in \cite[Corollaire 2.10]{Dat_parahoric}.  Part \eqref{item:pind-Iw-Hom} is proved in \cite[Lemma 1.10]{Crisp_parahoric}, and part \eqref{item:pind-Iw-UV} follows from parts \eqref{item:pind-Iw-Hom} and \eqref{item:pind-Iw-pindpres}. 

To prove part \eqref{item:pind-Iw-pindpres}: if $\pres_{U,V}(M)\neq 0$ then the adjunction in part \eqref{item:pind-Hom} of Theorem \ref{thm:pind-properties} gives a nonzero intertwiner $M\to \pind_{U,V}  \pres_{U,V}(M)$. Both of these representations are irreducible (by part \eqref{item:pind-Iw-irr}), and so they are isomorphic.

Part \eqref{item:pind-Iw-character} follows from the character formula for $\pres_{U,V}$ proved in \cite[Proposition 1.11]{Crisp_parahoric}. That formula implies that there is a nonzero $s\in \C$ such that  
\[
\ch_{\pres_{U,V}(\phi)}(l) = s \int_U \int_V \ch_\phi(vlu)\,  \dd u \,  \dd v 
\]
for all $l\in L$. Writing $\sim$ to indicate equality up to a nonzero scalar multiple, the operator $e_U e_{\pres_{U,V}(\phi)} e_V$ is thus given by  
\[
\begin{aligned}
e_U e_{\pres_{U,V}(\phi)} e_V  & \sim  \int_U \int_L \int _V \ch_{\pres_{U,V}(\phi)}(l^{-1}) ulv\, \dd u\, \dd l \, \dd v \\
&\sim \int_U \int_L \int _V \left( \int_U \int_V \ch_\phi (v_1^{-1} l^{-1} u_1^{-1})\, \dd u_1\,\ \dd v_1\right) ulv\, \dd u\, \dd l \, \dd v \\
&\sim \int_U \int_L \int _V \left( \int_U \int_V \ch_\phi (v_1^{-1} l^{-1} u_1^{-1})\, \dd u_1\,\ \dd v_1\right) u(u_1lv_1)v\, \dd u\, \dd l \, \dd v \\
&\sim \int_U \int_G \int_V \ch_\phi(g^{-1}) ugv\, \dd u\, \dd g\, \dd v \\
&\sim e_U e_\phi e_V 
\end{aligned}
\]
where in the third step we used invariance of the Haar measures, and in fourth we used the fact that the product of the Haar measures on $U$, $L$ and $V$ is a Haar measure on $G=ULV$.
\end{proof}

%%%%%%%%%%%%%%%%%%%
%                                                       %
%    Section 3 -- Clifford theory         %
%                                                       %
%%%%%%%%%%%%%%%%%%%

\section{Relations of the functors $\pind$ and $\pres$ with Clifford theory}\label{sec:Clifford}

The functors $\pind_{U,V}$ and $\pres_{U,V}$ can be difficult to work with, since the bimodule $\H(G)e_U e_V$ is not obviously the space of functions on any nice $G\times L$-space. The situation where $(U,L,V)$ is an actual, rather than a virtual, Iwahori decomposition of $G$ is significantly easier to deal with; see Sections  \ref{sec:orbit} and \ref{sec:Iwahori}, for instance. If $G$ admits only a virtual Iwahori decomposition, then $G$ contains an open  normal subgroup $G_0$ which admits an actual Iwahori decomposition (this is part of the definition). In this section we will firstly recall how Clifford theory reduces the study of representations of $G$ to that of projective representations of certain subgroups of $G/G_0$; and then we will show how the induction and restriction functors $\pind$ and $\pres$ are compatible with this reduction.

\subsection{Review of Clifford theory}\label{subsec:Clifford_review}

Let us recall the basic assertions of Clifford theory. Details can be found in~\cite{Karpilovsky}, for example.  

Let $G_0$ be an open normal subgroup of a profinite group $G$. Then $G$ acts by conjugation on the set $\Irr(G_0)$ of isomorphism classes of irreducible representations of $G_0$. 
For each $\phi\in \Irr(G_0)$ we let $\Rep(G)_\phi$ denote the category of smooth representations of $G$ whose restriction to $G_0$ contains only representations in the $G$-orbit~$G\cdot\phi$. 
The first assertion of Clifford theory is:
\begin{equation}\tag{C1}\label{C1}
\text{$\Rep(G)$ is equivalent to the product }\prod_{G\cdot \phi \in G\backslash \Irr(G_0)} \Rep(G)_\phi.
\end{equation}

Fix an irreducible representation $\phi:G_0\to \GL(W)$ of $G_0$, and let $G(\phi)$ denote the stabiliser of $\phi$ in $G$. Since $\phi$ is smooth, it is trivial on some open normal subgroup $G_{00}$ of $G$, and replacing $G$ by $G/G_{00}$ 
we might as well assume---as we shall, for the rest of Section \ref{subsec:Clifford_review}---that $G$ is finite. We use the counting measure on $G$ to define the convolution on $\H(G)$, so that the $\delta$-functions $\delta_g$ satisfy $\delta_g\delta_h=\delta_{gh}$.
 
Representations may be induced from $G(\phi)$ to $G$ in the usual way (see Section \ref{subsec:notation}). The second assertion of Clifford theory is:
\begin{equation}\tag{C2}\label{C2}
\text{The functor \ $\ind:\Rep(G(\phi))_\phi\to \Rep(G)_{\phi}$ \ is an equivalence of categories.}
\end{equation}
An inverse is given by the functor which sends a representation $M \in \Rep(G)_{\phi}$ to its {$G(\phi)$-}subspace $e_\phi M$, where $e_\phi\in \H(G_0)$ is the central idempotent associated to $\phi$. Note that the category $\Rep(G(\phi))_\phi$ is equivalent, in an obvious way, to the category of modules over the direct-summand $e_\phi \H(G(\phi))$ of the algebra $\H(G(\phi))$.

Let $\ol{G}$ denote the quotient $G/G_0$, and let $\theta:G\to \ol{G}$ be the quotient map. Schur's lemma implies that $\phi$ admits a \emph{projective extension} {to $G(\phi)$}, i.e.\ a map $\phi':G(\phi)\to \GL(W)$ which becomes a group homomorphism upon passing to the quotient $\operatorname{PGL}(W)$, and which satisfies 
\[
\phi'(g_0 g) = \phi(g_0)\phi'(g) \quad \text{and}\quad \phi'(g g_0) = \phi'(g)\phi(g_0)
\]
for all $g\in G(\phi)$ and all $g_0\in G_0$. These two properties imply that there is a two-cocycle $\gamma:\ol{G}(\phi)\times \ol{G}(\phi)\to \C^\times$ whose inflation to a cocycle on $G(\phi)$, which we also denote by $\gamma$, satisfies 
\begin{equation}\label{eqn:def.of.gamma}
\phi'(g_1)\phi'(g_2)= \gamma(g_1,g_2)^{-1}\phi'(g_1 g_2).
\end{equation}

We call $\gamma^{-1}$ the two-cocycle associated to $\phi'$; the cocycles associated to different choices of projective extensions of $\phi$ are cohomologous. The projective representation $\phi'$ may be regarded as a module over the twisted group algebra $\H^{\gamma^{-1}}(G(\phi))$, a construction that we now recall. To a two-cocycle $\alpha$ on a finite group $\Gamma$, one may associate the {\em twisted group algebra} $\H^\alpha(\Gamma)$, that is, the algebra of complex-valued functions on $G$, with twisted convolution multiplication $\cdot_\alpha$ defined on the basis $\left\{ \delta_g \ |\ g \in \Gamma \right \}$ of $\delta$-functions by 
\[
\delta_{g_1}\cdot_\alpha \delta_{g_2} \coloneq \alpha(g_1,g_2) \delta_{g_1 g_2}.
\]

We let $\Rep^\alpha(\Gamma)$ denote the category of $\H^\alpha(\Gamma)$-modules. An immediate consequence of the definition is that if $\alpha,\beta: \Gamma \times \Gamma \to \C^\times$ are two-cocycles, and $M$ and $N$ are modules over $\H^\alpha(\Gamma)$ and $\H^\beta(\Gamma)$  respectively, then $M \otimes N$ is naturally an $\H^{\alpha \beta}(\Gamma)$-module with respect to the diagonal action.

 Returning to our setup, if $M$ is an $\H^\gamma(\ol{G}(\phi))$-module, then by inflation $M$ is also an $\H^\gamma(G(\phi))$-module. As $W \in \Rep^{\gamma^{-1}}(G(\phi))$ we get that $M \otimes W$ is 
an $\H^{\gamma \cdot \gamma^{-1}}(\ol{G}(\phi))$-module, i.e.\ an ordinary (as opposed to a projective) representation of $G(\phi)$, and the third assertion of Clifford theory is: 
\begin{equation}\tag{C3}\label{C3}
\text{The functor \ $\otimes\phi':\Rep^{\gamma}(\ol{G}(\phi))\xrightarrow{M \mapsto M\otimes W} \Rep(G(\phi))_{\phi}$  \ is an equivalence of categories.}
\end{equation}
An equivalent formulation of \eqref{C3}, which we shall use below, is that the map
\begin{equation*}
\theta\otimes\phi' : \H(G(\phi))e_\phi \to \H^\gamma(\ol{G}(\phi))\otimes_{\C} \End (W) \qquad \delta_g e_\phi \mapsto \delta_{\theta(g)}\otimes \phi'(g)
\end{equation*}
is an isomorphism of algebras. Then the equivalence \eqref{C3} decomposes as 
\begin{equation}\label{eq:C3_decomposition}
\Rep^\gamma(\ol{G}(\phi)) \xrightarrow[\cong]{M\mapsto M\otimes W} \Mod\left(\H^\gamma(\ol{G}(\phi))\otimes \End (W)\right) \xrightarrow[\cong]{ (\theta\otimes\phi')^* } \Rep(G(\phi))_\phi.
\end{equation}
{Here $\Mod(R)$ denotes the category of left $R$-modules.}

This ends our review of Clifford theory. We shall now explain the compatibility with the functors $\pind$ and $\pres$.

\subsection{Induction and Clifford theory}

In this section we let $G$ be a profinite group, with a virtual Iwahori decomposition $(U,L,V)$ as in Definition \ref{def:vI}. We also fix one open normal subgroup $G_0\subset G$ for which  the product mapping 
 \[ U_0 \times L_0 \times V_0 \to G_0\]
 is a homeomorphism, where $H_0\coloneq H\cap G_0$ for every subgroup $H\subseteq G$. We consider the induction functors
 \[ \pind=\pind_{U,V}:\Rep(L)\to \Rep(G) \qquad \text{and}\qquad \pind_0 = \pind_{U_0,V_0}:\Rep(L_0)\to \Rep(G_0),\]
 along with their adjoint restriction functors $\pres$ and $\pres_0$, as in Definition \ref{def:pind}.

It follows from part \eqref{item:pind-Iw-irr} of Theorem \ref{thm:pind-Iw} that the functor $\pind_0$ sends irreducible representations of $L_0$ to irreducible representations of $G_0$, and thus produces a map from $\Irr(L_0)$ to $\Irr(G_0)$.

\begin{lemma}\label{L_eq_lemma}
The map $\pind_0:\Irr(L_0)\to \Irr(G_0)$ is $L$-equivariant and injective.
\end{lemma}

\begin{proof}
$L$ normalises $U_0$ and $V_0$, and so commutes with $e_{U_0}$ and $e_{V_0}$.
The injectivity follows from part \eqref{item:pind-Iw-prespind} of Theorem \ref{thm:pind-Iw}.
\end{proof}

The functors $\pind$ and $\pres$ are compatible with the decomposition \eqref{C1}, in the following sense.

\begin{theorem}\label{thm:C1}
With the above notation one has 
$\pind(\Rep(L)_\psi)\subseteq \Rep(G)_{\pind_0(\psi)}$,
for every $\psi\in \Irr(L_0)$. 
\end{theorem}

\begin{proof}
We first claim that $\H(G)e_U e_V$ is isomorphic,  as an $\H(G)$-$\H(L)$ bimodule,  to some  submodule of $\H(G)e_{U_0}e_{V_0}$. This is because 
\[
\H(G)e_U e_V \subseteq \H(G) e_{U_0} e_V \cong \H(G)e_V e_{U_0} \subseteq \H(G)e_{V_0} e_{U_0} \cong \H(G)e_{U_0} e_{V_0} ,
\]
where the inclusions hold because $U_0$ and $V_0$ are subgroups of $U$ and $V$, respectively, and the isomorphisms hold by part \eqref{item:pind-UV-VU} of Theorem \ref{thm:pind-properties}.

For each $N\in \Rep(L)_\psi$ we now have (up to $G$-equivariant isomorphism)
\begin{equation}\label{eqn:submodules.induction} 
\pind (N) \subseteq \H(G)e_{U_0} e_{V_0} \otimes_{\H(L)} N \subseteq \H(G)e_{U_0}e_{V_0}\otimes_{\H(L_0)} \res^L_{L_0}(N)
\end{equation}
where the first inclusion holds because of the inclusion of bimodules established above, and the second inclusion holds because the tensor product over $\H(L)$ is a quotient---and therefore also a submodule---of the tensor product over the subalgebra $\H(L_0)$. The restriction to $G_0$ of the right-hand side in \eqref{eqn:submodules.induction} is a direct sum of $G$-conjugates of $\pind_0 (\res^L_{L_0} (N))$, where $\res^L_{L_0}(N)$ is a direct sum of $L$-conjugates of $\psi$. Since $\pind_0$ is $L$-equivariant this implies that the restriction of $\pind (N)$ to $G_0$ is a direct sum of $G$-conjugates of $\pind_0(\psi)$, as claimed.
\end{proof}

For the rest of this section we shall fix an irreducible representation $\psi:L_0\to \GL(W)$ of $L_0$, and study the functor $\pind$ on the subcategory $\Rep(L)_\psi$. There is an open normal subgroup $G_{00}\subset G$ with an Iwahori decomposition $G_{00}= U_{00} L_{00} V_{00}$, such that $\psi$ is trivial on $L_{00}$. Part \eqref{item:pind-finite} of Theorem \ref{thm:pind-properties} allows us to replace $G$ by the quotient $G/G_{00}$, and so we may assume without loss of generality for the rest of this section that \textbf{$\boldsymbol G$ is a finite group}. We consequently take all Haar measures to be counting measures, so that the $\delta$ functions on elements of $G$ satisfy $\delta_g\delta_h=\delta_{gh}$ inside $\H(G)$.

To simplify the notation let us write $\phi$ for $\pind_0(\psi)$. Lemma \ref{L_eq_lemma} implies that $G(\phi)\cap L = L(\psi)$. Let $U(\phi)$ and $V(\phi)$ denote the inertia groups of $\phi$ in $U$ and $V$, respectively, and consider the functor
\[
\pind_\psi\coloneq \pind_{U(\phi),V(\phi)}:\Rep(L(\psi))_{\psi} \to \Rep(G(\phi))_{\phi}
\] 
given by tensor product with the $e_\phi\H(G(\phi))$-$e_\psi\H(L(\psi))$ bimodule $e_\phi \H(G(\phi))e_{U(\phi)}e_{V(\phi)}e_\psi$. 

The functors $\pind$ and $\pres$ are compatible with the equivalence \eqref{C2} as follows:

\begin{theorem}\label{thm:C2} The diagram
\begin{equation*} 
 \xymatrix@C=50pt{
\Rep(L)_{\psi} \ar[r]^-{\pind} & \Rep(G)_{\phi} \\
\Rep(L(\psi))_\psi \ar[u]^-{\ind}_-{\cong} \ar[r]^-{\pind_\psi} & \Rep(G(\phi))_{\phi} \ar[u]_-{\ind}^-{\cong}
}
\end{equation*}
commutes up to a natural isomorphism.
\end{theorem}

\begin{proof}
We will replace the right-hand vertical arrow in the diagram by its inverse, and prove that the diagram
\begin{equation}\label{Clifford_diagram2}
 \xymatrix@C=50pt{
\Rep(L)_{\psi} \ar[r]^-{\pind} & \Rep(G)_{\phi}\ar[d]^-{e_\phi} \\
\Rep(L(\psi))_\psi \ar[u]^-{\ind} \ar[r]^-{\pind_\psi} & \Rep(G(\phi))_{\phi}
}
\end{equation}
commutes up to natural isomorphism. This amounts to producing an isomorphism of $G(\phi)$-$L(\psi)$ bimodules
\begin{equation} \label{Clifford_equation} 
e_{\phi}\H(G) e_{U} e_{V} e_\psi \cong 
e_{\phi} \H(G(\phi))e_{U(\phi)}e_{V(\phi)} e_\psi.
\end{equation}

We first claim that
\begin{equation}\label{Clifford_step1}
e_{\phi}\H(G) e_{U} e_{V}e_\psi = e_{\phi}\H(G(\phi))e_{U(\phi)} e_{V} e_\psi,
\end{equation}
the equality holding inside $\H(G)$. For vectors $x$ and $y$ we write $x \sim y$ if they differ by a non-zero scalar multiple. To prove~\eqref{Clifford_step1} we compute for every $g\in G$:
\[
 \delta_g e_{U}e_{V} e_\psi  {=} \delta_g e_{U} ( e_{U_0} e_\psi e_{V_0} )e_{V} {e_\psi}
 \sim  \delta_g e_{U} (e_{U_0}e_\phi e_{V_0})e_{V} {e_\psi}
 {=} \delta_g e_U e_{\phi} e_V {e_\psi} 
 \sim \sum_{u\in U/U(\phi)} {\delta_{gu}} e_\phi e_{U(\phi)} e_{V} {e_\psi},
\] 
where {in the first step we have used that $e_\psi$ is an idempotent which commutes with $e_V$, and in the second step we have used part \eqref{item:pind-Iw-character} of Theorem \ref{thm:pind-Iw}}. Orthogonality of characters then implies 
\begin{align*}
e_\phi \delta_g e_{U}e_{V}e_\psi & \sim
 \sum_{u\in U/U(\phi)} e_\phi e_{(gu)\cdot\phi} {\delta_{gu}} e_{U(\phi)} e_{V} {e_\psi}\\
& \sim 
\begin{cases}   e_\phi  {\delta_{gu}}  e_{U (\phi)} e_{V}e_\psi & \text{if $\exists u\in U$ with $gu\in G(\phi)$,}\\ 0 &\text{otherwise.}
\end{cases}
\end{align*}
This shows that every element of the left-hand side of \eqref{Clifford_step1} can be written as an element of the right-hand side, and vice versa.

Now, the $G(\phi)\times L(\psi)$-equivariant map
\begin{equation}\label{Clifford_step2}
e_\phi \H(G(\phi)) e_{U(\phi)}e_{V(\phi)}e_\psi \xrightarrow{f\mapsto fe_{V}} 
e_\phi \H(G(\phi)) e_{U(\phi)}e_{V}e_\psi
\end{equation}
is obviously surjective. It is injective as well, for if $f\in \H(G(\phi))e_{V(\phi)}$ then 
\[
fe_V \sim \sum_{v\in V(\phi)\backslash V} fe_{V(\phi)} \delta_v = \sum_{v\in V(\phi)\backslash V} f\delta_v,
\]
where the functions $f\delta_v$ (as $v$ varies over $V(\phi)\backslash V$) are supported on the disjoint cosets $G(\phi)v$ and are therefore linearly independent. 
Thus the map \eqref{Clifford_step2} is an isomorphism. Composing with the equality \eqref{Clifford_step1} gives the desired isomorphism \eqref{Clifford_equation}.
\end{proof}

We are still fixing an irreducible representation $\psi:L_0\to \GL(W)$ and {letting} $\phi$ {denote the induced representation} $\pind_0(\psi):G_0\to GL(\pind_0 (W))$. 
Consider the  quotients 
\[
\ol{G}(\phi)= G(\phi)/G_0,\quad \ol{L}(\psi) = L(\psi)/L_0,\quad  \ol{U}(\phi)= U(\phi)/U_0,\quad  \ol{V}(\phi) =V(\phi)/V_0.
\] 
 
Choose a projective extension $\phi'$ of $\phi$ to $G(\phi)$, and let $\gamma^{-1}$ be the associated two-cocycle as in \eqref{eqn:def.of.gamma}.
Part~\eqref{item:pind-Iw-prespind} of Theorem \ref{thm:pind-Iw} implies that there is an $L_0$-equivariant isomorphism 
\[
\Theta:W\xrightarrow{\cong} \phi(e_{U_0})\phi(e_{V_0})\pind_0 (W),
\]
unique up to a nonzero scalar {multiple}. 
Since the subgroup $L(\psi)$ normalizes the subgroups $U_0$ and $V_0$ it follows that $\phi'(l)$ commutes with $\phi(e_{U_0})\phi(e_{V_0})$ for every $l\in L(\psi)$, and therefore
$\phi'(l)$ stabilises the subspace $\phi(e_{U_0})\phi(e_{V_0})\pind_0 (W) \subset \pind_0 (W)$. 
The map $\psi':L(\psi)\to \GL(W)$ defined by 
\begin{equation}\label{eq:psiprime_def} 
\psi'(l) = \Theta^{-1}\phi'(l)\Theta
\end{equation}
is then a projective extension of $\psi$ to $L(\psi)$, independent of the choice of $\Theta$.
(It does depend, however, on the choice of $\phi'$.) An easy argument shows that the resulting two-cocycle on $\ol{L}(\psi)$ is just the restriction of the two-cocycle $\gamma$ to ${\ol{L}(\psi)}$, and we shall therefore denote both two-cocycles by the same letter.

\begin{lemma}\label{lem:C3_eX}
Given a projective extension $\phi'$ of $\phi$ as above, there are unique scalars $a_x,b_y\in \C^\times$ for $x\in \ol{U}(\phi)$ and $y\in \ol{V}(\phi)$ such that the elements 
\[
e^{\phi'}_{\ol{U}(\phi)} \coloneq \frac{1}{|\ol{U}(\phi)|} \sum_{x\in \ol{U}(\phi)} a_x \delta_x 
\qquad\text{and}\qquad e_{\ol{V}(\phi)}^{\phi'} \coloneq \frac{1}{|\ol{V}(\phi)|}\sum_{y\in \ol{V}(\phi)} b_y \delta_y
\] 
are idempotents in $\H^\gamma(\ol{G}(\phi))$, commute with the subalgebra $\H^{\gamma}(\ol{L}(\psi))$, and such that the image of the elements $e_{U(\phi)}e_\phi$ and $e_{V(\phi)}e_\phi$ under the isomorphism of algebras
\[
\theta\ot \phi':\H(G(\phi))e_{\phi} \to  \H^{\gamma}(\ol{G}(\phi))\ot \End(\pind_0 (W))
\] 
are $e_{\ol{U}(\phi)}^{\phi'}\ot \phi(e_{U_0})$ and $e_{\ol{V}(\phi)}^{\phi'}\ot \phi(e_{V_0})$, respectively.
\end{lemma} 

\begin{proof}
We will prove the lemma for the $U$-subgroups. The proof for the $V$-subgroups is identical.

We know, by Parts \eqref{item:pind-Iw-prespind} and \eqref{item:pind-Iw-Hom}  of Theorem \ref{thm:pind-Iw}, 
that $\phi(e_{U_0})\pind_0 (W)\cong W$ as representations of $L_0$, and in particular that $\phi(e_{U_0})\pind_0(W)$ is irreducible over $L_0$.

Let $u\in U(\phi)$. Since $U_0$ is a normal subgroup of $U(\phi)$, we know that $u$ commutes with $e_{U_0}$ in $\H(G(\phi))$, and so $\phi'(u)$ induces a linear automorphism of the subspace $\phi(e_{U_0})\pind_0 (W)$. 
Moreover, for $l\in L_0$ we have that ${ulu^{-1}l^{-1}\in U\cap G_0 = U_0}$. Writing $ul=luu_0$ with $u_0\in U_0$, we observe that on $\phi(e_{U_0})\pind_0 (W)$ the operator $\phi'(u)$ commutes with $\phi(l)$ for every $l\in L_0$. Hence, by Schur's lemma,  the operator $\phi'(u)$ acts on $\phi(e_{U_0})\pind_0(W)$ as a non-zero scalar $a_u \in \C^\times$.      
The scalar $a_u$ depends only on the class of $u$ in the quotient $\ol{U}(\phi)=U(\phi)/U_0$, and so for each $x\in {\ol{U}(\phi)}$ we may define $a_x\coloneq a_{\tilde{x}}$, where $\tilde{x}\in U(\phi)$ is any lift of $x$. 

The image of the idempotent $e_{U(\phi)}e_\phi$ under $\theta\ot \phi'$ is therefore 
\[
\begin{split}
\theta\ot \phi' (e_{U(\phi)}e_\phi) &= \frac{1}{|U(\phi)|} \sum_{u\in U(\phi)} \theta(\delta_u)\otimes \phi'(u)\\ 
&= \frac{1}{|{\ol{U}(\phi)}|\cdot |U_0|} \sum_{\substack{x\in {\ol{U}(\phi)}\\  u_0\in U_0}} \delta_x \otimes \phi'(\tilde{x})\phi(u_0) \\ 
&= \frac{1}{|{\ol{U}(\phi)}|} \sum_{x\in {\ol{U}(\phi)}} \delta_x \otimes \phi'(\tilde{x})\phi(e_{U_0}) = e^{\phi'}_{\ol{U}(\phi)} \otimes \phi(e_{U_0}).
\end{split}
\]
 
From the fact that the element $e_{U(\phi)}e_\phi$ is an idempotent which commutes with the elements of $L(\psi)$ in $\H(G(\phi))$, and the fact that $\theta\ot \phi'$ is an algebra homomorphism, 
it follows immediately that $e_{\ol{U}(\phi)}^{\phi'}$ is an idempotent which commutes with the subalgebra $\H^{\gamma}({\ol{L}(\psi)})$.

Finally, the uniqueness of the scalars $a_x$ follows from the linear independence of the elements $\delta_x\otimes \phi(e_{U_0})$ in $\H^\gamma(\ol{G}(\phi))\otimes \End(\pind_0(W))$.
\end{proof}

\begin{remark}
{From the proof of {Lemma \ref{lem:C3_eX}} it follows that} the restriction of $\gamma$ to ${\ol{U}(\phi)}$ and to ${\ol{V}(\phi)}$ is cohomologous to the trivial two-cocycle.
Indeed, following the proof of the lemma, we see that for ${x_1,x_2\in {\ol{U}(\phi)}}$ {one has}  $a_{x_1}a_{x_2}\gamma(x_1,x_2) = a_{x_1x_2}$.
In other words, $a:{\ol{U}(\phi)}\to \C^\times$ is a coboundary which provides a trivialisation of the restriction of $\gamma$ to ${\ol{U}(\phi)}$. 
The same is true for ${\ol{V}(\phi)}$ and $b$. By changing the choice of $\phi'$ by a suitable  coboundary, {one} can therefore arrange that all the numbers $a_x$ and $b_y$ are 1.  {We shall continue to work with an arbitrary choice of $\phi'$ in what follows.}
\end{remark}

Lemma \ref{lem:C3_eX} implies that $\H^\gamma(\ol{G}(\phi))e_{\ol{U}(\phi)}^{\phi'} e_{\ol{V}(\phi)}^{\phi'}$ is an $\H^\gamma(\ol{G}(\phi))$-$\H^\gamma({\ol{L}(\psi)})$ bimodule. We denote by ${\pind_{{\ol{U}(\phi)},{\ol{V}(\phi)}}^{\phi'}:\Rep^\gamma({\ol{L}(\psi)})\to \Rep^\gamma(\ol{G}(\phi))}$ the functor of tensor product with this bimodule. Likewise, we denote by $\pres_{{\ol{U}(\phi)},{\ol{V}(\phi)}}^{\phi'}$ the functor of tensor product with the $\H^\gamma({\ol{L}(\psi)})$-$\H^\gamma(\ol{G}(\phi))$ bimodule $e_{\ol{U}(\phi)}^{\phi'} e_{\ol{V}(\phi)}^{\phi'}\H^\gamma(\ol{G}(\phi))$. The arguments of Theorem \ref{thm:pind-properties} carry over to this twisted setting, and show that the functors $\pind_{{\ol{U}(\phi)},{\ol{V}(\phi)}}^{\phi'}$ and $\pres_{{\ol{U}(\phi)},{\ol{V}(\phi)}}^{\phi'}$ are two-sided adjoints, and that up to natural isomorphism they do not depend on the order of ${\ol{U}(\phi)}$ and ${\ol{V}(\phi)}$.

Our induction functors are compatible with the final assertion \eqref{C3} of Clifford theory, as follows:

\begin{theorem}\label{thm:C3} Let $\phi'$ be a projective extension of $\phi$, with corresponding cocycle $\gamma^{-1}$, and let $\psi'$ be the projective extension of $\psi$ defined by \eqref{eq:psiprime_def}. The diagram
\begin{equation*} 
 \xymatrix@C=50pt{
\Rep(L(\psi))_{\psi} \ar[r]^-{\pind_{\psi}} & \Rep(G(\phi))_{\phi} \\
\Rep^{\gamma }({\ol{L}(\psi)}) \ar[u]_-{\cong}^-{\otimes\psi'} \ar[r]^-{\pind_{{\ol{U}(\phi)},{\ol{V}(\phi)}}^{\phi'}} & \Rep^{\gamma }(\ol{G}(\phi)) \ar[u]_-{\otimes \phi'}^-{\cong}
}
\end{equation*}
is commutative up to natural isomorphism.
\end{theorem}

\begin{example}\label{independence_example}
Suppose that the irreducible representation $\psi$ of $L_0$ satisfies $G(\pind_0\psi)=L(\psi)G_0$. We then have $\ol{G}(\phi)={\ol{L}(\psi)}$ in Theorem \ref{thm:C3}, and $\pind_{{\ol{U}(\phi)},{\ol{V}(\phi)}}^{\phi'}$ is the identity functor, and so we conclude from Theorems \ref{thm:C2} and \ref{thm:C3} that in this case
\[ 
\pind_{U,V}:\Rep(L)_\psi \to \Rep(G)_{\pind_0(\psi)}
\]
is an equivalence of categories. Specific examples of this kind arise in Section \ref{sec:sp}.
\end{example}

The proof of Theorem \ref{thm:C3} uses the following lemma, whose proof is a matter of straightforward linear algebra:

\begin{lemma}\label{lem:lin_alg}
Let $E$ be a finite dimensional vector space over $\C$, and let $S:E\to E$ be a linear endomorphism.
Let $A\subseteq \End_{\C}(E)$ be the centraliser of $S$ in $\End_{\C}(E)$. Then
{we have isomorphisms
\[
E\otimes \left(\Image(S)\right)^* \xrightarrow{e\otimes f\longmapsto [e'\mapsto ef(e')]} \Hom\left(\Image(S), E\right) \xrightarrow{T\mapsto T\circ S} \End(E)S
\]
of $\End(E)$-$A$-bimodules.}\hfill\qed
\end{lemma}
In our application the space $E$ will be $\pind_0(W)$, and the endomorphism $S$ will be the action of $e_{U_0}e_{V_0}e_\psi$.

\begin{proof}[Proof of Theorem \ref{thm:C3}]
Let
\[
T:\Mod\left( \H^\gamma({\ol{L}(\psi)})\otimes \End(W)\right) \to \Mod\left( \H^\gamma(\ol{G}(\phi))\otimes \End(\pind_0(W)\right)
\]
be the functor of tensor product with the bimodule 
$\H^\gamma(\ol{G}(\phi))e_{\ol{U}(\phi)}^{\phi'} e_{\ol{V}(\phi)}^{\phi'} \otimes_{\C} \left(\pind_0 (W)\otimes_{\C} W^*\right)$,
where $\pind_0(W)$ is viewed as a left $\End(\pind_0(W))$ module and $W^*$ is viewed as a right $\End(W)$-module in the obvious way.

We shall decompose the equivalences $\otimes\psi'$ and $\otimes\phi'$ into compositions of two equivalences, as in \eqref{eq:C3_decomposition}, and show both squares in the diagram
\begin{equation}\label{eq:C3_proof_diagram}
\xymatrix@C=60pt@R=40pt{
\Rep(L(\psi))_\psi \ar[r]^-{\pind_\psi} & \Rep(G(\phi))_\phi \\
\Mod\left( \H^\beta({\ol{L}(\psi)})\otimes \End(W)\right) \ar[u]_-{\cong}^-{(\theta\otimes\psi')^*} \ar[r]^-{T} &
\Mod \left( \H^\gamma(\ol{G}(\phi)) \otimes \End(\pind_0(W))\right)  \ar[u]^-{\cong}_-{(\theta\otimes\phi')^*} \\
\Rep^\beta({\ol{L}(\psi)}) \ar[u]_-{\cong}^-{\otimes W} \ar[r]^-{\pind_{{\ol{U}(\phi)},{\ol{V}(\phi)}}^{\phi'}} & \Rep^\gamma(\ol{G}(\phi)) \ar[u]^-{\cong}_-{\otimes \pind_0(W)} 
}
\end{equation}
commute. 

To show that the bottom square of \eqref{eq:C3_proof_diagram} commutes, let $M$ be an $\H^{\gamma}({\ol{L}(\psi)})$ module. We have natural isomorphisms of $\H^\gamma(\ol{G}(\phi))\otimes \End(\pind_0(W))$ modules 
\begin{align*}
T(M\otimes W)&= \left(\H^\gamma(\ol{G}(\phi))e_{{\ol{U}(\phi)}}^{\phi'}e_{{\ol{V}(\phi)}}^{\phi'} \otimes (\pind_0(W)\otimes W^*) \right)\otimes_{\H^\gamma({\ol{L}(\psi)})\otimes \End(W)} (M\otimes W) \\
& \cong \left(\H^\gamma(\ol{G}(\phi))e_{\ol{U}(\phi)}^{\phi'} e_{{\ol{V}(\phi)}}^{\phi'}\otimes_{\H^\gamma({\ol{L}(\psi)})}M\right) \otimes \left( \pind_0(W)\otimes (W^*\otimes_{\End(W)} W)\right) \\
& \cong \pind_{{\ol{U}(\phi)},{\ol{V}(\phi)}}^{\phi'}(M) \otimes \pind_0(W),
\end{align*}
because $W^*\otimes_{\End(W)} W\cong \C$, as $W$ is finite-dimensional. Thus the bottom square of the diagram commutes.

To show that the top square of \eqref{eq:C3_proof_diagram} commutes, it is enough to construct a linear isomorphism 
\[
F: e_\phi \H(G(\phi))e_{U(\phi)}e_{V(\phi)}e_\psi \longrightarrow \H^\gamma(\ol{G}(\phi))e_{\ol{U}(\phi)}^{\phi'} e_{\ol{V}(\phi)}^{\phi'} \otimes_{\C} \left(\pind_0 (W)\otimes_{\C} W^*\right)
\]
between the bimodules associated to the functors $\pind_\psi$ and $T$, satisfying 
\begin{equation}\label{eq:C3_proof_F}
F( f \cdot h \cdot k) = (\theta\otimes \phi')(f) \cdot F(h) \cdot (\theta\otimes\psi')(k)
\end{equation}
for all $f\in e_\phi\H(G(\phi))$, $h\in e_\phi\H(G(\phi))e_{U(\phi)}e_{V(\phi)}e_\psi$ and $k\in e_\psi\H(L(\psi))$. 

We shall construct $F$ as a composition $F=F_3 F_2  F_1$:
\begin{align*}
e_\phi \H(G(\phi))e_{U(\phi)}e_{V(\phi)}e_\psi & \xrightarrow[\cong]{F_1} 
\H^\gamma(\ol{G}(\phi))e_{\ol{U}(\phi)}^{\phi'}e_{{\ol{V}(\phi)}}^{\phi'} \otimes \End(\pind_0(W))\phi(e_{U_0}e_{V_0}e_\psi) \\
& \xrightarrow[\cong]{F_2} \H^\gamma(\ol{G}(\phi))e_{\ol{U}(\phi)}^{\phi'}e_{\ol{V}(\phi)}^{\phi'} \otimes \pind_0(W)\otimes \left( \phi(e_\psi e_{U_0}e_{V_0})\pind_0(W)\right)^* \\
& \xrightarrow[\cong]{F_3} \H^\gamma(\ol{G}(\phi)) e_{\ol{U}(\phi)}^{\phi'} e_{\ol{V}(\phi)}^{\phi'} \otimes \pind_0(W)\otimes W^*   ,
\end{align*}
where the isomorphisms $F_1$, $F_2$ and $F_3$ are defined below. 

The map $F_1$ is the restriction of the algebra isomorphism $\theta\otimes \phi'$ to the bimodule $e_\phi \H(G(\phi))e_{U(\phi)}e_{V(\phi)}e_\psi$. The image is as claimed because of Lemma \ref{lem:C3_eX}. By definition, $F_1$ satisfies 
\[
F_1(f\cdot h \cdot k) = (\theta\otimes\phi')(f) \cdot F_1(h) \cdot (\theta\otimes\phi')(k)
\]
(where $f$, $h$ and $k$ are as in \eqref{eq:C3_proof_F}).

The map $F_2$ is the identity on the first tensor factor, while on the second factor it is the isomorphism given by Lemma \ref{lem:lin_alg}, with $E=\pind_0(W)$ and $S=\phi(e_{U_0} e_{V_0} e_\psi)$. Clearly $F_2$ is a map of left $\H^\gamma(\ol{G}(\phi))\otimes \End(\pind_0(W))$ modules. Turning to the right module structure, fix $l\in L(\psi)$. The operator $\phi'(e_\psi\delta_l)\in \End(\pind_0(W))$ commutes with $\phi(e_\psi e_{U_0}e_{V_0})$, because $L(\psi)$ centralises $\psi$ and normalises $U_0$ and $V_0$. Therefore $\phi'(e_\psi\delta_l)$ lies in the algebra $A$ of Lemma \ref{lem:lin_alg}, and so the isomorphism $F_2$ satisfies   
\[
F_2 \left( F_1(h)\cdot \theta(l)\otimes\phi'(e_\psi\delta_l)\right) = F_2 F_1(h) \cdot \left(\theta(l)\otimes \phi'(e_\psi\delta_l)\right).
\]

The map $F_3$ is the identity on the first two tensor factors, while on the third factor it is the linear dual $\Theta^*$ of the isomorphism $\Theta: W\to \phi(e_{U_0}e_{V_0})\pind_0(W)$. (Note that $e_\psi$ acts as the identity on $W$.) Clearly $F_3$ is a map of left $\H^\gamma(\ol{G}(\phi))\otimes \End(\pind_0(W))$ modules. The isomorphism $\Theta$ satisfies 
$
\phi'(e_\psi\delta_l)\circ \Theta = \Theta\circ \psi'(e_\psi \delta_l),
$
by the definition \eqref{eq:psiprime_def} of $\psi'$, and we therefore have 
\[
F_3\left[ F_2 F_1(h) \cdot \left(\theta(l)\otimes\phi'(e_\psi\delta_l)\right) \right] = F(h)\cdot \left(\theta(l)\otimes\psi'(e_\psi\delta_l)\right).
\]
We have now shown that the isomorphism $F$ satisfies \eqref{eq:C3_proof_F}, and this completes the proof of Theorem \ref{thm:C3}.
\end{proof}

%%%%%%%%%%%%%%%%%%
%                                                   %
%   ORBIT METHOD SECTION    %
%                                                   %
%%%%%%%%%%%%%%%%%%

\section{The functor $\pind$ and the orbit method}\label{sec:orbit}

In this section we examine the induction functor $\pind_{U,V}$ in situations to which the orbit method applies, and show that it corresponds to a natural inclusion map on coadjoint orbits. We begin with an abstract formulation and then discuss a natural family of groups to which it applies, namely uniform {pro-$p$ groups and finite $p$-groups of nilpotency class less than $p$}. In particular, this family includes many compact open subgroups in reductive groups over $p$-adic fields.

\subsection{An abstract formulation}\label{subsec:abstract}

The orbit method in the context of profinite groups goes back to the work of Howe~\cite{Howe}. An abstract formulation was given by Boyarchenko and Sabitova in \cite{Boy-Sab}, and it is this latter point of view that we shall adopt here.

Let $G$ be a profinite group, let $\germ{g}$ be an abelian profinite group, and let $\exp:\germ{g}\to G$ be a homeomorphism satisfying
\begin{enumerate}[(A)]
\item The formula $\Ad_g(x)\coloneq \log(g\exp(x)g^{-1})$ for $g\in G$, $x\in \germ{g}$, and $\log=\exp^{-1}$ defines an action of $G$ on $\germ{g}$ by group automorphisms.
\item The pullback map $\exp^*:\H(G)^{G}\to \H(\germ{g})^{G}$, from the $\Ad_G$-invariant locally constant functions on $G$ to those on $\germ{g}$, is an isomorphism of convolution algebras.
\end{enumerate}
The adjoint action of $G$ on $\germ{g}$ induces a coadjoint action on the Pontryagin dual group~$\widehat{\germ{g}}$. It is shown in \cite[Theorem 1.1]{Boy-Sab} that for each irreducible smooth representation $\tau$ of $G$, with character $\ch_\tau\in \H(G)^G$, there is an $\Ad_G$-orbit $\Omega\subset\widehat{\germ{g}}$ such that
\begin{equation}\label{orbit_equation}
\exp^*(\ch_\tau) = |\Omega|^{-1/2}\sum_{\psi\in \Omega} \psi,
\end{equation}
and the map $\ch_\tau \mapsto \Omega$ sets up a bijection $\mathcal O_G:\Irr(G) \to G\backslash \widehat{\germ{g}}$ from the set of isomorphism classes of irreducible representations of $G$ to the set of coadjoint orbits in $\widehat{\germ g}$.

\begin{theorem}\label{orbit_theorem}  Let $G$, $\germ{g}$ and $\exp$ be as above. Let $U$, $L$ and $V$ be closed subgroups of $G$ such that:
\begin{enumerate}[\rm(1)]
\item $(U,L,V)$ is an Iwahori decomposition of $G$.
\item The preimages $\germ{u},\germ{l},\germ{v}$ of $U,L,V$ under $\exp$ are subgroups of $\germ{g}$, and $\germ g = \germ u \oplus \germ l \oplus \germ v$ as abelian groups.
\item The map $\exp:\germ{g}\to G$ restricts to homeomorphisms
\[
\germ{l}\to L,\quad \germ{l}\oplus\germ{u} \to LU\quad \text{and}\quad \germ{l}\oplus\germ{v}\to LV,
\]
each of which satisfies the conditions (A) and (B).
\end{enumerate}
{Then }the projection $\Lambda$ of $\germ{g}$ onto its summand $\germ{l}$ induces an injective map 
\[
\Lambda^*:L\backslash \widehat{\germ{l}} \to  G\backslash \widehat{\germ{g}}, \qquad \Omega \mapsto \Ad_G^*(\Omega\circ \Lambda)
\]
which makes the diagram
\[
 \xymatrix@C=40pt{ \Irr(L) \ar[r]^-{\pind_{U,V}} \ar[d]_-{{\mathcal O_L}} & \Irr(G) \ar[d]^-{{\mathcal O_G}}\\
L\backslash \widehat{\germ l}  \ar[r]^-{\Lambda^*}   & G\backslash \widehat{\germ g}   
 }
\]
commutative.
\end{theorem}

We require the following lemma:

\begin{lemma}\label{orbit_U_lemma}
Let $\Omega$ be an orbit in $LU\backslash\widehat{\germ{l}\germ{u}}$, corresponding via the orbit method to an irreducible representation $\tau$ of $LU$. Then $\tau$ is trivial on $U$ if and only if every $\psi\in \Omega$ is trivial on $\germ{u}$.
\end{lemma}

\begin{proof} 
If every $\psi\in \Omega$ is trivial on $\germ{u}$, then the character formula \eqref{orbit_equation} ensures that the character of $\tau$ is constant on $U$, and therefore that $\tau$ is trivial on $U$. Conversely, if $\tau$ is trivial on $U$ then its character is constant on $U$, with constant value $\dim(\tau)=|\Omega|^{1/2}$, and so \eqref{orbit_equation} implies that $\sum_{\psi\in \Omega}\psi(y)=|\Omega|$ for every $y\in \germ{u}$. Since each $\psi(y)$ is a complex number of modulus one, this equality forces $\psi(y)=1$ for every $\psi$ and every $y$.
\end{proof}

Note that by Theorem \ref{thm:pind-Iw}\eqref{item:pind-Iw-irr} the functor $\pind= \pind_{U,V}:\Rep(L)\to \Rep(G)$ preserves irreducibility, and therefore induces a map $\pind: \Irr(L)\to \Irr(G)$. We recall the following characterisation of the map~$\pind$ from Theorem~\ref{thm:pind-Iw}\eqref{item:pind-Iw-Hom}: given irreducible representations $\tau \in \Rep(G)$ and $\sigma\in \Rep(L)$, one has $\tau\cong \pind(\sigma)$ if and only if $\sigma$ is a common subrepresentation of $\tau^U$ and $\tau^V$.

\begin{proof}[Proof of Theorem \ref{orbit_theorem}]  Fix an orbit $\Omega\in L\backslash \widehat{\germ{l}}$ and let $\sigma={\mathcal O_L^{-1}(\Omega)\in \Irr(L)}$ be {the corresponding} irreducible representation of $L$. Let $\tau={\mathcal O_G^{-1}(\Lambda^*\Omega)\in \Irr(G)}$ be {the corresponding} irreducible representation {of $G$}. We will show that $\sigma$ is isomorphic to a subrepresentation of $\tau^U$. The same argument shows that $\sigma$ is isomorphic to a subrepresentation of $\tau^V$, and then Theorem~\ref{thm:pind-Iw}\eqref{item:pind-Iw-Hom} gives $\tau \cong \pind(\sigma)$ as required.

The characters $\ch_\rho$, where $\rho$ ranges over $\Irr(LU)$, constitute a linear basis for the space $\H(LU)^{LU}$. We let $P:\H(LU)^{LU}\to \H(LU)^{LU}$ be the projection
\[ P(\ch_\rho) = \begin{cases} \ch_\rho &\text{if $\rho$ is trivial on $U$} \\ 0 & \text{otherwise.}\end{cases}\]
On the other hand, the functions 
\[\chi_\Psi \coloneq |\Psi|^{-1/2}\sum_{\psi\in \Psi} \psi, \] as $\Psi$ ranges over $LU\backslash\widehat{\germ{l}\germ{u}}$, constitute a basis for $\H(\germ{l}\germ{u})^{LU}$, and we let $Q$ be the idempotent operator on $\H(\germ{l}\germ{u})^{LU}$ defined by
\[ Q(\chi_{\Psi}) = \begin{cases} \chi_{\Psi} & \text{if every $\psi\in \Psi$ is trivial on $\germ{u}$} \\ 0 & \text{otherwise}.\end{cases}\]

Lemma \ref{orbit_U_lemma} implies the commutativity of the middle square in the diagram 
\begin{equation}\label{orbit_proof_eq}
\xymatrix@R=30pt@C=50pt{
\H(G)^{G} \ar[r]^-{\text{restrict}} \ar[d]^-{\exp^*} & \H(LU)^{LU} \ar[r]^-{P} \ar[d]^-{\exp^*} & \H(LU)^{LU} \ar[r]^-{\text{restrict}} \ar[d]^-{\exp^*} & \H(L)^{L} \ar[d]^-{\exp^*} \\
\H(\germ{g})^{G} \ar[r]^-{\text{restrict}} & \H(\germ{l}\germ{u})^{LU} \ar[r]^-{Q} & \H(\germ{l}\germ{u})^{LU} \ar[r]^-{\text{restrict}} & \H(\germ{l})^L
}
\end{equation}
where \lq restrict\rq~ means restriction of functions. The two outer squares in the diagram obviously commute. For each irreducible representation $\rho$ of $G$, the composition along the top row of \eqref{orbit_proof_eq} sends the character of $\rho$ to the character of $\rho^U$.

Choose a point $\psi$ in the orbit $\Omega \subset  \widehat{\germ{l}}$, and write $\Lambda^*\Omega=\{{\psi\circ \Lambda},\phi_1,\ldots,\phi_n\}$. Since the character $\psi\circ\Lambda\in \widehat{\germ{g}}$ is trivial on $\germ{u}$, and has $\psi\circ\Lambda\restrict_{\germ{l}}=\psi$, we find that the composition along the bottom row of \eqref{orbit_proof_eq} sends $\chi_{\Lambda^*\Omega}=\exp^*(\ch_\tau)$ to the function 
\[ |\Lambda^*\Omega|^{-1/2} \left( \psi + \sum_{\phi_i\equiv 1\text{ on } \germ{u}} \phi_i\restrict_{\germ{l}}\right).\]
Since this sum contains $\psi$---and hence $\chi_\Omega=\exp^*(\ch_\sigma)$---with a positive coefficient, we conclude from the commutativity of \eqref{orbit_proof_eq} that $\tau^U$ contains a copy of $\sigma$.
\end{proof}

\subsection{Application to (pro-) $\boldsymbol p$-groups}
The results of the previous subsection apply to a rich and well-behaved family of (pro-) $p$-groups which we now discuss. Roughly speaking these groups admit good linearisations, that is, to each such group one may associate a Lie algebra that carries complete information on the group. 

\subsubsection*{Uniform pro-$p$-groups} A finite $p$-group  is called {\em powerful} if $[G,G] \subset G^p$ when $p$ is odd (and $[G,G] \subset G^4$ when $p=2$). Here $G^m$ is the group generated by $m$-powers. A pro-$p$ group is called powerful if it is the inverse limit of finite powerful groups. A pro-$p$ group is called {\em uniform} if it is powerful, finitely generated (as a pro-$p$ group), and torsion-free. To each uniform pro-$p$ group $G$ one may associate a {\em uniform} {$\Z_p$-Lie algebra} $\germ g = \Lie(G)$, that is, a  $\Z_p$-Lie algebra which is free of finite rank as a $\Z_p$-module, and which satisfies $[\germ g, \germ g]_{\Lie} \subset p\germ g$ for $p$ odd (and $[\germ g, \germ g]_{\Lie} \subset 4\germ g$ for $p=2$); see~\cite{DDMS} for a comprehensive treatment. This association defines an equivalence of categories between the category of uniform pro-$p$ groups and uniform $\Z_p$-Lie algebras. Starting with  a   uniform  Lie  algebra $\germ g$ this association is made concrete using the Campbell-Hausdorff  series 
\[
H(u,v)=\log\left(\exp(u)\exp(v)\right) = u+v+(\textrm{Lie brackets})  \in \Q \langle \!\langle u,v \rangle \!\rangle, 
\]
which is expressible in terms of $u,v\in \germ g$ by means of the Lie bracket, and which allows one to define a uniform pro-$p$ group $G$ having the same underlying set  as $\germ g$ and group operation $u \cdot v = H(u,v)$. Let $\exp$ denote the identity map on $\germ g$, thought of as a map from the Lie algebra $\germ g$ to the group $G$. This map is well-behaved with respect to passage to subgroups or quotients: Lie subalgebras of $\germ g$ correspond bijectively to closed uniform subgroups of $G$, and ideals in $\germ g$ correspond to normal subgroups in $G$; see \cite[Section 4.5]{DDMS}. Moreover, it is shown in \cite[Theorem~2.6]{Boy-Sab}     that for $p\neq 2$  this map $\exp$ satisfies the conditions (A) and (B) from Section \ref{subsec:abstract}, meaning that the orbit method applies and gives a bijection $\mathcal O_G:\Irr(G)\to G\backslash \widehat{g}$. {This generalises an earlier result of Howe \cite[Theorem 1.1]{Howe}.} 

\subsubsection*{Finite $p$-groups of nilpotency class less than $p$} There is a similar Lie-type correspondence for finite $p$-groups of nilpotency class less than $p$. To each group $G$ of this type one may associate a finite $\Z$-Lie algebra $\germ g=\Lie(G)$ which is nilpotent of class less than $p$, and whose additive group is a $p$-group, such that $G$ is isomorphic to the group $\exp(\germ g)$ whose underlying set is $\germ g$ and whose multiplication is given by the Campbell-Hausdorff series (which is finite, in this case); see \cite[Section 10.2]{Khukhro}. If $p$ is odd then the orbit method applies to the map $\exp:\germ g\to G$; see \cite[Theorem~2.6]{Boy-Sab}. 

\subsubsection*{Application of Theorem \ref{orbit_theorem}} For the rest of this section  let $G=\exp(\germ g)$ be either a uniform pro-$p$ group  or a finite $p$-group of nilpotency class less than $p$, with corresponding Lie algebra $\germ g$. For each subalgebra $\germ h$ of $\germ g$ we write $H$ for the corresponding subgroup $\exp(\germ h)$ of $G$.

\begin{definition}\label{def:Lie_Iwahori}
An \emph{Iwahori decomposition} of $\germ g$ is a triple of Lie subalgebras $(\germ u, \germ l, \germ v)$ of $\germ g$ such that   $[\germ l,\germ u]\subseteq \germ u$,   $[\germ l, \germ v]\subseteq \germ v$, and such that $\germ g = \germ u \oplus \germ l \oplus v$ as $\Z_p$-modules (in the uniform pro-$p$ case) or as $\Z$-modules (in the finite $p$-group case).
\end{definition}

\begin{lemma}\label{lem:Lie_Iwahori} 
If $(\germ u, \germ l, \germ v)$ is an Iwahori decomposition of $\germ g$, then $(U,L,V)$ is an Iwahori decomposition of $G$. 
\end{lemma}

\begin{proof} 
The Lie correspondence ensures that $U$, $L$ and $V$ are closed subgroups of $G$ such that $L$ normalises $U$ and $V$. The subgroups $V$ and $B\coloneq UL$   have trivial intersection in $G$, because the subalgebras $\germ v$ and $\germ u\oplus \germ l$  have trivial intersection in $\germ g$, and so the product map $U\times L \times V\to G$ is injective. We shall now show that this map is surjective.

 We must show that for each $x\in \germ b\coloneq \germ u\oplus \germ l$ and each $y\in \germ v$ one has $\exp(x+y)\in BV$. The Campbell-Hausdorff formula implies that $\exp(x+y) =  \exp(x)\exp(z_1)\exp(y)$ for some $z_1\in \germ g_1=[\germ g,\germ g]$. Writing $z_1=x_1+y_1$, where $x_1\in \germ b$ and $y_1\in \germ v$, another application of Campbell-Hausdorff gives $\exp(z_1)=\exp(x_1)\exp(z_2)\exp(y_1)$ for some $z_2\in \germ g_2=[\germ g,\germ g_1]$. Continuing in this way   we find $z_n \in {\germ g}_n = [\germ g,\germ g_{n-1}]$, $x_{n-1} \in \germ b$ and $y_{n-1} \in \germ v$, for every $n \in \N$, such that  $\exp(z_{n-1})=\exp(x_{n-1})\exp(z_n)\exp(y_{n-1})$, and we deduce that   
\[ 
\exp(x+ y  ) \in \bigcap_{n\geq 0} B\exp(\germ g_n) V = B\left(\bigcap_{n\geq 0} \exp(\germ g_n)\right) V = BV,
\]
where the first equality holds because  the groups $\exp(\germ g_n)$ form a descending chain and  $G$ is compact, and the second holds because $\germ g$ is either uniform or nilpotent. 

We are left to verify condition (2) of Definition \ref{def:vI}. If $G$ is finite this condition is trivially satisfied, so suppose that $G$ is a uniform pro-$p$ group. For each $n\geq 0$ the triple $(p^n\germ u, p^n\germ l, p^n\germ v)$ is an Iwahori decomposition of the {ideal} $p^n\germ g$ of $\germ g$, and so the above argument shows that the open {normal} subgroups $K_n = \exp(p^n\germ g)$ of $G$ satisfy condition (2).
\end{proof}

We now have the following corollary of Theorem~\ref{orbit_theorem}:

\begin{corollary} \label{cor:Lie_Iwahori}
Let $p$ be an odd prime. Let $G$ be either a uniform pro-$p$ group, or a finite $p$-group of nilpotency class less than $p$. Let $(\germ u, \germ l, \germ v)$ be an Iwahori decomposition of the Lie algebra $\germ g$ of $G$, and let $(U,L,V)$ be the corresponding Iwahori decomposition of $G$. The diagram
\[
 \xymatrix@R=30pt@C=50pt{ \Irr(L) \ar[r]^-{\pind_{U,V}} \ar[d]_-{{\mathcal O_L}} & \Irr(G) \ar[d]^-{{\mathcal O_G}}\\
L\backslash \widehat{\germ l}  \ar[r]^-{\Lambda^*}   & G\backslash \widehat{\germ g}  
 }
\]
is commutative.
\end{corollary}

\begin{proof}
This follows from Theorem \ref{orbit_theorem}. The hypothesis (1) of that theorem is satisfied because of Lemma \ref{lem:Lie_Iwahori}; hypothesis (2) is satisfied by assumption; and the hypothesis (3) is satisfied because of \cite[Theorem 2.6]{Boy-Sab}.
\end{proof}

We remark that for uniform pro-$2$ groups the orbit method does not fully apply, though one has weaker versions; see~\cite{Jaikin, Boy-Sab}.

\begin{example}\label{ex:orbit_tidy} 
In \lq real life\rq~one may find a rich supply of  groups to which the corollary may be applied. {Let $\mathcal G$ be a $p$-adic Lie group, let $\alpha$ be an automorphism of $\mathcal G$, and denote by  $\alpha_*$ the derived automorphism of the Lie algebra $\germ g$ {of $G$}. Then
\[
\germ u_\alpha  \coloneq \{x\in \germ g\ |\ \alpha_*^n(x)\to 0\text{ as }n\to\infty\}\quad \text{and}\quad \germ v_\alpha\coloneq \germ u_{\alpha^{-1}}
\]
are nilpotent Lie subalgebras of $\germ g$, normalised by the subalgebra
\[ 
\germ l_\alpha  \coloneq \{ x\in \germ g\ |\ \{\alpha_*^n(x)\ |\ n\in \Z\}\text{ is precompact in }\germ g\},
\]
and we have $\germ g = \germ u_\alpha\oplus \germ l_\alpha\oplus \germ v_\alpha$ as $\Q_p$-vector spaces. Moreover, $\germ u_\alpha$, $\germ l_\alpha$ and $\germ v_\alpha$ are the respective Lie algebras of the closed subgroups $\mathcal U_\alpha$, $\mathcal L_\alpha$ and $\mathcal V_\alpha$ of $\mathcal G$, where we are using the notation of Example \ref{ex:vI-tidy}. These assertions are proved in \cite[Theorem 3.5]{Wang}. It is shown in \cite[Lemma 3.3]{Gloeckner_manuscripta} that $\germ g$ contains arbitrary small open uniform $\Z_p$-Lie subalgebras $\germ k$ having $\germ k = (\germ u_\alpha\cap \germ k)\oplus  (\germ l_\alpha \cap \germ k) \oplus (\germ v_\alpha \cap \germ k)$. The compact open subgroups $K= \exp(\germ k)$ of $\mathcal G$ then have Iwahori decompositions $(\mathcal U_\alpha\cap K, \mathcal L_\alpha\cap K, \mathcal V_\alpha\cap K)$. If $p$ is odd, Corollary \ref{cor:Lie_Iwahori} describes the induction functor $\pind_{\mathcal U_\alpha\cap K, \mathcal V_\alpha\cap K}:\Rep(\mathcal L_\alpha\cap K)\to \Rep(K)$ in terms of the orbit method and of the projection $\germ k\to \germ l_\alpha\cap \germ k$.
  }
\end{example}

\begin{example} 
For a finite example, let $\O$ be a compact discrete valuation ring with maximal ideal $\germ p$ and  residue characteristic $p$. Let $K=K_1$ be the first principal congruence subgroup  in  $\GL_n(\O/\germ p^{\ell})$, for $\ell>1$. Then $K$ is a finite $p$-group of nilpotency class $\ell-1$, with Lie algebra $\germ k=M_n(\germ p/\germ p^\ell)$.
As explained in Example \ref{ex:vI-GLnO}, each partition $n=n_1+\cdots+n_m$ gives an Iwahori decomposition $(U\cap K, L\cap K, V\cap K)$ of $K$, corresponding to the decomposition of  $\germ k$ into block-upper-triangular, block-diagonal and block-lower-triangular matrices. If~${p>\ell-1}$ and odd, Corollary \ref{cor:Lie_Iwahori} gives a description of the resulting induction functor $\pind_{U\cap K, V\cap K}:\Rep(L\cap K)\to \Rep(K)$ in terms of the orbit method and the projection of  $\germ k$  onto its subalgebra of block-diagonal matrices. 
\end{example}

\begin{remark} We have taken the point of view of the theory of uniform groups due to its fairly concrete and algebraic formulation. Historically, the Lie correspondence {for (pro-)p-groups} goes back to the seminal work of Lazard~\cite{Lazard-finite}, \cite{Lazard}. {The technique of obtaining Iwahori decompositions of groups from decompositions of Lie algebras is well known in the setting of $p$-adic reductive groups: see \cite{WorksInProgress} and \cite{Deligne_support}, for example.}
\end{remark}

%%%%%%%%%%%%
%                                %
%      SP_4                 %
%                                %
%%%%%%%%%%%%

\newpage

\section{Case study: Siegel Levi subgroup in $\Sp_4(\O_2)$}\label{sec:sp}
 
Let $\O$ be a compact discrete valuation ring with maximal ideal $\mathfrak{p}$, a fixed uniformiser $\pi$ and finite residue field~$\k$ of odd characteristic. Let $\O_\ell:=\O/\mathfrak{p}^\ell$. In this section we illustrate how the results of the previous sections may be applied to study the representations of the symplectic group $\Sp_4(\O_2)$ that are induced, in the sense of Definition \ref{def:pind}, from the Siegel Levi subgroup of $2\times 2$ block-diagonal   matrices. {Note that this is equivalent to studying {those} induced representations of $\Sp_4(\O)$ which factor through $\Sp_4(\O_2)$: see Theorem \ref{thm:pind-properties}\eqref{item:pind-finite}.} The main results in this section are a double-coset formula, \`a la Mackey, for the composition of induction and restriction for these groups (Theorem~\ref{thm:sp_Mackey}); and an answer to a question of Dat regarding parahoric induction (Corollary~\ref{Dat_corollary}).

Let us introduce the notation used to state the Mackey formula. Let 
\[
G= \Sp_4(\O_2) = \{g\in \GL_4(\O_2) \ |\ g^{\transpose} j g = j\}, \quad\text{where}\quad j=\left[\begin{smallmatrix} & & -1 & \\ & & & -1 \\ 1 & & & \\ & 1 & &\end{smallmatrix}\right].
\]
This group admits a virtual Iwahori decomposition $(U,L,V)$, with
 \[ 
 \begin{split}
  L &= \left. \left\{ \begin{bmatrix} a & 0 \\ 0 & a^{-\transpose} \end{bmatrix} \ \right|\ a\in \GL_2(\O_2)\right\}, \quad \\
  U &= \left. \left\{ \begin{bmatrix} 1 & m \\ & 1 \end{bmatrix} \ \right|\ m\in M_2(\O_2),\ m=m^{\transpose} \right\},\quad \text{and}\quad V =U^\transpose,
 \end{split}
 \]
where $(\cdot)^{\transpose}$ means transpose and $(\cdot)^{-\transpose}$ means transpose inverse. We consider the associated functors 
\[
\pind_L^G\coloneq \pind_{U,V}:\Rep(L)\to \Rep(G) \qquad \text{and}\qquad \pres^G_L\coloneq \pres_{U,V}:\Rep(G)\to \Rep(L).
\]

The subgroup $L\cong\GL_2(\O_2)$   has a virtual Iwahori decomposition $(U',D,V')$, where 
\[
\begin{split}
D &= \{ \diag(\alpha,\delta,\alpha^{-1},\delta^{-1})\ |\ \alpha,\delta\in \O_2^\times\},\quad \\
U'&= \left.\left\{ \diag\left( \left[\begin{smallmatrix} 1 & \beta  \\ & 1 \end{smallmatrix}\right], \left[\begin{smallmatrix} 1 & \\ -\beta& 1\end{smallmatrix}\right]\right) \ \right|\ \beta \in \O_2\right\} \quad\text{and}\quad V'= (U')^\transpose.
\end{split}
\]
We consider the associated functors 
\[
\pind_D^L\coloneq \pind_{U',V'}:\Rep(D)\to \Rep(L) \qquad \text{and}\qquad \pres^L_D\coloneq \pres_{U',V'}:\Rep(L)\to \Rep(D).
\]

{We let
\[
W_G \coloneq N_G(D)/D \qquad \text{and}\qquad W_L\coloneq N_L(D)/D 
\]
denote the  Weyl groups of $D$ in $G$ and in $L$, respectively}. We write $\Ad_g$ for the conjugation action of a group on itself and subsets thereof, and with a slight abuse of notation also for the corresponding action on representations. 

\begin{theorem}\label{thm:sp_Mackey}
There is a natural isomorphism of functors $\Rep(L)\to \Rep(L)$,
\[
\pres^G_L \pind_L^G \cong \bigoplus_{g\in W_L\backslash W_G/W_L} \pind_{gLg^{-1}\cap L}^L \, \Ad_g \, \pres^L_{L\cap g^{-1}Lg} .
\]
\end{theorem}

\begin{remarks}\label{rem:sp_Mackey} Let us unpack Theorem \ref{thm:sp_Mackey} a {little}.
\begin{enumerate}[(1)]
\item The right-hand side of the formula in Theorem \ref{thm:sp_Mackey} is a sum over a set of representatives $g\in N_G(D)$ for the double cosets of $W_L$ in $W_G$; the resulting functor does not depend on the choices made, up to natural isomorphism. 
\item For each $g\in N_G(D)$, the intersection $gLg^{-1}\cap L$ is either $L$ or $D$. The functors $\pind_D^L$ and $\pres^L_D$ were defined above; the functors $\pind_L^L$ and $\pres^L_L$ are, by definition, the identity functors on $\Rep(L)$.
\item The group $W_L$ is the two-element group generated  (modulo $D$)  by the matrix \[
t{=} \diag\left( \sigma,\sigma \right) \in L,\qquad \text{where}\qquad \sigma=\left[\begin{smallmatrix} & -1 \\ 1& \end{smallmatrix}\right]\in \GL_2(\O_2).
\]   The eight-element dihedral group $W_G$ is generated (modulo $D$) by $W_L$ together with the matrix
\[
w\coloneq  \begin{bmatrix} {0}  &   & -1 &   \\  & 1 &   &   \\ 1 &   & {0} &   \\   &   &   & 1 \end{bmatrix}.
\]
Defining \[
s\coloneq \begin{bmatrix} & \sigma \\ \sigma^{-1} & \end{bmatrix} \in G,
\]
we have the double-coset decomposition 
\[
W_G = W_L \sqcup W_L s W_L \sqcup W_L w W_L.
\]
The element $s$ normalises $L$, while   $wLw^{-1}\cap L = D$. Putting all of this together, the formula in Theorem~\ref{thm:sp_Mackey} takes the following more explicit form:
\begin{equation*} 
\pres^G_L \pind_L^G \cong \id \oplus \Ad_s \oplus \pind_D^L \Ad_w \pres^L_D.
\end{equation*} 
\item Note that the definition of the functors $\pind_L^G$ and $\pres^G_L$, and the statement of Theorem \ref{thm:sp_Mackey}, continue to make sense when $\O_2$ is replaced by $\O_\ell$, or indeed by any  finite (or profinite) commutative ring. Over $\O_1$, the formula is valid: as explained in Example \ref{ex:pind-field}, the functors $\pind_L^G$, $\pres^G_L$, $\pind_D^L$ and $\pres^L_D$ are isomorphic in that case to Harish-Chandra functors, and the formula in Theorem~\ref{thm:sp_Mackey} is an instance of the well-known formula~{\eqref{eq:Mackey_intro}} for the composition of these functors (cf. \cite[Theorem 5.1]{DM}).  We do not know whether the formula in Theorem~\ref{thm:sp_Mackey} is valid for $\Sp_4$ over more general rings; the proof presented below relies on some very special features of~$\O_2$.
\end{enumerate}
\end{remarks}

Our strategy for proving Theorem \ref{thm:sp_Mackey} is as follows. Reduction modulo $\pi$ gives rise to a surjective group homomorphism $G= \Sp_4(\O_2) \to \Sp_4(\k)$, whose kernel  is an abelian group isomorphic to the Lie algebra~$\germ{sp}_4(\k)$. In Sections \ref{subsec:Sp4_congruence} and \ref{subsec:Sp4_Clifford} we apply the orbit method and Clifford theory to reduce  Theorem \ref{thm:sp_Mackey} to a statement about orbits and representations of stabilisers for the adjoint action of $\Sp_4(\k)$ on  $\germ{sp}_4(\k)$. In Sections~\ref{subsec:Sp4_centralisers} and~\ref{subsec:Sp4_Mackey_proof} we verify the theorem through a case-by-case analysis of the orbits (with some details  postponed to Appendix~\ref{appendix}). 

For the semisimple orbits  our induction and restriction functors correspond to Harish-Chandra induction and restriction for (reductive) subgroups of $\Sp_4(\k)$, and our Mackey formula  follows from the well-known Mackey formula \eqref{eq:Mackey_intro} for Harish-Chandra functors.  The computation for the non-semisimple orbits---and in particular, for the one nilpotent orbit that is relevant here---is more subtle. In Corollary~\ref{Dat_corollary} we shall see that it is precisely this nilpotent orbit that witnesses the difference between our induction functor and Dat's parahoric induction.

\subsection{The congruence subgroup}\label{subsec:Sp4_congruence}

Let $G_0$ denote the kernel of the reduction map $\Sp_4(\O_2) \to \Sp_4(\k)$, and let 
 \[
 \germ{g}=\germ{sp}_4(\k ) = \{ y\in  M_4(\k )\ |\ jy + y^{\transpose} j =0\},
 \] 
 viewed as an additive abelian group on which $G$ acts via the adjoint action of its quotient $\Sp_4(\k)$. To {reduce} the notational load we shall write
 \begin{equation*}
g \cdot y = \Ad_g(y)=gyg^{-1}~~\text{(modulo $\pi$)}, \quad g \in G, y \in \germ g.
\end{equation*}

\begin{lemma}\label{lem:Sp4_exp} 
 The map $\exp:\germ{g}\to G_0$ defined as the composition
\[ 
\germ{g} \xrightarrow{y\mapsto \pi y} \pi \germ{sp}_4(\O_2) \xrightarrow{z \mapsto 1+z} G_0
\]
is a $G$-equivariant group isomorphism.
\end{lemma}

\begin{proof} Clear.
\end{proof}
 
For every subgroup $H$ of $G$ we set  
\[
H_0 \coloneq H \cap G_0,\quad \overline{H}_{\phantom{0}} \coloneq HG_0/G_0 \cong H/H_0,\quad \text{and}\quad 
\germ{h} \coloneq \log(H_0),
\]
where $\log:G_0\to \germ{g}$ denotes the inverse to $\exp$. In particular, $\germ l$ is the additive subgroup of $M_4(\k)$ consisting of the block-diagonal matrices $\diag(x,-x^{\transpose})$, for $x\in M_2(\k)$.

It is easily checked that the triple of subgroups $(\germ u, \germ l, \germ v)$ forms an Iwahori decomposition of $\germ g$, and it follows that the triple $(U_0, L_0, V_0)$ is an Iwahori decomposition of $G_0$. Similarly, $(\germ u', \germ d, \germ v')$ is an Iwahori decomposition of $\germ l$, and so $(U'_0, D_0, V_0')$ is an Iwahori decomposition of $L_0$.

\begin{lemma}\label{lem:Sp4_pairing}
Choose and fix a nontrivial character $\zeta:\k\to \C^\times$. For each $y\in \germ g$, denote by $\phi_y:G_0\to \C^\times$ the character 
\[
\phi_y :g \mapsto \zeta\circ \trace\left( \log (g ) y\right). 
\]
The mapping $y\mapsto \phi_y$ is a $G$-equivariant bijection $\germ g\xrightarrow{\cong} \Irr(G_0)$, which restricts to an $L$-equivariant bijection $\germ l \xrightarrow{\cong} \Irr(L_0)$, and to a $D$-equivariant bijection $\germ d\xrightarrow{\cong} \Irr(D_0)$.
\end{lemma}

\begin{proof} 
Let $\langle z,y\rangle\coloneq \zeta\circ\trace(zy)$ for $z,y\in M_4(\k)$. It is well-known that the map $M_4(\k)\to \widehat{M_4(\k)}$ sending $y$ to $\langle \cdot, y\rangle$ is an isomorphism. By Pontryagin duality  this map restricts to an isomorphism between $\germ g$ and the dual of $M_4(\k)/\germ{g}^\perp$, where $\germ{g}^\perp=\{z\in M_4(\k)\ |\ \langle z,\germ g\rangle =1\}$. Let $\germ{g}'=\{z\in M_4(\k)\ |\ jz-z^\transpose j=0\}$. For each $z\in \germ{g}'$ and $y\in \germ g$ we have 
$\trace(zy)=\trace(\Ad_j(z)\Ad_j(y)) = -\trace(zy)$,
showing that $\germ{g}'\subseteq \germ{g}^\perp$. We also have $M_4(\k)=\germ{g}\oplus \germ{g}'$ (this is the eigenspace decomposition for the involution $y\mapsto \Ad_j(y^\transpose)$), and since $\germ{g}$ and its dual $M_4(\k)/\germ{g}^\perp$ have the same cardinality we must have $\germ{g}'=\germ{g}^\perp$. Thus the pairing $\langle\cdot,\cdot\rangle$ restricts to an isomorphism $\germ g\to \widehat{\germ g}$.  Composing with the isomorphism $\widehat{\log}:\widehat{\germ{g}}\to \widehat{G_0}=\Irr(G_0)$ shows that $y\mapsto \phi_y$ is an isomorphism $\germ g\to \Irr(G_0)$.  The $G$-equivariance of this map   follows from the invariance of the trace. Similar arguments apply to $\germ l$ and $\germ d$.
\end{proof}

Theorem \ref{orbit_theorem}, applied to this particularly simple setting, gives the following identification of the induction maps
\[
\pind_0 \coloneq \pind_{U_0,V_0}:\Irr(L_0)\to \Irr(G_0)\quad \text{and}\quad \pind_0'\coloneq \pind_{U_0',V_0'}:\Irr(D_0)\to \Irr(L_0).
\]

\begin{lemma}\label{lem:Sp4_orbit}
The diagram 
\[
\xymatrix@R=30pt@C=50pt{
\Irr(D_0) \ar[r]^-{\pind_0'} 
 & \Irr(L_0) \ar[r]^-{\pind_0} & \Irr(G_0) \\
\germ d \ar[r]^-{\mathrm{inclusion}} \ar[u]^-{y\mapsto \phi_y}_-{\cong} & \germ l \ar[r]^-{\mathrm{inclusion}} \ar[u]^-{y\mapsto \phi_y}_-{\cong} & \germ g \ar[u]^-{y\mapsto \phi_y}_-{\cong}
}
\]
is commutative.
\end{lemma}

\begin{proof}
In view of Theorem \ref{orbit_theorem} and Lemma \ref{lem:Sp4_exp}, it is enough to observe that the diagram 
\[
\xymatrix@R=30pt@C=60pt{ 
\widehat{\germ d} \ar[r]^-{(\Lambda')^*} & \widehat{\germ l} \ar[r]^-{\Lambda^*} & \widehat{\germ g} \\
\germ d \ar[u]^-{y\mapsto \langle \cdot, y\rangle}_-{\cong}\ar[r]^-{\mathrm{inclusion}} &
\germ l \ar[u]^-{y\mapsto \langle\cdot, y\rangle}_-{\cong}  \ar[r]^-{\mathrm{inclusion}} & \germ g \ar[u]^-{y\mapsto \langle \cdot, y\rangle } _-{\cong}
}
\]
commutes, where $\Lambda$ is the projection of $\germ g = \germ u \oplus \germ l \oplus \germ v$ onto its summand $\germ l$, and $\Lambda'$ is the projection of $\germ l$ onto its summand $\germ d$. 
\end{proof}

\subsection{Application of Clifford theory}\label{subsec:Sp4_Clifford}  

We shall use Theorems \ref{thm:C1}, \ref{thm:C2} and \ref{thm:C3} to transport the functors  $\pind_L^G$ and $\pres_L^G$ to the setting of (projective) representations of the centralisers $\overline{L}(y)\subseteq \GL_2(\k)$ and $\overline{G}(y)\subseteq \Sp_4(\k)$ associated to the characters $\phi_y$. 

The first assertion \eqref{C1} of Clifford theory decomposes the categories $\Rep(D)$, $\Rep(L)$ and $\Rep(G)$ as products over the sets $D\backslash \Irr(D_0)$, $L\backslash \Irr(L_0)$ and $G\backslash \Irr(G_0)$, respectively. For each $y\in \germ g$, let $\phi_y$ be the character in $\Irr(G_0)$ defined in Lemma \ref{lem:Sp4_pairing}. We denote by 
\[
E^G_{y}: \Rep(G)\to \Rep(G)_{\phi_y} 
\]
the projection  onto the subcategory associated to (the $G$-orbit of) the character $\phi_y$. We similarly define $E^L_y$ and $E^D_y$, for $y\in \germ l$ and $y\in\germ d$  respectively.

For each $y\in \germ l$ we write
\[
G(y,\germ l) \coloneq \{g\in G\ |\ g\cdot y\in \germ l\}
\]
for the set of elements in $G$ which conjugate $y$ back into $\germ l$. Notice that $G(y,\germ l)$ is stable under left multiplication by $L$, and under right multiplication by $G(y)$.

The first step is to show that we may deal with ordinary, as opposed to projective, representations of the centralisers.  

\begin{lemma}\label{lem:Sp4_extension}
There is a family of maps $(\phi'_y)_{y\in \germ l}$ with the following properties:
\begin{enumerate}[\rm(1)]
\item $\phi'_y$ is a one-dimensional (ordinary) representation of the centraliser $G(y)$ that extends $\phi_y$.
\item For each $g\in G(y,\germ l)$ one has $\Ad_g(\phi'_y)=\phi'_{g\cdot y}$.
\item $\phi'_y(g)=1$ for all $g\in U(y)\cup V(y)$.
\item If $y\in \germ d$ then $\phi'_y(g)=1$ for all $g\in U'(y)\cup V'(y)$.
\end{enumerate}
\end{lemma}

\begin{proof}
For each $y\in \germ l \subset M_4(\k)$, let $H(y)$ denote the centraliser of $y$ in the group $\GL_4(\O_2)$ (which acts on $M_4(\k)$  through the adjoint action of its quotient $\GL_4(\k)$). Singla showed in \cite[Proposition 2.2]{Singla_GLn}  that the character $\phi_y$ extends to a linear character of $H(y)$. If $\phi'_{y}$ is such an extension, then for each $g\in G(y,\germ l)$ the character $\Ad_g(\phi'(y))$ is an extension of $\phi_{g\cdot y}$ to $H(g\cdot y)$.  Moreover, if $g\in G(y)$ then $\Ad_g(\phi'_y)=\phi'_y$. We may thus choose a family of characters $\phi'_y$ satisfying (1) and (2) by fixing one $y$ in each $G$-orbit, choosing an extension $\phi'_y$ as above, and then defining $\phi'_{g\cdot y}\coloneq \Ad_g(\phi'_y)$ for each $g\in G(y,\germ l)$. 

We will prove that the characters $\phi'_y$ constructed above are  trivial on $U(y)$ and $V(y)$ by showing that these two groups belong to the commutator subgroup of $H(y)$. Indeed, let $m\in M_2(\O_2)$ be any matrix such that the $4\times 4$ matrix 
$
u= \left[\begin{smallmatrix} 1 & m  \\   & 1\end{smallmatrix} \right]
$
lies in $H(y)$. Then the matrices 
$
u' = \left[\begin{smallmatrix} 1 & m/2 \\    & 1\end{smallmatrix}\right]$ and $z = \left[\begin{smallmatrix} 1 &   \\   & -1\end{smallmatrix}\right]
$
also lie in $H(y)$, and we have $u=[u',z]$. This shows  that $U(y)$ lies in $[H(y),H(y)]$, and a similar argument applies to $V(y)$. Thus the family $\phi'_y$ constructed above satisfies condition (3).

Finally, if $y\in \germ d$, then a similar argument to the above shows that $U'(y)$ and $V'(y)$ belong to the commutator subgroup of the centraliser of $y$ inside the block-diagonal subgroup $\diag(\GL_2(\O_2),\GL_2(\O_2))\subset \GL_4(\O_2)$, and so property (4) is also satisfied.
\end{proof}

For the rest of {Section \ref{sec:sp}} we fix a family of characters $\phi'_y$ as in Lemma \ref{lem:Sp4_extension}. {As explained in Section \ref{subsec:Clifford_review},} Clifford theory gives equivalences of categories
\[
F^L_y : \Rep(\ol{L}(y))\xrightarrow{\otimes\phi'_y} \Rep(L(y))_{\phi_y} \xrightarrow{\ind_{L(y)}^L} \Rep(L)_{\phi_y} 
\]
and
\[
F^G_y : \Rep(\ol{G}(y))\xrightarrow{\otimes\phi'_y} \Rep(G(y))_{\phi_y} \xrightarrow{\ind_{G(y)}^G} \Rep(G)_{\phi_y} .
\]

\begin{lemma}\label{lem:Sp4_FG} 
For each $y\in \germ l$,  each $g\in G(y,\germ l)$ and each $h\in N_G(L)$, the diagrams
\[
\xymatrix@R=30pt@C=50pt{
\Rep(G)_{\phi_y} \ar[r]^-{\id} & \Rep(G)_{\phi_{g\cdot y}} \\
\Rep(\ol{G}(y)) \ar[u]^-{F^G_y} \ar[r]^-{\Ad_{ {g}}} & \Rep(\ol{G}(g\cdot y))\ar[u]_-{F^G_{g\cdot y}}
} \qquad \text{and}\qquad 
\xymatrix@R=30pt@C=50pt{
\Rep(L)_{\phi_y} \ar[r]^-{\Ad_h} & \Rep(L)_{\phi_{h\cdot y}} \\
\Rep(\ol{L}(y)) \ar[u]^-{F^L_y} \ar[r]^-{\Ad_{ {h}}} & \Rep(\ol{L}(h\cdot y))\ar[u]_-{F^L_{h\cdot y}}
} 
\]
commute  up to natural isomorphism.
\end{lemma}

\begin{proof}
The commutativity of the second diagram follows from  property (2) of Lemma \ref{lem:Sp4_extension}, and from the well-known fact that $\Ad_h\circ \ind_{L(y)}^L \cong \ind_{L(h\cdot y)}^L \circ \Ad_h$. The commutativity of the first diagram follows from the same argument, plus the fact that $\Ad_g$ is isomorphic to the identity functor on $\Rep(G)$ for every $g\in G$.  
\end{proof}

For each $y\in \germ l$  we consider the functors
\[
\pind_{\overline{L}(y)}^{\ol{G}(y)} \coloneq \pind_{\ol{U}(y),\ol{V}(y)}:\Rep(\ol{L}(y))\to \Rep(\ol{G}(y)) \quad \text{and}\quad \pres^{\ol{G}(y)}_{\ol{L}(y)}: \Rep(\ol{G}(y))\to \Rep(\ol{L}(y)).
\]

\begin{lemma}\label{lem:Sp4_C3_LG}
For each $y\in \germ l$ the  diagrams
\[
\xymatrix@R=30pt@C=50pt{
\Rep(L)_{\phi_y} \ar[r]^-{\pind_L^G E^L_y} & \Rep(G)_{\phi_y} \\
\Rep(\overline{L}(y)) \ar[r]^-{\pind_{\ol{L}(y)}^{\ol{G}(y)}} \ar[u]^-{F^L_y}_-{\cong} & \Rep(\ol{G}(y))\ar[u]_-{F^G_y}^-{\cong}
} 
\qquad \text{and}\qquad 
\xymatrix@R=30pt@C=50pt{
\Rep(G)_{\phi_y} \ar[r]^-{E^L_y\pres_L^G} & \Rep(L)_{\phi_y} \\
\Rep(\overline{G}(y)) \ar[r]^-{\pres_{\ol{L}(y)}^{\ol{G}(y)}} \ar[u]^-{F^G_y}_-{\cong} & \Rep(\ol{L}(y))\ar[u]_-{F^L_y}^-{\cong} 
} 
\]
 commute up to natural isomorphism.
\end{lemma}

\begin{proof}
The fact that $\phi'_y$ is trivial on the subgroups $U(y)$ and $V(y)$ ensures that the functions $a$ and $b$ of Lemma \ref{lem:C3_eX} are {identically} equal to $1$, and thus that the functor $\pind_{\ol{U}(y),\ol{V}(y)}^{\phi'_y}$ appearing in Theorem \ref{thm:C3} is equal to $\pind_{\ol{L}(y)}^{\ol{G}(y)}$. This proves the commutativity of the $\pind$-diagram; taking adjoints proves the commutativity of the $\pres$-diagram.
\end{proof}

Combining Lemmas \ref{lem:Sp4_FG} and \ref{lem:Sp4_C3_LG} gives immediately:

\begin{lemma}\label{lem:Sp4_C3_yz}
For each $y\in \germ l$ and each $g\in G(y,\germ l)$  the diagram
\[
\xymatrix@R=30pt@C=50pt{
\Rep(L)_{\phi_y} \ar[r]^-{\pind_L^G E^L_y} & \Rep(G)_{\phi_y} \ar[r]^-{\id} & \Rep(G)_{\phi_{g\cdot y}} \ar[r]^{E^L_{g\cdot y}\pres^G_L} & \Rep(L)_{\phi_{g\cdot y}} \\
\Rep(\ol{L}(y)) \ar[u]^-{F^L_y} \ar[r]^-{\pind_{\ol{L}(y)}^{\ol{G}(y)}} & \Rep(\ol{G}(y)) \ar[u]^-{F^G_y} \ar[r]^-{\Ad_{ {g}}} & \Rep(\ol{G}(g\cdot y)) \ar[u]^-{F^G_{g\cdot y}} \ar[r]^-{\pres_{\ol{L}(g\cdot y)}^{\ol{G}(g\cdot y)}} & \Rep(\ol{L}(g\cdot y)) \ar[u]^-{F^L_{g\cdot y}}
}
\]
commutes up to natural isomorphism.
\hfill\qed
\end{lemma}

Now we use Clifford theory to analyse the right-hand side $\id\oplus \Ad_s \oplus \pind_{D}^L\Ad_w\pres^L_D$ of the Mackey formula (cf. Remarks \ref{rem:sp_Mackey}(3)).   For each pair of elements $y,z\in \germ l$, define a functor 
\[
\Delta(z,y):\Rep(\overline{L}(y)) \to \Rep(\overline{L}(z)),\qquad \Delta(z,y)=\begin{cases} \Ad_{ {l}} & \text{if }z=l\cdot y \\ 0 & \text{if }z\not\in L\cdot y.\end{cases}
\]
 Note that $\Delta(z,l)$ is well-defined up to natural isomorphism, because $\Ad_l\cong \id$ on $\Rep(\ol{L}(y))$ for every $l\in L(y)$.

\begin{lemma}\label{lem:Sp4_RHS_1}
For each $y,z\in \germ l$ the diagrams  
\[
\xymatrix@R=30pt@C=70pt{
\Rep(L)_{\phi_y} \ar[r]^-{E^L_z  E^L_y} & \Rep(L)_{\phi_z} \\
\Rep(\ol{L}(y)) \ar[r]^-{\Delta(z,y)} \ar[u]^-{F^L_y} & \Rep(\ol{L}(z))\ar[u]_-{F^L_z}
}\qquad \text{and}\qquad 
\xymatrix@R=30pt@C=70pt{
\Rep(L)_{\phi_y} \ar[r]^-{E^L_z \Ad_s E^L_y} & \Rep(L)_{\phi_z} \\
\Rep(\ol{L}(y)) \ar[r]^-{\Delta(z,s\cdot y)\Ad_s} \ar[u]^-{F^L_y} & \Rep(\ol{L}(z))\ar[u]_-{F^L_z}
}
\]
commute up to natural isomorphism.
\end{lemma}

\begin{proof}
The commutativity of the first diagram follows from Lemma \ref{lem:Sp4_FG}, and from the fact that $E_z^L E_y^L=0$ unless $y$ and $z$ are $L$-conjugate. The commutativity of the second diagram follows from a similar argument, plus the equality $\Ad_s E_y^L = E_{s\cdot y}^L \Ad_s$.
\end{proof}

The analysis of the functors $\pind_D^L$ and $\pres^L_D$ follows the above analysis of $\pind_L^G$ and $\pres^G_L$. Because $D$ is abelian we have $D(y)=D$ for each $y\in \germ d$. Clifford theory gives an equivalence of categories 
\[
F^D_y: \Rep(\ol{D}) \xrightarrow{\otimes\phi'_y} \Rep(D)_{\phi_y}, 
\]
such that for each $h\in N_G(D)$ the diagram 
\[
\xymatrix@R=30pt@C=50pt{
\Rep(D)_{\phi_y} \ar[r]^-{\Ad_h} & \Rep(D)_{\phi_{h\cdot y}} \\
\Rep(\ol{D}) \ar[r]^-{\Ad_{ {h}}} \ar[u]^-{F^D_y} & \Rep(\ol{D}) \ar[u]_-{F^D_{h\cdot y}}
}
\]
commutes up to natural isomorphism.

We consider the functors 
\[
\pind_{\ol{D}}^{\ol{L}(y)} \coloneq \pind_{\ol{U_0'}(y),\ol{V_0'}(y)}:\Rep(\ol{D})\to \Rep(\ol{L}(y)) \qquad \text{and}\qquad \pres^{\ol{L}(y)}_{\ol{D}}\coloneq \pres_{\ol{U_0'}(y),\ol{V_0'}(y)}:\Rep(\ol{L}(y)) \to \Rep(\ol{D}).
\] 

For each $y,z\in \germ l$, we define the functor 
\[
\Xi(z,  y): \Rep(\ol{L}(y))\to \Rep(\ol{L}(z)) 
\]
as the direct sum, over $d\in \germ d$, of the compositions 
\[
\Rep(\ol{L}(y)) \xrightarrow{\Delta(d,y)} \Rep(\ol{L}(d)) \xrightarrow{\pres^{\ol{L}(d)}_{\ol{D}}} \Rep(\ol{D}) \xrightarrow{\Ad_w} \Rep(\ol{D}) \xrightarrow{\pind_{\ol{D}}^{\ol{L}(w\cdot d)}} \Rep(\ol{L}(w\cdot d)) \xrightarrow{\Delta(z, w\cdot d)} \Rep(\ol{L}(z)).
\]

\begin{lemma}\label{lem:Sp4_RHS_2}
For each $y,z\in \germ l$ the diagram 
\[
\xymatrix@R=30pt@C=80pt{
\Rep(L)_{\phi_y}\ar[r]^-{E_z^L \left( \pind_D^L \Ad_w \pres^L_D\right) E_y^L} & \Rep(L)_{\phi_z} \\
\Rep(\ol{L}(y)) \ar[r]^-{ \Xi(z,  y)} \ar[u]^-{F^L_y} &
\Rep(\ol{L}(z)) \ar[u]_-{F^L_z}
}
\]
commutes up to natural isomorphism.  
\end{lemma}

\begin{proof}
Decomposing the category $\Rep(D)$ over $\Irr(D_0)\cong \germ d$, we have 
\[
\pind_D^L \Ad_w \pres^L_D = \bigoplus_{d\in \germ d} \pind_D^L \Ad_w E_d^D \pres^L_D = \bigoplus_{d\in \germ d} \pind_D^L E_{w\cdot d}^D \Ad_w E_d^D \pres^L_D = \bigoplus_{d\in \germ d} E_{w\cdot d}^L \pind_D^L E_{w\cdot d}^D \Ad_w E_d^D \pres^L_D E_d^L 
\]
where in the last equality we have used Theorem \ref{thm:C1}. After applying Lemma \ref{lem:Sp4_RHS_1} to the functors $E_y^L E_{w\cdot d}^L$ and $E_d^L E_z^L$, we are left to prove the commutativity of 
\[
\xymatrix@R=30pt@C=50pt{
\Rep(L)_{\phi_d} \ar[r]^-{E^D_d \pres^L_D} & \Rep(D)_{\phi_d} \ar[r]^-{\Ad_w} & \Rep(D)_{\phi_{w\cdot d}} \ar[r]^-{\pind_D^L} & \Rep(L)_{\phi_{w\cdot d}} \\
\Rep(\ol{L}(d)) \ar[r]^-{\pres^{\ol{L}(d)}_{\ol{D}}} \ar[u]^-{F^L_d} & \Rep(\ol{D}) \ar[r]^-{\Ad_w} \ar[u]^-{F^D_d} & \Rep(\ol{D})\ar[r]^-{\pind_{\ol{D}}^{\ol{L}(w\cdot d)}} \ar[u]^-{F^D_{w\cdot d}} & \Rep(\ol{L}(w\cdot d)) \ar[u]^-{F^L_{w\cdot d}}
}
\] 
for each $d\in \germ d$. This follows from Theorem \ref{thm:C3}, just as in Lemma \ref{lem:Sp4_C3_yz}.
\end{proof}

The end result of our Clifford analysis is as follows:

\begin{corollary}\label{cor:Sp4_Clifford}
Theorem \ref{thm:sp_Mackey} is equivalent to the assertion that for every $y\in \germ l$ and every $g\in G(y,\germ l)$, there is a natural isomorphism
\[
\pres^{\ol{G}(g\cdot y)}_{\ol{L}(g\cdot y)}\, \Ad_{ {g}} \, \pind_{\ol{L}(y)}^{\ol{G}(y)} \cong \Delta(g\cdot y,y) \bigoplus \Delta(g\cdot y, s\cdot y)\Ad_s \bigoplus  \Xi(g\cdot y,  y)
\]
of functors $\Rep(\ol{L}(y))\to \Rep(\ol{L}(g\cdot y))$.
\end{corollary}

\begin{proof}
By Lemma \ref{lem:Sp4_pairing} and \eqref{C1} we have 
\[
\pres^G_L \pind_L^G = \bigoplus_{L\cdot y,\ L\cdot z\, \in L\backslash \germ l} E^L_z \pres^G_L \pind_L^G E^L_y.
\]
Theorem \ref{thm:C1} implies that 
\[
E^L_z \pres^G_L \pind_L^G E^L_y = E^L_z \pres^G_L E^G_z E^G_y \pind_L^G E^L_y,
\]
and $E^G_z E^G_y=0$ unless $z$ and $y$ lie in the same $G$-orbit. This proves that
\begin{equation}\label{eq:Sp4_Clifford_cor1}
\pres^G_L\pind_L^G = \bigoplus_{\substack{L\cdot y\in L\backslash \germ l, \\ g\in L\backslash G(y,\germ l)/G(y)}} E^L_{g\cdot y} \pres^G_L \pind_L^G E^L_y.
\end{equation}

Clifford theory likewise gives a decomposition 
\[
\id\oplus \Ad_s \oplus \pind_D^L \Ad_w \pres^L_D =  \bigoplus_{L\cdot y,\ L\cdot z\, \in L\backslash \germ l} \left( E^L_{z}E^L_y \oplus E^L_z\Ad_s E^L_y \oplus E^L_z\pind_D^L \Ad_w \pres^L_D E^L_y\right),
\]
and Lemmas \ref{lem:Sp4_RHS_1} and \ref{lem:Sp4_RHS_2} imply that each term in the sum vanishes if $z$ and $y$ are not $G$-conjugate. This proves that 
\begin{equation}\label{eq:Sp4_Clifford_cor2}
\id\oplus \Ad_s \oplus \pind_D^L \Ad_w \pres^L_D = \bigoplus_{\substack{L\cdot y\in L\backslash \germ l, \\ g\in L\backslash G(y,\germ l)/G(y)}}  \left(E^L_{g\cdot y}E^L_y \oplus E^L_{g\cdot y} \Ad_s E^L_y \oplus E^L_{g\cdot y} \pind_D^L \Ad_w \pres^L_D E^L_y\right)  .
\end{equation}

Thus the reformulation  of Theorem \ref{thm:sp_Mackey} given in Remarks \ref{rem:sp_Mackey}(3) is equivalent to the existence of a natural isomorphism, for each $y\in \germ l$ and each $g\in G(y,\germ l)$, between the $(y,g)$ terms on the right-hand sides of \eqref{eq:Sp4_Clifford_cor1} and \eqref{eq:Sp4_Clifford_cor2}. Conjugating each of these terms by the equivalences $F^L_{y}$ and $F^L_{g\cdot y}$, and applying Lemmas \ref{lem:Sp4_C3_yz}, \ref{lem:Sp4_RHS_1} and \ref{lem:Sp4_RHS_2}, we bring Theorem \ref{thm:sp_Mackey} into the asserted form.
\end{proof}

\subsection{Centralisers}\label{subsec:Sp4_centralisers} 

We now present the facts about the centralisers $\overline{G}(y)$ and $\overline{L}(y)$ and   about the orbit  spaces   $L\backslash G(y,\germ l)/G(y)$ that will be needed for the proof of Theorem \ref{thm:sp_Mackey}. More details, and the proofs of the assertions made here, are given in Appendix~\ref{appendix}.

Fix $x\in M_2(\k)$, and let $y=\diag(x,-x^\transpose)$ be the corresponding element of $\germ l$. We divide our analysis according to the Jordan normal form of $x$. Up to conjugacy by $\ol{L}\cong \GL_2(\k)$, the following nine cases   exhaust all of the possibilities. In the following $s$, $t$, $w$ and $\sigma$ are as in Remarks \ref{rem:sp_Mackey}. We shall write  \lq $L\backslash G(y,\germ l)/G(y)=\{g,h,k\}$\rq~  to mean that $G(y,\germ l) = L g G(y) \sqcup L h G(y) \sqcup Lk G(y)$.
 
\medskip

\noindent {\bf Case  1:} $x=\diag(\mu,\mu)$, $\mu\in \k$. In this case $\overline{L}(y)=\overline{L} \cong \GL_2(\k)$. There are two subcases:

\smallskip

\noindent {\bf  1A:}  $\mu \ne 0$. Here $\overline{G}(y)=\overline{L}(y)$, and $L\backslash G(y,\germ l)/G(y)=\{1,s,w\}$. 

\smallskip

\noindent{\bf 1B:} ${\mu=0}$. Here $\ol{G}(y)=\ol{G}$ and $L\backslash G(y,\germ l)/G(y)=\{1\}$.

\medskip

\noindent {\bf Case  2:} $x=\diag(\mu,\nu)$, $\mu \ne \nu$. In this case $\overline{L}(y)=\overline{D}$. There are three subcases:

\smallskip
\noindent{\bf  2A:} {$\mu \neq \pm \nu$, $\mu\neq 0\neq \nu$.} Here $\overline{G}(y)=\overline{L}(y)$, and $L\backslash G(y,\germ l)/G(y) = \{1,s,w,wt\}$. 

\smallskip
\noindent {\bf 2A$\!^\star$:} {$\nu= 0$.} Here $\overline{G}(y)$ is a   reductive group over $\k$; $\ol{L}(y)=\ol{D}$ is a rational maximal torus, whose Weyl group in $\ol{G}(y)$ is  generated (modulo $\ol{D}$) by the involution $t^{-1}w t$; and $\ol{U}(y)$ and $\ol{V}(y)$ are the unipotent radicals of an opposite pair of rational Borel subgroups of $\ol{G}(y)$ containing $\ol{L}(y)$.  We have $L\backslash \ol{G}(y)/G(y)=\{1,w\}$.
 
 \smallskip
\noindent{\bf 2B:} {${\mu =-\nu}$.} In this case we have $\ol{G}(y)=\Ad_w(\overline{L})$, $\ol{U}(y) = \Ad_w\left( \ol{V'}  \right)$ and $\ol{V}(y) = \Ad_w\left( \ol{U'}  \right)$, while $L\backslash G(y,\germ l)/G(y)=\{1,w,wt\}$.

\medskip

\noindent {\bf Case 3:}  $x=\big[\begin{smallmatrix} \alpha & \beta \\  \mu\beta &  \alpha\end{smallmatrix}\big], \mu \in \k$ non-square, $\alpha\in \k$, $\beta\in \k^\times$. In this case $\k_2\coloneq M_2(\k)(x)$ is a quadratic field extension of $\k$, and $\overline{L}(y)=\{\diag(a,a^{-\transpose})\ | \ a\in \k_2^\times \} \cong \k_2^\times$. There are two subcases.

\smallskip
\noindent{\bf 3A:} {${\alpha\neq 0}$.}  We have 
$\ol{G}(y)=\overline{L}(y)$ and $L\backslash G(y,\germ l)/G(y) = \{1,s\}$. 

\smallskip
\noindent{\bf 3B:} {$\alpha=0$.} Here $\ol{G}(y)$ is a reductive group over $\k$; $\ol{L}(y)$ is a rational maximal torus of $\ol{G}(y)$ whose Weyl group is generated by  $s$; and $\ol{U}(y)$ and $\ol{V}(y)$ are the unipotent radicals of an opposite pair of rational Borel subgroups of $\ol{G}(y)$ containing $\ol{L}(y)$. In this case $L\backslash G(y,\germ l)/G(y)=\{1\}$.

\medskip

\noindent{\bf Case 4:}  ${x=\left[\begin{smallmatrix} \mu & 1 \\ & \mu \end{smallmatrix}\right]}$. In this case $\overline{L}(y)=\{\diag(a,a^{-\transpose})\ |\ a\in \k[x]\} \cong \GL_1(\k[\epsilon]/(\epsilon^2))$. There are two subcases.

\smallskip
\noindent{\bf  4A:} {${\mu\neq 0}$.} Here $\overline{G}(y)=\overline{L}(y)$ and $L\backslash G(y,\germ l)/G(y) = \{1,s\}$. 

 \smallskip
\noindent{\bf  4B:} {$\mu =0$.} The subgroups $\ol{U}(y)$ and $\ol{V}(y)$ commute with one another in $\ol{G}(y)$, and we have 
\[
\ol{G}(y) = \left( \ol{U}(y)\times \ol{V}(y) \right) \rtimes \left( \ol{L}(y)\rtimes S\right) 
\]
where $S$ is the two-element group generated by $s$. We have $L\backslash G(y,\germ l)/G(y)=\{1\}$.

\subsection{Proof of Theorem \ref{thm:sp_Mackey}}\label{subsec:Sp4_Mackey_proof}

In this section we shall use the results of the previous section to prove that for each $y\in \germ l$ and each $g\in G(y,\germ l)$, there is a natural isomorphism 
\begin{equation}\label{eq:sp_Mackey_proof}
\pres^{\ol{G}(g\cdot y)}_{\ol{L}(g\cdot y)}\, \Ad_{ {g}} \, \pind_{\ol{L}(y)}^{\ol{G}(y)} \cong \Delta(g\cdot y,y) \bigoplus \Delta(g\cdot y, s\cdot y)\Ad_s \bigoplus  \Xi(g\cdot y,   y)
\end{equation}
of functors $\Rep(\ol{L}(y))\to \Rep(\ol{L}(g\cdot y)$. By Corollary \ref{cor:Sp4_Clifford}, this constitutes a proof of Theorem \ref{thm:sp_Mackey}. We recall from Section \ref{subsec:Sp4_Clifford} that 
\[
\Delta(z,y)=\begin{cases} \Ad_{ {l}} & \text{if }z=l\cdot y \\ 0 & \text{if }z\not\in L\cdot y \end{cases} \qquad \text{and}\qquad 
\Xi(z,y) = \bigoplus_{d\in \germ d} \left(\Delta(z,w\cdot d) \pind_{\ol{D}}^{\ol{L}(w\cdot d)} \Ad_w \pres^{\ol{L}(d)}_{\ol{D}} \Delta(d,y)\right)
\]
for all $y,z\in \germ l$.

  The proof of \eqref{eq:sp_Mackey_proof} goes through a case-by-case analysis of the various possibilities for $y=\diag(x,-x^\transpose)$. The cases are labelled as in Section \ref{subsec:Sp4_centralisers}. The reader who is more interested in ideas than in details might like to focus on cases  1A ,  2A$\!^\star$  and  4B, which together contain all of the techniques used in the other cases.

\medskip
\noindent{\bf Case 1A:} Take $y=\diag(\mu,\mu,-\mu,-\mu)$, $\mu\neq 0$. We must consider $g=1$, $g=s$ and $g=w$. 

For $g=1$, the left-hand side of \eqref{eq:sp_Mackey_proof} is the identity on $\Rep(\ol{L})$, because all of the centralisers are equal to $\ol{L}$. We have $\Delta(y,s\cdot y)=0$ because $y$ and $s\cdot y=-y$ are not $L$-conjugate. The only diagonal matrix $d\in \germ d$ that is $L$-conjugate to $y$ is $d=y$ itself, and we have $\Delta(y,w\cdot y)=0$, and so $\Xi(y,y)=0$. Thus the only nonzero term on the right-hand side of \eqref{eq:sp_Mackey_proof} is $\Delta(y,y)$, which is the identity on $\Rep(\ol{L})$. Thus the two sides of \eqref{eq:sp_Mackey_proof} are isomorphic.

For $g=s$ all of the centralisers are again equal to $\overline{L}$, and so the left-hand side of \eqref{eq:sp_Mackey_proof} equals $\Ad_s$. One finds as above that the only nonzero term on the right-hand side is $\Delta(s\cdot y, s\cdot y)\Ad_s$, which equals  $\Ad_s$.

For $g=w$ we have $w\cdot y = \diag(-\mu,\mu,\mu,-\mu)$, and so the centralisers of $g\cdot y$ are as in case  2B. The left-hand side of \eqref{eq:sp_Mackey_proof} is thus equal to $\pres^{\ol{G}(w\cdot y)}_{\ol{D}} \Ad_w$. Since $\Delta(w\cdot y, y)$ and $\Delta(w\cdot y, s\cdot y)$ are both zero, 
the only potentially nonzero term on the right-hand side of \eqref{eq:sp_Mackey_proof} is $\Xi(w\cdot y, y)$. Since the only diagonal matrix that is $L$-conjugate to $y$ is $y$ itself, we have 
\[
\Xi(w\cdot y,  y) = \Delta(w\cdot y, w\cdot y) \pind_{\ol{D}}^{\ol{L}(w\cdot y)} \Ad_w \pres^{\ol{L}(y)}_{\ol{D}} \Delta(y,y) =  \Ad_w \pres^{\ol{L}}_{\ol{D}} .
\]
Now, we have $\ol{G}(w\cdot y)=\Ad_w(\ol{L})$, $\ol{U}(w\cdot y)=\Ad_w(\ol{V'})$, and $\ol{V}(w\cdot y) = \Ad_w(\ol{U'})$ (see case 2B in Section \ref{subsec:Sp4_centralisers}), 
and therefore
\[
\pres^{\ol{G}(w\cdot y)}_{\ol{D}} \Ad_w =   \pres_{\ol{U}(w\cdot y),\ol{V}(w\cdot y)} \Ad_w \cong \Ad_w \pres_{\ol{V'},\ol{U'}} \cong \Ad_w\pres_{\ol{U'},\ol{V'}} = \Ad_w \pres^{\ol{L}}_{\ol{D}}
\]
where we used Theorem \ref{thm:pind-properties}\eqref{item:pind-UV-VU} to switch $\ol{U'}$ and $\ol{V'}$. This completes the proof of \eqref{eq:sp_Mackey_proof} in case  1A.

\medskip
\noindent{\bf Case 2A:} Take $y=\diag(\mu ,\nu,-\mu ,-\nu)$, where $\mu$ and $\nu$ are nonzero and $\mu \neq \pm \nu$. We must consider $g=1$, $g=s$, $g=w$ and $g=wt$. For each of these $g$ the matrix $g\cdot y$ is again of the form  2A, and so all of the centralisers appearing in \eqref{eq:sp_Mackey_proof} are equal to $\overline{D}$, and the left-hand side of \eqref{eq:sp_Mackey_proof} is equal to the functor $\Ad_g$ on~$\Rep(\overline{D})$. 

For $g=1$ the functor $\Delta( y, y)$ equals the identity, while $\Delta(y, s\cdot y)=0$ (because $y$ and $s\cdot y$ are not $L$-conjugate) and $\Xi(y,y)=0$ (because the only diagonal matrices that are $L$-conjugate to $y$ are $y$ and $t\cdot y$, and neither of these is $L$-conjugate to $w\cdot y$). So both sides of \eqref{eq:sp_Mackey_proof} equal the identity.  

For $g=s$ the functor  $\Delta(s\cdot y, s\cdot y)$ is the identity, while $\Delta(s\cdot y,y)$ and $\Xi(s\cdot y, s\cdot y)$ are both zero. So both sides of \eqref{eq:sp_Mackey_proof} equal $\Ad_s$. 

For $g=w$, the only potentially nonzero term on the right-hand side of \eqref{eq:sp_Mackey_proof} is $\Xi(w\cdot y,  y)$. There are two diagonal matrices $d\in \germ d$ that are $L$-conjugate to $y$, namely $y$ itself and $t\cdot y$. Since $w\cdot y=\diag(-\mu,\nu,\mu,\-\nu)$ and $wt\cdot y=\diag(-\nu,\mu,\nu,-\mu)$ are not $L$-conjugate, we have $\Delta(w\cdot y, wt\cdot y)=0$, and so the summand in $\Xi(w\cdot y, w\cdot y)$ corresponding to $d=t\cdot y$ is equal to zero. Therefore  
\[
\Xi(w\cdot y,  y) = \Delta(w\cdot y, w\cdot y) \pind_{\ol{D}}^{\ol{L}(w\cdot y)} \Ad_w \pres^{\ol{L}(y)}_{\ol{D}} \Delta(y,y) = \Ad_w
\] 
as required.

For $g=wt$  the argument of the previous paragraph shows that the right-hand side of \eqref{eq:sp_Mackey_proof} is equal to $\Xi(wt\cdot y, y)$, and that only the $d=t\cdot y$ summand in the latter is nonzero. We have 
\[
 \Xi(wt\cdot y, y)   = \Delta(wt\cdot y,wt\cdot y) \pind_{\ol{D}}^{\ol{L}(wt\cdot y)} \Ad_w \pres^{\ol{L}(t\cdot y)}_{\ol{D}} \Delta(t\cdot y,y)  
  =  \Ad_w \Ad_t 
 \]
 because $\Delta(t\cdot y,y)=\Ad_t$ and all of the centralisers equal $\ol{D}$. This completes the proof of \eqref{eq:sp_Mackey_proof} in case  2A.

\medskip
\noindent{\bf Case 2A$\!^\star$:} Let $y=\diag(\mu,0,-\mu,0)$ where $\mu\neq 0$. We must consider $g=1$ and $g=w$.

For $g=1$, the left-hand side of \eqref{eq:sp_Mackey_proof} equals $\pres^{\ol{G}(y)}_{\ol{D}} \pind_{\ol{D}}^{\ol{G}(y)}$. We are in the situation of Example \ref{ex:pind-field}, and so $\pres^{\ol{G}(y)}_{\ol{D}}$ and $\pind_{\ol{D}}^{\ol{G}(y)}$ are isomorphic  to the functors of Harish-Chandra restriction and induction (respectively) for the maximal torus $\ol{D}\subset \ol{G}(y)$. Since the Weyl group of $\ol{D}$ in $\ol{G}(y)$ is equal to $\{1,t^{-1}wt\}$, the usual Mackey formula \eqref{eq:Mackey_intro} (cf. \cite[Theorem 5.1]{DM}) for the composition of Harish-Chandra functors   gives  
\[
\pres^{\ol{G}(y)}_{\ol{D}} \pind_{\ol{D}}^{\ol{G}(y)} \cong \id \oplus \Ad_{t^{-1}w t}.
\]

Still taking $g=1$, we have $\Delta(g\cdot y,s\cdot y)=0$, and so  the right-hand side of \eqref{eq:sp_Mackey_proof} equals $\id \oplus \Xi(y,y)$. The only diagonal matrices that are $L$-conjugate to $y$ are $d=y$ and $d=t\cdot y$. For $d=y$ we have $\Delta(y,w\cdot y)=0$, and so the only potentially nonzero summand in $\Xi(y,y)$ is the one corresponding to $d=t\cdot y$. 
Computing this summand, we find 
\[
\Xi(y,y)=\Delta(y,wt\cdot y) \pind_{\ol{D}}^{\ol{L}(wt\cdot y)} \Ad_w \pres^{\ol{L}(t\cdot y)}_{\ol{D}} \Delta(t\cdot y, y) = \Ad_{t^{-1}} \Ad_w \Ad_t,
\]
because $\ol{L}(wt\cdot y)=\ol{L}(t\cdot y)=\ol{D}$. Thus the right-hand side of \eqref{eq:sp_Mackey_proof} is, like the left-hand side, isomorphic to $\id\oplus \Ad_{t^{-1}wt}$.

Now take $g=w$. Notice that $w\cdot y = -y$. The left-hand side of \eqref{eq:sp_Mackey_proof} is 
\[
\pres^{\ol{G}(y)}_{\ol{D}} \Ad_w \pind_{\ol{D}}^{\ol{G}(y)}  \cong \Ad_w \pres^{\ol{G}(y)}_{\ol{D}} \pind_{\ol{D}}^{\ol{G}(y)} \cong \Ad_w \oplus \Ad_{ts},
\]
where for the first isomorphism we have used Theorem \ref{thm:pind-properties}\eqref{item:pind-UV-VU}, and for the second we have used the Mackey formula \eqref{eq:Mackey_intro} for Harish-Chandra induction together with the equality $wt^{-1}wt=ts$ in $G$.

Keeping $g=w$ and turning to the right-hand side of \eqref{eq:sp_Mackey_proof}, the term $\Delta(w\cdot y,y)$ vanishes, while the fact that $w\cdot y = ts\cdot y$ implies that $\Delta(w\cdot y, s\cdot y)\Ad_s=\Ad_{ts}$. So we are left to show that $\Xi(w\cdot y,y)=\Ad_w$. The $d=y$ term in $\Xi(w\cdot y, y)$ is equal to 
\[
\Delta(w\cdot y, w\cdot y) \pind_{\ol{D}}^{\ol{D}} \Ad_w \pres^{\ol{D}}_{\ol{D}} \Delta(y,y) = \Ad_w, 
\]
while the $d=t\cdot y$ term vanishes because $\Delta(w\cdot y, t\cdot y)=0$. Thus both sides of \eqref{eq:sp_Mackey_proof} are isomorphic to $\Ad_w\oplus \Ad_{ts}$ in this case.

\medskip
\noindent{\bf Case 3A:}  Take $x=\left[\begin{smallmatrix}   \alpha & \beta \\ \beta\mu & \alpha \end{smallmatrix}\right]$, where ${\mu\in \k}$ is a non-square and $\alpha,\beta\in \k^\times$, and let $y=\diag(x,-x^\transpose)$.  We must consider $g=1$ and $g=s$. We have $\ol{G}(y)=\ol{G}(s\cdot y) = \ol{L}(s\cdot y)=\ol{L}(y)$, so that the left-hand side of \eqref{eq:sp_Mackey_proof} is equal to $\Ad_g$ for each $g$. Note that since $y$ is not $L$-conjugate to a diagonal matrix we have $\Xi(z,y)=0$ for every $z$. 
 
For $g=1$ we have $\Delta(y,y)=\id$ while $\Delta(y,s\cdot y)=0$, so both sides of \eqref{eq:sp_Mackey_proof} equal the identity.

For $g=s$ we have $\Delta(s\cdot y,y)=0$ while $\Delta(s\cdot y,s\cdot y)=\id$ and so both sides of \eqref{eq:sp_Mackey_proof} equal $\Ad_s$. So \eqref{eq:sp_Mackey_proof} holds in case  3A .

\medskip
\noindent{\bf Case 4A:} Take ${x=\left[\begin{smallmatrix} \mu & 1 \\ & \mu \end{smallmatrix}\right]}$, where ${\mu\neq 0}$, and let $y=\diag(x,-x^\transpose)$. The argument is the same as in case  3A.

\medskip
\noindent{\bf Case 1B:} Take $y=0$. We need only consider $g=1$. Then \eqref{eq:sp_Mackey_proof} becomes the assertion that 
\[
\pres^{\ol{G}}_{\ol{L}} \pind_{\ol{L}}^{\ol{G}} \cong \id\oplus \Ad_s \oplus \pind_{\ol{D}}^{\ol{L}}\Ad_w \pres^{\ol{L}}_{\ol{D}}.
\]
This is true: the functors $\pind_{\ol{L}}^{\ol{G}}$ and $\pres^{\ol{G}}_{\ol{L}}$ identify, as in Example \ref{ex:pind-field}, with the functors of Harish-Chandra induction and restriction for the Siegel Levi subgroup in $\ol{G}=\Sp_4(\k)$, and the above formula is just the standard Mackey formula \eqref{eq:Mackey_intro} for the composition of these functors.

\medskip
\noindent{\bf Case 2B:} Let $y=\diag(\mu,-\mu,-\mu,\mu)$, $\mu\neq 0$. We must consider $g=1$, $g=w$ and $g=wt$. 
 
 For $g=1$  the left-hand side of \eqref{eq:sp_Mackey_proof} is equal to 
 \[
 \pres^{\Ad_w(\ol{L})}_{\ol{D}} \pind_{\ol{D}}^{\Ad_w(\ol{L})} = \Ad_w \pres^{\ol{L}}_{\ol{D}} \pind_{\ol{D}}^{\ol{L}} \Ad_{w^{-1}} \cong \Ad_w(\id\oplus \Ad_t) \Ad_{w^{-1}} \cong \id \oplus \Ad_{s}
 \]
 where we have identified $\pind_{\ol{D}}^{\ol{L}}$ and $\pres^{\ol{L}}_{\ol{D}}$ with Harish-Chandra functors and applied the usual Mackey formula \eqref{eq:Mackey_intro} for the group $\ol{L}\cong \GL_2(\k)$ and its diagonal torus $\ol{D}$. On the right-hand side of \eqref{eq:sp_Mackey_proof} we have $\Delta(y,y)=\id$ and $\Delta(y,s\cdot y)=\id$, so we are left to show that $\Xi(y,y)=0$. The only $d\in \germ d$ with $\Delta(d,y)\neq 0$ are $d=y$ and  $d=t\cdot y$. In both of these cases we have $\Delta(y, w\cdot d)=0$, and so $\Xi(y,y)=0$ as required.
 
The $g=w$ and $g=wt$ cases follow the argument for the \lq $g=w$ component\rq~ of case  1A. The left-hand side of \eqref{eq:sp_Mackey_proof} is isomorphic to $\Ad_g \pind_{\ol{D}}^{\Ad_g(\ol{L})}$, while the right-hand side is isomorphic to $\pind_{\ol{D}}^{\ol{L}} \Ad_g$, and the two sides are isomorphic to each other by Theorem \ref{thm:pind-properties}\eqref{item:pind-UV-VU}.

\medskip
\noindent{\bf Case 3B:} Let ${x=\left[\begin{smallmatrix}   & \beta \\ \beta\mu &  \end{smallmatrix}\right]}$, ${\mu\in \k}$ a non-square, ${\beta\in \k^\times}$, and take $y=\diag(x,-x^\transpose)$. We need consider only $g=1$. We have on the one hand $\Delta(y,s\cdot y)=\id$, while on the other hand $\Xi(y,y)=0$ (since $y$ is not $L$-conjugate to any $d\in \germ d$), and so  \eqref{eq:sp_Mackey_proof} reads 
\[
\pres^{\ol{G}(y)}_{\ol{L}(y)} \pind_{\ol{L}(y)}^{\ol{G}(y)} \cong \id \oplus \Ad_s.
\]
This is true: the functors on the right-hand side are isomorphic to Harish-Chandra functors as in Example \ref{ex:pind-field}, and the above formula is the usual Mackey formula for these functors.

\medskip
\noindent{\bf Case 4B:} Take $x=\left[\begin{smallmatrix} 0 & 1 \\ 0 & 0 \end{smallmatrix}\right]$ and $y=\diag(x,-x^\transpose)$. We need only consider $g=1$. As in case 3B, the right-hand side of \eqref{eq:sp_Mackey_proof} is $\id\oplus \Ad_s$, while the left-hand side is $\pres^{\ol{G}(y)}_{\ol{L}(y)} \pind_{\ol{L}(y)}^{\ol{G}(y)}$. Since $\ol{U}(y)$ and $\ol{V}(y)$ commute, the latter functor is isomorphic as in Example \ref{ex:pind-commuting} to the tensor product with the $\H(\ol{L}(y))$-bimodule 
\[
e_{\ol{U}(y)} e_{\ol{V}(y)} \H(\ol{G}(y)) e_{\ol{U}(y)} e_{\ol{V}(y)} \cong \H\left( (\ol{U}(y)\times \ol{V}(y))\backslash \ol{G}(y) / (\ol{U}(y)\times \ol{V}(y)\right) = \H\left( (\ol{U}(y)\times \ol{V}(y))\backslash \ol{G}(y)\right),
\]
with the last equality holding because $\ol{U}(y)\times \ol{V}(y)$ is normal in $\ol{G}(y)$. The semidirect product decomposition of $\ol{G}(y)$ given in Section~\ref{subsec:Sp4_centralisers} for this case implies that 
\[
\H\left( (\ol{U}(y)\times \ol{V}(y))\backslash \ol{G}(y)\right) \cong \H(\ol{L}(y)\rtimes S)\cong \H(\ol{L}(y))\oplus \H(\ol{L}(y))s
\]
as $\H(\ol{L}(y))$-bimodules, and so the corresponding tensor product functor $\pres^{\ol{G}(y)}_{\ol{L}(y)}\pind_{\ol{L}(y)}^{\ol{G}(y)}$ is isomorphic to $\id\oplus \Ad_s$ as required. 

This completes the proof of \eqref{eq:sp_Mackey_proof} and hence, by Corollary \ref{cor:Sp4_Clifford}, of Theorem \ref{thm:sp_Mackey}. \hfill\qed

\subsection{Comparison with  parahoric induction}\label{parahoric_section}
 
We now come to the second corollary of the analysis of Sections~\ref{subsec:Sp4_congruence}--\ref{subsec:Sp4_centralisers}. In addition to the functor 
\[ \pind_{U,V}:\Rep(L)\to \Rep(G)\]
that we have been considering until now, we also have the functor
\[ \pind_{U_0,V} :\Rep(L)\to \Rep(G),\]
which is an example of \emph{parahoric induction} as defined by Dat in \cite{Dat_parahoric}. More precisely, $\pind_{U_0,V}$ is the restriction of a parahoric induction functor for $\Sp_4(\O)$ to the subcategory of representations inflated from $\Sp_4(\O_2)$, where~$\O$ is the ring of integers in a non-archimedean local field. It follows immediately from the definitions that  we have a  natural inclusion  $ \pind_{U,V} \subseteq \pind_{U_0, V}$. 
We shall show that  this inclusion is proper. This gives a negative answer to  \cite[Question 2.15]{Dat_parahoric} in this case.

\begin{corollary}\label{Dat_corollary}
Let $x\in M_2(\k)$ and consider $y=\diag(x,-x^\transpose)\in \germ l$. The restrictions of the functors 
\[
\pind_{U,V} ,\ \pind_{U_0,V}  :\Rep(L) \to \Rep(G)
\]
to the subcategory $\Rep(L)_{\phi_y}$ are mutually nonisomorphic if $x$ is   nonzero and nilpotent; and these restrictions are  mutually isomorphic if $x$ is zero or non-nilpotent.
\end{corollary}

\begin{proof}
The computations of Section  \ref{subsec:Sp4_Clifford}  show that there are commutative (up to natural isomorphism) diagrams 
\[
\xymatrix@R=30pt@C=50pt{ 
\Rep(L)_{\phi_y} \ar[r]^-{\pind_{U,V}} & \Rep(G)_{\phi_y} \\
\Rep(\ol{L}(y)) \ar[r]^-{\pind_{\ol{U}(y),\ol{V}(y)}} \ar[u]^-{F^L_y}_-{\cong} & \Rep(\ol{G}(y)) \ar[u]_-{F^G_y}^-{\cong} 
} \qquad \text{and}\qquad
\xymatrix@R=30pt@C=50pt{ 
\Rep(L)_{\phi_y} \ar[r]^-{\pind_{U_0,V}} & \Rep(G)_{\phi_y} \\
\Rep(\ol{L}(y)) \ar[r]^-{\pind_{\ol{V}(y)}} \ar[u]^-{F^L_y}_-{\cong} & \Rep(\ol{G}(y)) \ar[u]_-{F^G_y}^-{\cong} 
}  
\]
If $x$ is semisimple  then  $\ol{G}(y)$ is a finite reductive group, and $\ol{U}(y)$ and $\ol{V}(y)$ are the unipotent radicals of an opposite pair of rational parabolic subgroups  with common Levi subgroup $\ol{L}(y)$. The functor  $\pind_{\ol{V}(y)}$is the Harish-Chandra induction functor  associated to the  parabolic subgroup  $\ol{L}(y)\ol{V}(y)$    of $\ol{G}(y)$, and as in Example \ref{ex:pind-field} the natural inclusion $\pind_{\ol{U}(y),\ol{V}(y)}\subseteq \pind_{\ol{V}(y)}$ is an isomorphism. 

If $x$ is neither semisimple nor nilpotent, as in Case  4A, then the groups $\ol{U}(y)$ and $\ol{V}(y)$ are both trivial, $\ol{G}(y)=\ol{L}(y)$, and the functors $\pind_{\ol{U}(y),\ol{V}(y)}$ and $\pind_{\ol{V}(y)}$ are both isomorphic to the identity.

We are left to consider the case where $x$ is nilpotent; say $x=\left[\begin{smallmatrix} 0 & 1\\ 0 & 0 \end{smallmatrix}\right]$. In this case we have 
\[
\ol{G}(y)= \left( \ol{U}(y)\times \ol{V}(y)\right) \rtimes \left(\ol{L}(y)\rtimes S\right),
\]
from which it follows (cf. Example \ref{ex:pind-commuting}) that the functors $\pind_{\ol{U}(y),\ol{V}(y)}$ and $\pind_{\ol{V}(y)}$ are isomorphic, respectively, to the compositions
\[
\begin{aligned}
& \pind_{\ol{U}(y),\ol{V}(y)}:\Rep(\ol{L}(y)) \xrightarrow{\infl} \Rep\left( \left(\ol{U}(y)\times \ol{V}(y) \right)\rtimes \ol{L}(y)\right) \xrightarrow{\ind} \Rep(\ol{G}(y)),  \\
& \pind_{\ol{V}(y)}: \Rep(\ol{L}(y)) \xrightarrow{\infl} \Rep(\ol{V}(y)\rtimes \ol{L}(y)) \xrightarrow{\ind} \Rep(\ol{G}(y)).
\end{aligned}
\]
The functor $\pind_{\ol{U}(y),\ol{V}(y)}$ thus scales the $\C$-dimension of representations by a factor of 
\[
\left[\ol{G}(y) : (\ol{U}(y)\times \ol{V}(y))\rtimes \ol{L}(y) \right]  = |S| = 2,
\]
 while $\pind_{\ol{V}(y)}$ scales the dimension by 
 \[
 \left[\ol{G}(y) :\ol{V}(y)\rtimes \ol{L}(y) \right] = |S|\cdot |\ol{U}(y)| = 2  |\k|.
 \]
  Thus $\pind_{U,V}$ is not isomorphic to $\pind_{U_0,V}$ as functors on $\Rep(L)_{\phi_y}$.
\end{proof}

The above proof also shows that the parahoric induction and restriction functors do not satisfy the analogue of Theorem \ref{thm:sp_Mackey}:

\begin{corollary}
Let $x\in M_2(\k)$ be nonzero and nilpotent, and let $y=\diag(x,-x^\transpose)$. The restriction of the functor 
\[
\pres_{U_0,V}\pind_{U_0,V} : \Rep(L) \to \Rep(L)
\]
to the subcategory $\Rep(L)_{\phi_y}$ is not isomorphic to $\id\oplus \Ad_s$.
\end{corollary}

\begin{proof}
The proof of Corollary \ref{Dat_corollary} showed that for each nonzero $M\in \Rep(L)_{\phi_y}$ there is a proper inclusion $\pind_{U,V}(M)\subsetneq \pind_{U_0,V}(M)$, and hence a proper inclusion
\[
\Hom_L(M,M\oplus \Ad_s(M)) \cong \End_G(\pind_{U,V}(M)) \subsetneq \End_G(\pind_{U_0,V}(M)) \cong \Hom_L(M, \pres_{U_0,V}\pind_{U_0,V}(M)).
\]
Thus $\pres_{U_0,V}\pind_{U_0,V}(M)$ is not isomorphic to $M\oplus \Ad_s(M)$.
\end{proof}

\begin{remarks}
\begin{enumerate}[(1)]
\item A straightforward computation with the functors $\pind_{\ol{V}(y)}$ and $\pres_{\ol{V}(y)}$, using the semidirect product decomposition of $\ol{G}(y)$, shows that for each irreducible $M\in \Rep(L)_{\phi_y}$ one has 
\[
\dim_{\C} \End_G(\pind_{U_0,V}(M)) = \begin{cases} |\k|+1 & \text{if }M\cong \Ad_s(M),\\
|\k| & \text{if }M\not\cong \Ad_s(M).\end{cases}
\]
\item The nilpotent orbit $L\cdot y$ is the only one on which the Mackey formula fails to hold for the  functors  $\pind_{U_0,V}$ and $\pres_{U_0,V}$: on all of the other orbits  our proof of Theorem \ref{thm:sp_Mackey} carries over  to the parahoric functors, thanks to Corollary \ref{Dat_corollary}.
\end{enumerate}
\end{remarks}

%%%%%%%%%%%%%%%%%%%
%                                                       %
%    Section 6 -- Iwahori                    %
%                                                       %
%%%%%%%%%%%%%%%%%%%

\section{Representations of the Iwahori subgroup of the general linear group}\label{sec:Iwahori}

Let $\O$ be a {compact} discrete valuation ring with maximal ideal $\germ p$. In this section we shall present a simple application of the functors $\pind_{U,V}$ and $\pres_{U,V}$ to the representation theory of the \emph{Iwahori subgroups}
\[
I_n = I_n(\O) = \{ g\in \GL_n(\O)\ |\ g\textrm{ is upper-triangular modulo }\germ p\} 
\]
We shall relate the representations of $I_n$ to representations of its block-diagonal subgroups. Before stating the main result let us establish some notation (borrowed from \cite{Bernstein-Zelevinsky}) for these subgroups.

Let $\mathcal P_n$  denote the set of compositions (also called ordered partitions) of $n$: an element $\alpha\in\mathcal P_n$ is thus an  ordered tuple of positive integers $(\alpha_1,\alpha_2,\ldots,\alpha_m)$ having $\sum \alpha_i = n$. The \emph{blocks} of $\alpha$ are  the subsets 
\[
b_1(\alpha) = \{1,\ldots,\alpha_1\},\quad  b_2(\alpha)=\{\alpha_1+1,\ldots,\alpha_1+\alpha_2\},\quad \textrm{etc.}
\]
of $\{1,\ldots,n\}$. We shall usually write $n$, instead of $(n)$, for the composition with one block. 

The set $\mathcal P_n$ is partially ordered by refinement: $\alpha\leq \beta$ if each block of $\beta$ is a union of blocks of $\alpha$. This partial order makes $\mathcal P_n$ into a lattice, the greatest lower bound $\alpha\wedge\beta$ of two compositions being the composition whose blocks are the nonempty intersections $b_i(\alpha)\cap b_j(\beta)$ of the blocks of $\alpha$ and $\beta$. We also have an associative order-preserving product 
\[
\mathcal P_n \times \mathcal P_m \to \mathcal P_{n+m}, \qquad (\alpha,\beta)\mapsto \alpha\cdot \beta
\]
given by concatenation.

Given a composition $\alpha\in \mathcal P_n$ we denote by 
\[
I_\alpha   = \{g\in I_n\ |\ g_{ij}=0 \textrm{ unless $i$ and $j$ lie in the same block of $\alpha$}\}
\]
the closed subgroup of $\alpha$-block-diagonal matrices in $I_n$. These groups are compatible with the concatenation product:
\begin{equation}\label{eq:Iw_concat}
I_{\alpha\cdot \beta} \cong I_{\alpha}\times I_{\beta}
\end{equation}
in an obvious way, and this gives an equivalence on smooth representations,
\[
\Rep(I_\alpha)\times \Rep(I_\beta) \xrightarrow[\cong]{(M_\alpha,M_\beta)\mapsto M_\alpha\otimes M_\beta} \Rep(I_{\alpha\cdot\beta}).
\]

We also consider the groups
\[
\begin{split}
U_\alpha &= \left\{g\in I_n\ \left|\ \ \begin{aligned}&g\textrm{ is upper-triangular};\ g_{ii}=1 \textrm{ for every $i$}; \textrm{ and}\\ &g_{ij}=0\textrm{ if $i\neq j$ and $i$ and $j$ lie in the same block of $\alpha$}  \end{aligned}\ \right. \right\}, ~\text{and}\ \\ V_\alpha &= U_\alpha^\transpose \cap I_n. 
\end{split}
\]

If $\beta$ is a second composition with $\alpha\leq \beta$, we define 
\[
U_\alpha^\beta = U_\alpha \cap I_\beta \quad \text{and}\quad V_\alpha^\beta= V_\alpha \cap I_\beta. 
\]
If $\alpha\leq \beta\in \mathcal P_n$ and $\gamma\leq \delta\in \mathcal P_m$, then 
the isomorphism $I_{\beta\cdot\delta}\cong I_\beta\times I_\delta$ of \eqref{eq:Iw_concat} restricts to isomorphisms 
\begin{equation}\label{eq:Iw_concat2}
U_{\alpha\cdot\gamma}^{\beta\cdot\delta} \cong U_\alpha^\beta\times U_\gamma^\delta \qquad \text{and}\qquad V_{\alpha\cdot\gamma}^{\beta\cdot\delta} \cong V_\alpha^\beta\times V_\gamma^\delta.
\end{equation}

 \begin{example} 
 If $\alpha=(2,1)$ and $\beta=(3)$, then 
 \[
 I_\alpha = \left[\begin{smallmatrix}\O^\times  & \O \\ \germ p & \O^\times & \\ & & \O^\times \end{smallmatrix}\right], \quad 
 U_\alpha^\beta =   \left[\begin{smallmatrix} 1 &  & \O \\ & 1 & \O \\ & & 1 \end{smallmatrix}\right],\quad 
 V_\alpha^\beta = \left[\begin{smallmatrix} 1 &  & \\  & 1 & \\ \germ p & \germ p & 1 \end{smallmatrix}\right]
\]
where the blanks indicate zeros. 
\end{example}

\begin{lemma}\label{lem:Iw-linalg}\quad 
\begin{enumerate}[\rm(1)]
\item For each pair of compositions $\alpha \leq \beta$ in $\mathcal P_n$, the triple $(U_\alpha^\beta, I_\alpha, V_\alpha^\beta)$ is an Iwahori decomposition of $I_\beta$.
\item For each triple of compositions $\alpha \leq \beta \leq \gamma$ one has 
\[
U_{\alpha}^\gamma = U_{\alpha}^\beta \ltimes U_{\beta}^\gamma \qquad \text{and} \qquad  V_{\alpha}^\gamma = V_{\alpha}^\beta \ltimes V_{\beta}^\gamma.
\]
\item For each pair of compositions $\alpha,\beta\in \mathcal P_n$ one has 
\[
U_{\alpha\wedge\beta}^\alpha = U_{\beta} \cap I_\alpha  \qquad \text{and}\qquad V_{\alpha\wedge\beta}^\alpha = V_\beta \cap I_\alpha. 
\]
\end{enumerate}
\end{lemma}

\begin{proof}
Part (1) is well-known, and can be established by elementary linear algebra as in \cite[3.11]{Bernstein-Zelevinsky}. Part (2) follows immediately from the Iwahori decompositions. Part (3) boils down to the (manifestly true) assertion that for integers $i$ and $j$ lying in the same block of $\alpha$, $i$ and $j$ lie in the same block of $\alpha\wedge\beta$ if and only if they lie in the same block of $\beta$. 
\end{proof}

\begin{definition}
For each pair of compositions $\alpha \leq \beta $ in $\mathcal P_n$, consider the functors 
\[
\pind_{\alpha}^\beta = \pind_{U_{\alpha}^\beta, V_{\alpha}^\beta} : \Rep(I_\alpha) \to \Rep(I_\beta) \qquad \text{and}\qquad 
\pres^\beta_\alpha = \pres_{U_{\alpha}^\beta, V_{\alpha}^\beta} : \Rep(I_\beta) \to \Rep(I_\alpha).
\]
\end{definition}

{The functors $\pind_{\alpha}^\beta$ and $\pres^\beta_\alpha$ are examples of parahoric induction as defined in \cite{Dat_parahoric}.} Theorems \ref{thm:pind-properties} and \ref{thm:pind-Iw} give some basic properties of these functors. Let us mention two that will be used below:

\begin{lemma}\label{lem:pind-Iw-props}
\begin{enumerate}[\rm(1)]
\item If $\alpha\leq \beta\leq \gamma$ are compositions of $n$, then
\[
\pind_{\alpha}^\gamma \cong \pind_{\beta}^\gamma \pind_{\alpha}^\beta \qquad \text{and}\qquad \pres^\gamma_\alpha \cong \pres^\beta_\alpha \pres^\gamma_\beta.
\]
\item If $\alpha\leq \beta\in \mathcal P_n$ and $\gamma\leq \delta \in \mathcal P_m$, then 
the diagram 
\[
\xymatrix@C=50pt{
\Rep(I_\alpha)\times \Rep(I_{\gamma}) \ar[r]^-{\pind_\alpha^\beta \times \pind_\gamma^\delta} \ar[d]_-{\otimes} & \Rep(I_\beta)\times \Rep(I_\delta) \ar[d]^-{\otimes} \\
\Rep(I_{\alpha\cdot\gamma}) \ar[r]^-{\pind_{\alpha\cdot\gamma}^{\beta\cdot\delta}} & \Rep(I_{\beta\cdot\delta}) 
}
\]
commutes up to natural isomorphism, as does the corresponding diagram of adjoint functors $\pres$.
\end{enumerate}
\end{lemma}

\begin{proof}
Part (1) follows from part (2) of Lemma \ref{lem:Iw-linalg} and part \eqref{item:pind-stages} of Theorem \ref{thm:pind-properties}. Part (2) follows from the compatibility of the decompositions \eqref{eq:Iw_concat} and \eqref{eq:Iw_concat2}.
\end{proof}

\begin{definition}
An irreducible representation $M$ of $I_n$ will be called \emph{primitive} if $\pres^n_\alpha(M) =0$ for every composition $\alpha\in \mathcal P_n$ except for $\alpha=n$. We denote the set of isomorphism classes of primitive irreducible representations by $\Prim(I_n)$.
\end{definition} 

The following lemma is key to our analysis of the functors $\pind$ and $\pres$.

\begin{lemma}\label{lem:Iw_pres}
Let $\alpha,\beta\in \mathcal P_n$ be compositions of $n$, and let $M$ be an irreducible representation of $I_n$. 
If $\pres^n_\alpha (M)$ and $\pres^n_\beta (M)$ are both nonzero, then so is $\pres^n_{\alpha\wedge\beta}(M) $.
\end{lemma}

\begin{proof}
Since the representation $M$ is irreducible and smooth,  it factors through the quotient map $I_n(\O) \to I_n(\O/\germ p^{\ell})$ for some $\ell$. The functors $\pind$ and $\pres$ commute with inflation (Theorem \ref{thm:pind-properties}\eqref{item:pind-finite}), and so we may replace $\O$ by $\O/\germ p^{\ell}$ and assume throughout the proof that $I_n$ is a finite group. 

We know that $N\coloneq \pres^n_{\alpha} (M)$ is nonzero. 
Therefore, up to isomorphism, we can write
\[
M = \pind^n_{\alpha} (N) = \H (I_n) e_{U_{\alpha}}e_{V_{\alpha}}\ot_{\H(I_{\alpha})} N= \H(I_n)e_{U_{\alpha}} e_{V_{\alpha}} e_{U_{\alpha}} e_{V_{\alpha}}\ot_{\H(I_{\alpha})} N.
\]
(In the last equality we used Proposition \ref{prop:z}.)

We know that the subspace 
\begin{equation}\label{eqn:res.n.beta.M}
\pres^n_{\beta}(M) = e_{U_{\beta}}e_{V_{\beta}}\H (I_n) e_{U_{\alpha}}e_{V_{\alpha}}e_{U_{\alpha}}e_{V_{\alpha}}\ot_{\H(I_{\alpha})} N
\end{equation}
of $M$ is nonzero. By part (2) of Lemma \ref{lem:Iw-linalg} we know that each element of $V_{\alpha}\subseteq V_{\alpha\wedge\beta}$ can be written as the product of an element of $V_{\beta}$ with an element of 
$V^{\beta}_{\alpha\wedge\beta}= V_{\alpha}\cap I_{\beta}$. Therefore, using the Iwahori decomposition of $I_n$ with respect to $\alpha$, we get  that $I_n=V_\alpha I_\alpha U_\alpha=V_\beta V^{\beta}_{\alpha\wedge\beta} I_\alpha U_\alpha$, which allows us to replace $\H(I_n)$ by $\H(V^{\beta}_{\alpha\wedge\beta})$ in \eqref{eqn:res.n.beta.M} and write   
\[
\begin{split}
\pres^n_{\beta}(M)&=e_{U_{\beta}}e_{V_{\beta}}\H (V^{\beta}_{\alpha\wedge\beta}) e_{U_{\alpha}}e_{V_{\alpha}} e_{U_{\alpha}} e_{V_{\alpha}}\ot_{\H(I_{\alpha})} N \\
&= \H (V^{\beta}_{\alpha\wedge\beta}) e_{U_{\beta}} e_{V_{\beta}} e_{U_{\alpha}} e_{V_{\alpha}} e_{U_{\alpha}} e_{V_{\alpha}} \ot_{\H(I_{\alpha})} N,
\end{split}
\]
where the second equality holds because the elements of $\H(V^{\beta}_{\alpha\wedge\beta})\subset \H(I_{\beta})$ commute with $e_{U_{\beta}}$ and $e_{V_{\beta}}$.
So we see that $\pres^n_\beta(M)$ is generated as a representation of $I_\beta$ by its subspace
\begin{equation}\label{eq:Iw_pres_pf_subspace}
e_{U_{\beta}} e_{V_{\beta}} e_{U_{\alpha}} e_{V_{\alpha}} e_{U_{\alpha}} e_{V_{\alpha}} \ot_{\H(I_{\alpha})} N
\end{equation}
and hence that this subspace is nonzero.

We now write $e_{V_{\beta}} = e_{V_{\beta}}e_{V_{\alpha\wedge\beta}^{\alpha}}$. We use the fact that elements of $\H(V_{\alpha\wedge\beta}^{\alpha})\subset \H(I_{\alpha})$ commute with $e_{U_{\alpha}}$ 
and that $e_{V_{\alpha\wedge\beta}^\alpha}e_{V_{\alpha}} = e_{V_{\alpha\wedge\beta}}$ (by part (2) of Lemma \ref{lem:Iw-linalg}), to obtain 
\[
e_{V_\beta} e_{U_\alpha} e_{V_\alpha} = e_{V_\beta}e_{V_{\alpha\wedge\beta}^{\alpha}}e_{U_\alpha} e_{V_\alpha} =
e_{V_\beta} e_{U_\alpha} e_{V_{\alpha\wedge\beta}^{\alpha}} e_{V_\alpha} = 
e_{V_\beta} e_{U_\alpha} e_{V_{\alpha\wedge\beta}}.
\]
A similar argument shows that $e_{U_\beta} e_{V_\beta} e_{U_\alpha} = e_{U_{\alpha\wedge\beta}} e_{V_\beta} e_{U_\alpha}$, and so the subspace \eqref{eq:Iw_pres_pf_subspace} is equal to 
\[
e_{U_{\alpha\wedge\beta}} e_{V_\beta} e_{U_\alpha}  e_{V_{\alpha\wedge\beta}} e_{U_{\alpha}} e_{V_{\alpha}} \ot_{\H(I_{\alpha})} N.
\]
This non-zero subspace of $M$ is contained in the subspace 
\[
\begin{aligned}
e_{U_{\alpha\wedge\beta}} \H(I_n) e_{V_{\alpha\wedge\beta}} e_{U_{\alpha}} e_{V_{\alpha}} \ot_{\H(I_{\alpha})} N & = 
e_{U_{\alpha\wedge\beta}} \H(I_{\alpha\wedge\beta})  e_{V_{\alpha\wedge\beta}} e_{U_{\alpha}} e_{V_{\alpha}} \ot_{\H(I_{\alpha})} N \\
& = e_{U_{\alpha\wedge\beta}} e_{V_{\alpha\wedge\beta}} e_{U_{\alpha}} e_{V_{\alpha}} \ot_{\H(I_{\alpha})} N,
\end{aligned}
\]
where  we have used the Iwahori decomposition of $I_n$ with respect to $\alpha\wedge\beta$, and the inclusion $I_{\alpha\wedge\beta}\subseteq I_\alpha$. But this last nonzero subspace of $M$ is exactly $\pres^n_{\alpha\wedge\beta}(M)$, so we are done. 
\end{proof}

Let us now present the main results of this section:

\begin{theorem}\label{thm:Iw_primitive}
Let $ M $ be an irreducible representation of the Iwahori subgroup $I_n\subset \GL_n(\O)$. There is a unique {composition} $\alpha=(\alpha_1,\ldots,\alpha_m)$ of $n$, and unique primitive irreducible representations $ M _i\in \Prim(I_{\alpha_i})$, such that 
\[
 M  \cong \pind_\alpha^n ( M _1\otimes\cdots\otimes  M _m).
\]
\end{theorem}

\begin{proof}
First note the following consequence of part (2) of Lemma \ref{lem:pind-Iw-props}: if $ M _1,\ldots, M _m$ are irreducible representations of $I_{\alpha_1},\ldots,I_{\alpha_m}$, then 
\begin{equation}\label{eq:Iw_prim} 
 M _i\textrm{ is primitive for all } i \quad \Longleftrightarrow \quad \pres^\alpha_\gamma( M _1\otimes\cdots\otimes M _m)=0\textrm{ for all } \gamma \lneq \alpha.
\end{equation}

Consider the set 
\[
\mathcal Q = \{\alpha\in \mathcal P_n\ |\ \pres^n_\alpha (M) \neq 0\},
\]
which is nonempty since it contains the {composition} $n$. Let $\alpha=(\alpha_1,\ldots,\alpha_m)$ be the greatest lower bound of $\mathcal Q$ in the lattice $\mathcal P_n$; Lemma \ref{lem:Iw_pres} implies that $\alpha\in \mathcal Q$. The (nonzero) irreducible representation $\pres^n_\alpha (M)$ of the group $I_\alpha$ decomposes uniquely as a tensor product 
\[
\pres^n_\alpha(M)  \cong \bigotimes_{i=1}^m  M _i
\]
 of irreducible representations of the factors $I_{\alpha_i}$ of $I_\alpha$ (cf. \eqref{eq:Iw_concat}). If $\gamma\lneq\alpha$ then 
\[
\pres^\alpha_\gamma\left( \bigotimes  M _i\right) \cong \pres^n_\gamma (M)  =0 
\]
by Lemma \ref{lem:pind-Iw-props} part (1) and the minimality of $\alpha$, and so all of the $ M _i$'s are primitive by \eqref{eq:Iw_prim}. Since by part \eqref{item:pind-Iw-pindpres} of Theorem \ref{thm:pind-Iw} we have ${M  \cong \pind_\alpha^n\pres^n_\alpha (M)}$, we are done with the \lq existence\rq~part of the proof.

The uniqueness follows from \eqref{eq:Iw_prim}: if $\pres^n_\beta (M)  \cong N_1\otimes\cdots\otimes N_\ell$, where the $N_i$ are all primitive, then we must have $\beta=\alpha$ by minimality, and then $N_i\cong M_i$ for each $i$ by the uniqueness of the tensor product decomposition.
\end{proof}

Lemma \ref{lem:Iw_pres} also implies the following simple formula for the composition of induction and restriction:

\begin{proposition}\label{prop:Iw_ri}
For all $\alpha,\beta \in \mathcal P_n$ and all $ M \in \Irr(I_\alpha)$  one has 
\[
\pres^n_\beta \pind_{\alpha}^n (M)  \cong \pind_{\alpha\wedge\beta}^\beta \pres^\alpha_{\alpha\wedge\beta} (M) .
\]
\end{proposition}

\begin{proof}
If $\pres^n_\beta(\pind_\alpha^n (M) )$ is nonzero, then---since $\pres^n_\alpha(\pind_\alpha^n M )\cong M $ is also nonzero---Lemma \ref{lem:Iw_pres} implies that 
\[
\pres^\alpha_{\alpha\wedge\beta} (M)  \cong \pres^n_{\alpha\wedge\beta}(\pind_\alpha^n (M) ) \neq 0.
\]
In other words, if $\pres^\alpha_{\alpha\wedge\beta} (M) =0$, then $\pres^n_\beta\pind_\alpha^n (M) =0$ too.

If $\pres^\alpha_{\alpha\wedge\beta} (M)  \neq 0$, then we can use Theorem \ref{thm:pind-Iw} and Lemma \ref{lem:pind-Iw-props}(1) to compute
\[
\pres^n_\beta \pind_\alpha^n  (M)  \cong 
\pres^n_\beta \pind_\alpha^n (\pind_{\alpha\wedge\beta}^\alpha\pres^\alpha_{\alpha\wedge\beta}  (M) ) \cong 
\pres^n_\beta \pind_{\alpha\wedge\beta}^n  \pres^\alpha_{\alpha\wedge\beta} (M)   \cong
\pres^n_\beta \pind_\beta^n (\pind_{\alpha\wedge\beta}^\beta \pres^\alpha_{\alpha\wedge\beta} (M) ) \cong
\pind_{\alpha\wedge\beta}^\beta \pres^\alpha_{\alpha\wedge\beta} (M) 
\]
as claimed. 
\end{proof}
Theorem \ref{thm:Iw_primitive} has the following corollary, which gives a neat description of the way the representations of all the groups $I_n$ (for $n\geq 0$) fit in together.
Namely, let $\mathcal{K}\coloneq \bigoplus_{n\geq 0} K_0\left({\Rep_f}(I_n)\right)$ denote the direct sum of the Grothendieck groups of the categories of {finite-dimensional smooth representations of the groups $I_n$},   with the convention that $I_0$ is the trivial group . The maps induced {on Grothendieck groups} by the functors 
\[
\Rep(I_n)\times \Rep(I_m) \to \Rep(I_{n+m}),\qquad (M_1,M_2)\mapsto \pind_{(n,m)}^{n+m}(M_1\otimes M_2)
\] 
equip $\mathcal{K}$ with a graded multiplication structure.
It follows from Lemma \ref{lem:pind-Iw-props} that this multiplication is associative. Since the irreducible representations of $I_n$ constitute a $\Z$-basis for $K_0({\Rep_f}(I_n))$, Theorem \ref{thm:Iw_primitive} implies the following result:
\begin{corollary}
The ring $\mathcal{K}$ is isomorphic to  
$
\Z\left\langle \bigsqcup_{n\geq 0} \Prim(I_n) \right\rangle
$, the non-commutative polynomial algebra with indeterminates the primitive irreducible representations.\hfill\qed 
\end{corollary}

%%%%%%%%%
% Appendix %%%
%%%%%%%%%

\appendix
\section{Centralisers for the adjoint action of $\Sp_4(\k)$}\label{appendix}

In this section we give proofs of the assertions in Section~\ref{subsec:Sp4_centralisers} regarding the centralisers $\overline{G}(y)$ and $\overline{L}(y)$ and the   spaces   $L\backslash G(y,\germ l)/G(y)$. Part of the computations here can be deduced from \cite{Srinivasan}, where the cardinalities of the centralisers of elements of $\Sp_4(\k)$ are computed, by using the Cayley map. As we require the precise structure of the centralisers we give a detailed computation below.

Fix $x\in M_2(\k)$, and let $y=\diag(x,-x^\transpose)$ be the corresponding element of $\germ l$. Clearly we have 
\[
\ol{L}(y) = \{ \diag(a,a^{-\transpose})\ |\ a\in M_2(\k)(x)\},
\]
where $M_2(\k)(x)$ denotes the centraliser of $x$ in the algebra $M_2(\k)$. Elements of $M_2(\k)$ are either scalar or regular (in the sense of admitting a cyclic vector in $\k^2$). We therefore have
\begin{equation*}\label{eqn:dichotomy}
M_2(\k)(x) =
  \begin{cases}
  M_2(\k)& \text{if $x$ is a scalar matrix}, \\
     \k[x] & \text{if $x$ is non-scalar}.
  \end{cases}
\end{equation*}

Turning to the centralisers in $\overline{G}=\Sp_4(\k)$, let us first note that the matrices $x$ and $-x^\transpose$ give rise to two $\k[T]$-module structures on $\k^2$, and that the centraliser of $y$ in $\GL_4(\k)$ is isomorphic, in an obvious way, to the automorphism group of the direct sum $\k^2_x\oplus \k^2_{-x^\transpose}$ of these modules. 

\begin{lemma}\label{lem:Sp_trace0}
For each $x\in M_2(\k)$ with $\trace(x)=0$, the centraliser $\GL_4(\k)(y)$ of $y=\diag(x,-x^\transpose)\in M_4(\k)$ inside $\GL_4(\k)$ is given by
\[
\GL_4(\k)(y) =  
\Sigma\cdot \GL_2\left( M_2(\k)(x) \right)\cdot \Sigma^{-1} 
\]
where $\sigma = \left[\begin{smallmatrix} & -1 \\ 1 & \end{smallmatrix}\right]\in \GL_2(\k)$ and $\Sigma=\left[\begin{smallmatrix} 1&  \\ & \sigma \end{smallmatrix}\right]\in \GL_4(\k)$. 
\end{lemma}

\begin{proof}
If $\trace(x)=0$ then $\sigma x \sigma^{-1} = -x^\transpose$, and so $\id\oplus\sigma:\k^2_x\oplus \k^2_x\to \k^2_x\oplus \k^2_{-x^\transpose}$ is a $\k[T]$-module isomorphism. Conjugating $\GL_2(M_2(\k)(x))=\Aut(\k^2_x\oplus \k^2_x)$ by this isomorphism gives the asserted description of $\GL_4(\k)(y)$. 
\end{proof}

We now proceed to the computation of  $\ol{L}(y)$, $\ol{G}(y)$ and $L\backslash G(y,\germ l)/G(y)$ in each of the cases listed in Section~\ref{subsec:Sp4_centralisers}. Note that $\trace(x)\neq 0$ in the \lq A\rq~cases, while $\trace(x)=0$ in the \lq B\rq~cases.

\medskip
\noindent{\bf Case 1:} $x=\diag(\mu,\mu)$. 

\smallskip\noindent 
We have $M_2(\k)(x)=M_2(\k)$, so $\ol{L}(y)=\ol{L}$. For $\ol{G}(y)$ and $L\backslash G(y,\germ l)/G(y)$ there are two subcases to consider:

\medskip
\noindent{\bf 1A:} $\mu\neq 0$.  
Since $x$ and $-x^\transpose$ share no eigenvalue, there are no nonzero morphisms between the $k[T]$ modules $\k^2_x$ and $\k^2_{-x^\transpose}$, and consequently we have  $\overline{G}(y)=\overline{L}(y)=\overline{L}$. 

We claim that $L\backslash G(y,\germ l)/G(y)=\{1,s,w\}$. This is equivalent to the claim that there are, up to conjugacy by $L$, three $G$-conjugates of $y$ lying in $\germ l$: namely $y$ itself, $s\cdot y$, and $w\cdot y$.  Indeed, any $G$-conjugate of $y$ in $\germ{l}$ must be split and semisimple, and must therefore be $L$-conjugate to a diagonal matrix $z$ whose entries form a permutation of the entries of $y$. Since $z$ lies in $\germ l$, and hence is of the form $\diag(z_1,z_2,-z_1,-z_2)$,  the only possibilities for $z$ are 
\[
\diag(\mu,\mu,-\mu,-\mu),\ \diag(-\mu,-\mu,\mu,\mu),\ \diag(-\mu,\mu,\mu,-\mu),\text{ or } \diag(\mu,-\mu,-\mu,\mu).
\]
The first three are equal to $y$, $s\cdot y$ and $w\cdot y$ respectively, while the last is $L$-conjugate to $w\cdot y$.

\medskip
\noindent {\bf 1B:} {$\mu=0$.} 
Obviously $\overline{H}(y)=\overline{H}$ for all  $H\subseteq G$, and $G(y,\germ l)=G(y)$.

\medskip
\noindent{\bf Case 2:} $x=\diag(\mu,\nu)$,  $\mu \neq  \nu$.

\smallskip\noindent
We have $M_2(\k)(x) = \left\{\left. \left[\begin{smallmatrix} \alpha & \\ & \beta \end{smallmatrix}\right]\ \right|\ \alpha,\beta\in \k\right\}\cong \k\oplus \k$, and so $\ol{L}(y)=\ol{D}$ is the group of diagonal matrices in $\ol{G}$. 
There are three subcases to consider:

\medskip
\noindent{\bf 2A:} $\nu\neq \pm\mu$, $\mu\neq 0\neq \nu$. 
Similar arguments to those of Case 1A show  that $\overline{G}(y)=\overline{L}(y)$, and that $L\backslash G(y,\germ l)/G(y) = \{1,s,w,wt\}$. 

\medskip
\noindent{\bf 2A$\!^\star$:} $\nu= 0$. 
The space $\Hom_{\k[T]}(\k^2_x,\k^2_{-x^\transpose})$ is one-dimensional, spanned by $p = \left[\begin{smallmatrix} 0 & \\ & 1\end{smallmatrix}\right]$, and so we have
\[
\GL_4(\k)(y) = \left\{\left. \begin{bmatrix} a & b \\ c& d\end{bmatrix}\ \right|\ a,d\in \overline{D},\ b,c\in \k p\right\}.
\]
Applying the condition $j^{-1}g^t j=g^{-1}$ defining $\Sp_4(\k)$ to a matrix of the above form, we find that 
\[
\ol{G}(y) = \left\{\left. \begin{bmatrix} \alpha_1  & & & \\ & \alpha_2 &  &\beta  \\ & & \delta_1 & \\ & \gamma & & \delta_2 \end{bmatrix}\in \GL_4(\k)\ \right|\ \alpha_1\delta_1=1=\alpha_2\delta_2-\beta\gamma \right\}\cong GL_1(\k)  \times \SL_2(\k).
\]
The Weyl group of $\SL_2(\k)$ with respect to its diagonal torus is generated by the matrix $\sigma$, and so the Weyl group of $\ol{G}(y)$ with respect to $\ol{D}$ is generated by the matrix 
\[ 
\begin{bmatrix} 1 & & & \\
& 0 & & -1 \\
& & 1 & \\ & 1 & & 0 \end{bmatrix} = t^{-1} w t.
\]
Up to $L$-conjugacy, the $G$-conjugates of $y$ lying in $\germ l$ are $y$ and $-y=w\cdot y$, and so $L\backslash G(y,\germ l)/G(y)=\{1,w\}$.
 
\medskip
\noindent{\bf 2B:} $\nu=-\mu$.
Let $z=\diag(-\mu,-\mu, \mu, \mu)$, so that $y=w\cdot z$. Then $\ol{G}(y)=\Ad_w(\ol{G}(z))$, and $\ol{G}(z)=\ol{L}$ as in Case 1A. Since $\Ad_{w}^{-1}(\ol{U})\cap \ol{L}=\ol{V'}$, and $\Ad_w^{-1}(\ol{V}) \cap \ol{L} = \ol{U'}$, we have $\ol{U}(y) = \Ad_w\left( \ol{V'}  \right)$ and $\ol{V}(y) = \Ad_w\left( \ol{U'}  \right)$. The argument of Case 1A gives $L\backslash G(y,\germ l)/G(y)=\{1,w,wt\}$.

 \medskip
 \noindent{\bf Case 3:} $x=\big[\begin{smallmatrix} \alpha & \beta \\  \mu\beta &  \alpha\end{smallmatrix}\big], \mu \in \k$ non-square, $\alpha\in \k$, $\beta\in \k^\times$.
 
 \smallskip\noindent
 In this case $M_2(\k)(x)=\left\{\big[\begin{smallmatrix} \alpha_1 & \beta_1 \\ \mu \beta_1 & \alpha_1 \end{smallmatrix}\big] \mid \alpha_1,\beta_1 \in \k\right\}$ is a quadratic field extension of $\k$, which we shall denote by $\k_2$. There are two subcases to consider. 
 
\medskip
\noindent{\bf 3A:} $\alpha \neq 0$.
Similar arguments to those of case  1A  (considering the eigenvalues in $\k_2$) show that   $\overline{G}(y)=\overline{L}(y)$, while $L\backslash G(y,\germ l)/G(y)=\{1,s\}$.
 
\medskip
\noindent{\bf 3B:}  $\alpha=0$.
Since $\trace(x)= 0$, Lemma \ref{lem:Sp_trace0} implies that
$\Ad_{\Sigma^{-1}}: \GL_4(\k)(y) \to \GL_2(\k_2)$
is an isomorphism. Observing that $\Ad_{\Sigma^{-1}}(j)= \left[\begin{smallmatrix} & \sigma^{-1} \\ \sigma^{-1} & \end{smallmatrix}\right]$, and that $\Sigma^\transpose = \Sigma^{-1}$, we find that the isomorphism $\Ad_{\Sigma^{-1}}$ sends $\overline{G}(y)$ to 
\[
\Ad_{\Sigma^{-1}}(\overline{G}(y)) = \{ g\in \GL_2(\k_2)\ |\ g^* g = 1\},
\]
where
\[
\begin{bmatrix} a & b \\ c & d\end{bmatrix}^* = \begin{bmatrix} \sigma d^\transpose \sigma^{-1} & \sigma b^\transpose \sigma^{-1} \\ \sigma c^\transpose \sigma^{-1} & \sigma a^\transpose \sigma^{-1} \end{bmatrix}.
\]
The map $a\mapsto \sigma a^\transpose \sigma^{-1}$ is a nontrivial $\k$-algebra  automorphism of $\k_2$, and so is equal to the nontrivial element $a\mapsto a^{|\k|}$ in $\operatorname{Gal}(\k_2/\k)$. 

Let $\overline{\k}$ denote an algebraic closure of $\k_2$. The above computations show that $\Ad_{\Sigma^{-1}}$ restricts to an isomorphism from $\overline{G}(y)$ to the (unitary) group   $\operatorname{GU}_2(\k)$ of fixed points   of the automorphism
\[
\GL_2(\overline{\k})\to \GL_2(\overline{\k}),\qquad  \begin{bmatrix} a & b \\ c & d  \end{bmatrix} \mapsto \begin{bmatrix} d^{|\k|} & b^{|\k|} \\ c^{|\k|} & a^{|\k|} \end{bmatrix}^{-1}.
\]
 The subgroup $\overline{L}(y)$ corresponds under this isomorphism to the {non-split rational} maximal torus of diagonal matrices {$\{\diag(a,a^{-|\k|}) \mid a \in \k_2^\times \}$}    in $\operatorname{GU}_2(\k)$, while $\overline{U}(y)$ and $\overline{V}(y)$ correspond to the unipotent radicals of  {rational} Borel subgroups of upper / lower triangular matrices. The Weyl group of $\operatorname{GU}_2(\k)$ with respect to its diagonal torus is generated by the matrix $\left[\begin{smallmatrix} & -1 \\ -1 &   \end{smallmatrix}\right]=\Ad_{\Sigma^{-1}}(s)$.
 
The argument of case 1A shows that all of the $G$-conjugates of $y$ lying in $L$ are already $L$-conjugate, and so we have $L\backslash G(y,\germ l)/G(y)=\{1\}$. 

\medskip
\noindent{\bf Case 4:} ${x=\left[\begin{smallmatrix} \mu & 1 \\ & \mu \end{smallmatrix}\right]}$.

\smallskip\noindent
We have $M_2(\k)(x)=  \left\{\left[\begin{smallmatrix} \alpha & \beta \\  & \alpha \end{smallmatrix}\right] \mid \alpha,\beta \in \k\right\} \cong \k[\epsilon]/(\epsilon^2)$. There are two subcases to consider.

\medskip
\noindent{\bf 4A:} $\mu\neq 0$.
Arguing as in case  1A  once again, we find that $\overline{G}(y)=\overline{L}(y)$, while 
$L\backslash G(y,\germ l)/G(y)= \{1,s\}$.

\medskip
\noindent{\bf 4B:} $\mu=0$.
Arguing as in case  3B, we find that the isomorphism 
\[
\Ad_{\Sigma^{-1}}: \GL_4(\k)(y) \to \GL_2(\k[x])
\]
of Lemma \ref{lem:Sp_trace0} restricts to an isomorphism between $\overline{G}(y)$ and the  group $Q\subset \GL_2(\k[x])$ of fixed points of the involution
\[
\GL_2(\k[x])\to \GL_2(\k[x]), \qquad \begin{bmatrix} a & b \\ c & d \end{bmatrix} \mapsto \begin{bmatrix} d^\# & b^\# \\ c^\# & a^\#\end{bmatrix}^{-1}
\]
where $\#$ denotes the $\k$-automorphism $x\mapsto -x$ of $\k[x]$. We have furthermore
\[
\begin{aligned}
 \Ad_{\Sigma^{-1}}(\overline{L}(y))& =\left\{\left. \begin{bmatrix} a & \\  & a^{-\#}\end{bmatrix}\in \GL_2(\k[x]) \ \right|\ a\in \k[x]^\times \right\}\eqcolon H, \\
 \Ad_{\Sigma^{-1}}(\overline{U}(y)) &=\left\{ \left. \begin{bmatrix} 1 & b \\ & 1 \end{bmatrix}\in \GL_2(\k[x]) \ \right| \ b\in x\k[x] \right\}\eqcolon X, \quad \text{and} \\
 \Ad_{\Sigma^{-1}}(\overline{V}(y))& = \left\{ \left. \begin{bmatrix} 1 & \\ c & 1 \end{bmatrix}\in \GL_2(\k[x]) \ \right| \ c\in x\k[x] \right\}\eqcolon Y.
\end{aligned}
\]
Let $S$ denote the two-element subgroup of $\ol{G}(y)$ generated by $s$, and let $R$ denote the subgroup $\Ad_{\Sigma^{-1}}(S)$ of $Q$; thus $R$ is the two-element group generated by $r=\Ad_{\Sigma^{-1}}(s)=\left[\begin{smallmatrix} & -1 \\ -1 & \end{smallmatrix}\right]$. 

The subgroups $X$ and $Y$ commute in $Q$, because $x^2=0$. Since $H$ normalises $X$ and $Y$, this implies that the product $XHY$ is a subgroup of $Q$, equal to $(X\times Y)\rtimes H$. 
Explicitly, 
\[
XHY = \left\{ \left. \begin{bmatrix} a & b \\ c & a^{-\#} \end{bmatrix}\ \right| \ a\in \k[x]^\times,\ b,c\in x\k[x]\right\},
\]
i.e. the group of $q\in Q$ such that $q$ is diagonal modulo $x$.

Now, for each $q\in Q$, the reduction of $q$ modulo $x$ is a fixed point of the involution
\[
\GL_2(\k)\to \GL_2(\k), \qquad \begin{bmatrix} a & b \\ c & d \end{bmatrix} \mapsto \begin{bmatrix} d & b \\ c & a \end{bmatrix}^{-1} 
\]
and so $q$ modulo $x$ is either of the form $\left[\begin{smallmatrix} a & \\ & a^{-1}\end{smallmatrix}\right]$   or $\left[\begin{smallmatrix} & b \\ b^{-1} & \end{smallmatrix}\right]$.
Thus the homomorphism 
\[
Q \to \{\pm 1\}, \qquad q\mapsto \det(q\text{ modulo } x)
\]
has kernel $XHY$, and is split by the homomorphism 
\[
\{\pm 1\}\to Q,\qquad -1\mapsto r.
\]
This gives a decomposition $Q  = \left( (X\times Y)\rtimes H \right) \rtimes R$. Since conjugation by $r$ preserves $H$ and permutes $X$ and $Y$, we may rewrite this decomposition as $Q= \left(X\times Y\right)\rtimes \left( H\rtimes R\right)$. Applying $\Ad_{\Sigma}$ gives 
\[
\ol{G}(y) = \left( \ol{U}(y)\times \ol{V}(y)\right) \rtimes \left(\ol{L}(y)\rtimes S\right).
\]

As in Case 3B we have $G\cdot y\cap \germ l = L\cdot y$, and so $L\backslash G(y,\germ l)/G(y)=\{1\}$.

\bibliographystyle{siam}
\bibliography{CMO}

\end{document}